\documentclass[a4paper,11pt]{article}
\usepackage{graphicx} 
\usepackage{amsmath,amssymb,amsthm,amsfonts,amstext,amsbsy,amscd,color,dsfont}
\usepackage{subcaption}
\graphicspath{{fig_report/}} 
\usepackage{float}
\usepackage{makecell}
\usepackage{tikz}
\usepackage{circuitikz}
\usepackage[utf8]{inputenc}
\usepackage{algorithm}
\usepackage{algorithmic}
\usepackage{mathtools}
\usepackage{stmaryrd}
\usepackage{hyperref}
\newcommand{\1}{{\mathchoice {\rm 1\mskip-4mu l} {\rm 1\mskip-4mu l}{\rm 1\mskip-4.5mu l} {\rm 1\mskip-5mu l}}}

\DeclarePairedDelimiterX\Basics[1](){ #1}
\newcommand{\T}{\mathbb{R}_+}
\newcommand{\blue}[0]{\textcolor{black}}

\newcommand{\dd}{
		\mathop{}\mathopen{}\mathrm{d}
	}
\usepackage{verbatim}
\newtheorem{thm}{Theorem}[section]
\newtheorem{prop}{Proposition}[section]
\newtheorem{assumption}{}

\newtheorem{corollary}{Corollary}[thm]

\newcommand{\footremember}[2]{%
    \footnote{#2}
    \newcounter{#1}
    \setcounter{#1}{\value{footnote}}%
}
 
\newcommand{\E}{\mathbb{E}}
\newcommand{\N}{\mathbb{N}}

\newcommand{\R}{\mathbb{R}}
\newcommand{\kinf}{||k||_{t,\lambda}}
\newcommand{\kinfI}{||k||_{I,\lambda}}

\newcommand{\St}{\mathbb{S}}
\newcommand{\Zt}{Z_{t,b_m}}

\newcommand{\one}{\mathds{1}}
\newtheorem{definition}{Definition}[section]
\newtheorem{remark}{Remark}[section]
\newtheorem{lemma}{Lemma}[section]
\usepackage[top=2.5cm, bottom=2.5cm, left=2.5cm , right=2.5cm]{geometry}

\title{Nonparametric hazard rate estimation with associated kernels and minimax bandwidth choice.}

\author{Luce Breuil \footremember{CMAP}{CMAP, CNRS, École polytechnique, Institut Polytechnique de Paris, Inria, Palaiseau, France; E-mail : luce.breuil@polytechnique.edu} \and  Sarah Kaakaï  \footremember{P13}{LAGA, CNRS, Université Sorbonne Paris Nord, Villetaneuse, France  E-mail : kaakai@math.univ-paris13.fr.}
\footremember{ERC}{\footnotesize The research of Sarah Kaakai was funded by the European Union (ERC, SINGER, 101054787). Views and opinions expressed are however those of the author(s) only and do not necessarily reflect those of the European Union or the European Research Council. Neither the European Union nor the granting authority can be held responsible for them.}}
\date{}

\begin{document}
\maketitle

\begin{abstract}
In this paper, we introduce a general theoretical framework for nonparametric hazard rate estimation using associated kernels, whose shapes depend on the point of estimation. Within this framework, we establish rigorous asymptotic results, including a second-order expansion of the MISE, and a central limit theorem for the proposed estimator. We also prove a new oracle-type inequality for both local and global adaptive bandwidth selection, extending the Goldenshluger–Lepski method to the context of associated kernels.  Our results   propose a systematic way to construct and analyze new associated kernels. Finally, we show that the general framework applies to the Gamma kernel, and we provide several examples of applications on simulated data and experimental data for the study of aging. 
\end{abstract}



\section{Introduction}

In many fields, the ability to assess the rate at which events occur, often called the hazard rate, is of central importance. In survival analysis, hazard rate estimation plays a crucial role in demography and biology, for instance in studying disease occurrence, or the influence of genetics factors, environmental conditions, or medical treatments on the risk of death. Hazard rate estimation also arises in various fields including economics, finance, reliability, or insurance. 
This paper is also motivated by biological applications in aging, and particularly the 2-phases model of aging  introduced in \cite{tricoire_new_2015}. This model is based on the biological evidence that drosophila flies  present a sharp decline of several health indicators prior to their death, a behavior which was since then observed in several organisms \cite{ Cansell2025.06.14.659669,rera_intestinal_2012,ageing_cell}. 
Estimating accurately the rates of transition between states is therefore essential to better understand the underlying  biological mechanisms.

When the shape of the hazard rate function is unknown, nonparametric approaches can be particularly useful for estimating the hazard rate without prior knowledge. Kernel estimators are among the most widely used nonparametric estimators. They were first introduced  for density estimation \cite{Rosenblatt_ker_dens}, and most existing results on kernel estimators deal with density estimation. 
However, the theory on density  can be extended to hazard rate estimation by considering a ratio estimator defined as a kernel density estimator over a survival function estimator. 
First order equivalents of the variance and expectation for this ratio kernel hazard rate estimator, along with a central limit theorem, were first obtained in \cite{hazard_1,Hazard2}. These results have been extended in the presence of censoring  in \cite{Lo1985DensityAH}.  Another approach consists in smoothing the increments of the piecewise constant Nelson-Aalen estimator of the cumulated hazard. The Nelson-Aalen estimator for i.i.d observed times $(\tau_i)_{1\leq i \leq m}$ is given by \eqref{eq:nelson-aalen} (see e.g. \cite{andersen_statistical_1993}): 
\begin{align}
    &\Hat{H}_m(t) = \sum_{\tau_i \leq t} \frac{1}{m-N_{\tau_i^-}}, \text{ with }  N_t = \sum_{i=1}^{m} \one_{\{ \tau_i \leq t\}}.   \label{eq:nelson-aalen}
\end{align}
The smoothed hazard estimator is then  defined by
\begin{equation}
\label{kernelestimator}
   \hat{k}_m(t) =  \sum_{i\geq 1} \frac{1}{m-N_{\tau_i^-}} \kappa_{t,b}(\tau_i), 
\end{equation}
where $\kappa_{t,b}$ is a kernel, converging to a Dirac measure as the bandwidth $b$ goes to $0$. This estimator is particularly relevant as it is easier to implement than the ratio estimator, and is numerically faster compared to the ratio estimator as well as more robust \cite{Patil01011994,Comp_estim_Uzu_92}. Furthermore, its expectation can be exactly computed unlike that of the ratio estimator \cite{Patil01011994,Rosenblatt_rice_76, smooth_wang_14}.

The kernel estimator \eqref{kernelestimator}  was introduced and shown to be unbiased in \cite{Hazard2}, where a first-order asymptotic approximation of its variance is also provided. These pointwise results on the expectation and variance were later extended to the case of censored observations in \cite{Tanner_Wong_HR_TCL}, along with a central limit theorem. In a more general framework based on counting processes, \cite{Ramlau_haz} established the convergence of the mean squared error and asymptotic normality. A higher-order expansion of the bias was obtained in \cite{Yandell}.
Global results concerning the convergence of the mean integrated squared error (MISE), including higher-order approximations, can be found in the context of general counting processes with multiplicative intensity in \cite{andersen_statistical_1993}. Hazard rate kernel estimators have also been studied under various dependence conditions, see, e.g., \cite{IZENMAN1990233,ROUSSAS198981}.  

These results apply for the most common kernels usually considered, which are defined by
\begin{equation}
   \forall (t,y) \in \R^2, \qquad \kappa_{t,b}(y) = \frac{1}{b}\kappa\left(\frac{t-y}{b}\right), \label{eq:sym_ker_est}
\end{equation}
where $\kappa$ is a symmetric function integrating to 1. 
An important issue with  estimators based on symmetric kernels, such as defined by \eqref{eq:sym_ker_est},  is that they fail to estimate correctly functions with compact supports (or supports bounded on one end) at the end point(s), see e.g. \cite{Boundary_kernel} or \cite{Hess1999HazardFE}. This is the case when estimating a hazard rate for which the support is a subset of $\R_+$. If the hazard does not vanish at 0 or near the boundary of the support, it is critical to use an estimator that does not introduce bias. This situation can occur when examining factors causing a high initial mortality, or for instance when taking into account infant mortality. In \cite{Orava2011KnearestNK},  a nearest neighbor bandwidth choice was proposed, combined with standard kernels. For the density estimation problem with bounded support, other approaches include  Lagrange, Laguerre and Bernstein polynomials estimators \cite{Laguerre_density,2020_Helali,Bernstein_vitale}, or boundary modifications \cite{boundary_modif,transform}.

In the early 2000s, Chen introduced new kernel functions $\kappa_{t,b}$ (Beta and Gamma kernels) in order to solve the boundary problem, initially for densities supported on $[0,1]$ and $\R_+$ \cite{beta_kernel,Gamma_kernel}, and with shapes depending on the point  $t$ at which they are evaluated. In particular, the kernels can be asymmetric for $t$ close to the  support boundary. Over the past two decades, several kernels have been introduced and studied independently,  including the reciprocal inverse Gaussian (RIG) kernel \cite{RIG_kernel}, Weibull \cite{Salha_weibull}, Erlang \cite{Erlang_Salha2014}, or see also \cite{seven_asym_ker,Boundary_kernel,Salha02092023} for other examples.  These so-called associated kernels are particularly efficient as they are both easy to implement and, in their multiplicity, provide solutions for different estimation problems depending on the support or the shape of the underlying function. They have been vastly used in various fields such as agronomics, biology, climate, finance, insurance, or medicine, and have become a standard practical method to estimate density and hazard rate without boundary bias.

However, theoretical convergence results for associated kernels have been mostly obtained for the density estimation problem, and separately for specific kernels. For instance, Chen in \cite{Gamma_kernel}, Scaillet in \cite{RIG_kernel} and Bertin and Klutchnikoff in \cite{minimax_Beta} provide first order asymptotic equivalents for the MISE. The ratio-type hazard rate estimator for the Weibull, Erlang and Lognormal kernels have been studied in \cite{Salha_weibull,Erlang_Salha2014,Salha02092023}. These works demonstrate asymptotic normality for each specific kernel, but no asymptotic equivalent for the bias or variance of the ratio estimator has been obtained. In contrast, very limited results exist for the hazard estimator \eqref{kernelestimator} with  associated kernels. To our knowledge, only \cite{Bouez_gamma_haz} obtained results for the Gamma kernel.

Despite the significant interest in associated kernels, there is a lack of a unified theoretical framework and results.  For instance, \cite{seven_asym_ker} investigates seven associated kernels independently for the cumulative distribution function, and \cite{l_block_CV} studies three different kernels for density estimation. A more general framework is proposed in \cite{Esstafa2022AsymptoticPO}, in the case of discrete probability distributions.  More recently, the continuous case has been addressed for density estimation in  \cite{esstafa:hal-04112846}, where first-order results on the MISE and asymptotic normality are established.

A key ingredient of kernel estimation is the choice of bandwidth, which can significantly impact the quality of the estimator. Various methods exist, such as  cross-validation \cite{Bandwidth_kernel_crossval,Crossval_Rudemo_82} or local bandwidth selection procedures \cite{Mller1990LocallyAH,Orava2011KnearestNK}. 
In their seminal paper
\cite{GL_11}, Goldenshluger and Lepski introduced and studied an adaptive minimax bandwidth selection method for density estimation. This procedure allows to choose an optimal bandwidth without a priori knowledge of the underlying regularity of the estimated function, thus automatically achieving optimal convergence rate. It also allows for a data driven local bandwidth choice.  This method has been vastly studied in the case of density estimation with classical kernels (see e.g. \cite{Bertin_2016,Doumic_2012,GL_11,LACOUR20163774}), \blue{and for boundary kernels in \cite{Bertin_adaptive}. An adaptive bandwidth choice for density estimation with the Beta kernel is developed in \cite{Bertin_adaptive_beta}}. In \cite{recurrent_event_Bouaziz},  results  were obtained on intensity estimation for recurrent event processes for classical kernels on a compact support, which is more restrictive than the associated kernel framework we propose to study. To our knowledge, no result on an adaptive bandwidth choice with associated kernel exists for hazard rate estimation.\\[2mm]

\indent In this paper, we first provide a unified framework for hazard rate estimation using associated kernels. We introduce general assumptions, under which we prove rigorous results, including an asymptotic expansion for the MISE, and asymptotic normality. 
These results include the few existing results, and extend them to any associated kernel verifying our assumptions. The general setting  also allows us to avoid some of the tedious computations when studying a particular kernel. By giving assumptions that should be verified by the kernel, we  provide a checklist of how to construct such a kernel for hazard rate estimation, and a better overall understanding of the relevance of such kernels and their key properties.

 We then introduce an adaptive bandwidth choice for hazard rate kernel estimators with associated kernels. The lack of assumptions on the exact dependence of the kernel in $t$ and $b$ prevents from using the convolution functional classically considered for adaptive bandwidth choice.
As we consider kernels that do not have bounded supports, the study of this method introduces some theoretical challenges. In particular, results on density kernel estimators cannot be as easily extended to hazard estimation.

We also present numerical results on simulated and real data to compare the performance of this estimator to other kernel estimators. In particular, using the experimental data taken from \cite{tricoire_new_2015}, we show that the death rate is very high for drosophila which have just undergone the transition to a physiologically aged state, and then decreases, a phenomenon which was not captured by kernel estimators using standard kernels.

We first present the theoretical setting and introduce kernel hazard rate estimation and associated kernels in Section \ref{sec:the_frame}. We then state and prove in Section \ref{sec:conv_results} the convergence of the mean integrated square error of the estimator as well as an asymptotic equivalent, by finding equivalents for the bias and variance. We also prove asymptotic normality of the hazard rate associated kernel estimator. Secondly, we prove an oracle type inequality for an adaptive bandwidth selection method in our framework in Section \ref{sec:minimax}, both in a pointwise and global setting.  Finally, we  provide some numerical examples on simulated and experimental data in Section \ref{sec:num_simul}.

\section{Settings}

\label{sec:the_frame}
\subsection{Hazard rate kernel estimation}
Let $(\Omega, \mathcal{F}, \mathbb P)$ be the probability space. 
We consider $m$  i.i.d event time observations $(\tau_i)_{1\leq i \leq m}$. 
The cumulative distribution function (cdf) of the random variables $\tau_i$, defined on its support $\R_+$, is denoted by $F$. We assume that the distribution admits a probability density function (pdf) $f$ and a hazard rate $k$, and recall that:
\begin{align}
    &k(t) = \frac{f(t)}{1-F(t)}, \quad  F(t)= \mathbb{P}(\tau \leq t) = 1 - e^{-\int_0^t k(u) \dd u}, \quad f(t) = k(t)e^{-\int_0^t k(u) \dd u}, \quad \forall t\geq 0. 
\end{align}
For the remainder of the paper, we assume that the hazard rate $k$
\blue{is in a Hölder space (see e.g.  \cite{tsybakov_nonparametric_2009}). In the following, $\lfloor \beta \rfloor$ denotes the floor function of $\beta$. 
\begin{definition}
\label{def:holder}
 Let $\beta > 0$ and $L > 0$. We denote by $\Sigma(\beta, L)$ the Hölder class  of functions on $\R_+$, meaning that for all $k \in \Sigma(\beta, L)$, $k \in C^{\lfloor \beta \rfloor}(\R ^+)$ and for  $l=\lfloor \beta \rfloor $
$$\left|k^{(l)}(t+z) - k^{(l)}(t)\right| \leq L|z|^{\beta - l}, \quad \forall z,t \in \R_+.$$
\end{definition}
}

The event times can be represented by the counting process $N$, defined by: 
\begin{align*}
   N_t = \sum_{i=1}^{m} \one_{\{ \tau_i \leq t\}}, \quad \forall t \geq 0,  
\end{align*}
with  $(\tau_i)_{1\leq i \leq m}$ the sequence of unordered event times. Let $(\mathcal{F}_t)_{t\geq 0}$ be the natural filtration generated by the counting process. In the framework of i.i.d $(\tau_i)_{1\leq i \leq m}$, $N$ admits the $(\mathcal{F}_t)$- multiplicative intensity $ (k(t)(m -N_{t^-})) \one_{N_{t-} < m}$.

The Nelson-Aalen estimator (see e.g. \cite{andersen_statistical_1993}) provides a nonparametric estimator for the cumulative hazard rate $A(t) = \int_0^t k(s)ds $ in the multiplicative intensity setting, given in our framework by 
\begin{equation*}
   \hat{H}_m(t) =  \sum_{\tau_i\leq t} \frac{1}{m - N_{\tau_i^-}}. 
\end{equation*}
A smooth estimator $\hat{k}_m$ of the hazard rate itself can be derived from the previous equation: 

\begin{equation}
   \hat{k}_m(t) =  \sum_{i\geq 1} \frac{1}{m-N_{\tau_i^-}} \kappa_{t,b}(\tau_i) = \int_0^{+\infty} \frac{\kappa_{t,b}(s)}{m-N_{s^-}} \one_{\{N_{s-} < m \}}dN_s,  \label{eq:k_estim}
\end{equation}
with $\kappa_{t,b}$ a kernel function which converges to the delta Dirac  at $t$ when $b$ (the bandwidth) goes to zero. 
Note that the estimator $\hat{k}_m$ can be rewritten as 
\begin{equation} \label{k:other_form}
    \hat{k}_m(t) = \frac{1}{m}\sum_{i=1}^m \frac{\kappa_{t,b}(\tau_i)}{1-\hat{F}_m(\tau_i)}\end{equation} 
with $\hat{F}_m(x) = \frac{1}{m}\sum_{i=1}^m \one_{\{\tau_i < x\}}.$\\

The most common kernels usually considered are defined as in \eqref{eq:sym_ker_est}, 
with $\kappa$ a symmetric positive function integrating to 1,  such as the Gaussian, rectangle, triangle or Epanechnikov kernels (see \cite{tsybakov_nonparametric_2009}). 

Note that the dependence of classical kernels in $t$ only comes down to a translation of the kernel, and the parameter $b$ only affects its standard deviation. The symmetry of classical kernels also ensures that $t$ and $y$ are interchangeable in equation \eqref{eq:sym_ker_est}.  
The explicit dependence of the kernel $\kappa_{t,b}$ in $t$ and $b$ also facilitates the proof of convergence results by using changes of variables in order to rely only on properties of $\kappa$ (see \cite{andersen_statistical_1993,Tanner_Wong_HR_TCL, hazard_1, Hazard2}). This allows to get general results that do not depend on the point of estimation $t$.

\subsection{Continuous associated kernels} 

In this paper, we adopt the more general framework of associated kernels for which the shape of the kernel depends on the point of estimation and which are particularly adapted for resolving boundary bias when estimating on bounded supports. We recall below the general definition of associated kernels, as  introduced in \cite{esstafa:hal-04112846,adjm/1545361441}:

\begin{definition}[Associated kernel]\label{def:cont_ass}
Let $b>0$ be the bandwidth. An associated kernel $(\kappa_{t,b})$ is a parametrized probability density function $\kappa_{t,b}$ defined on its support $\St \subseteq \mathbb{R}_+$ \blue{such that for any $y \in \St, t \mapsto \kappa_{t,b}(y)$ is continuous} and verifying for all  $t \in \St$ 
    \begin{equation}\label{eq:asymp_def}
 \Lambda(t,b) := \E(Z_{t,b}) - t\xrightarrow[b \rightarrow 0]{}0 \quad \text{and} \quad Var(Z_{t,b}) \xrightarrow[b \rightarrow 0]{}0,
\end{equation}
where $Z_{t,b}$ denotes the random variable with pdf $\kappa_{t,b}$. In particular, this ensures that $Z_{t,b} \xrightarrow[b \rightarrow 0]{L^2} t$.

\end{definition}
\textbf{Notation} We denote by $||.||_{\infty}$ the $L_\infty$ norm on $\St$. 

\begin{remark}  
The vast majority of kernels $(\kappa_{t,b})$ can be defined on a support independent of $t$ and $b$. 
This is for example true for the Gamma and Beta kernels, log-normal and Weibull kernels \cite{esstafa:hal-04112846}. However, the results presented here can be straightforwardly extended to a case where the support depends on $t$ and/or $b$ provided 
\begin{equation*}
    \forall t \in \T, t \in \St_{t,b} \text{ and } \forall x \in \R, t \mapsto \one_{\{\St_{t,b}\}}(x) \text{ is continuous in } t.
\end{equation*}
As  the support of the associated kernel $\St$ and the support of the hazard rate $\R_+$ are dissociated, the results we present are valid for $t \in \St$ (for the Beta kernel for example, $\St = [0,1]$). \\\blue{In the case where $\St \subsetneq  \R_+$, we consider for ease of notations  that $\kappa_{t,b}$ can be extended by $0$ to a $\mathcal{C}^0$ function on $\R_+$. }
\end{remark}

The Gamma kernel (without interior bias), introduced in \cite{Gamma_kernel}, is an example of associated kernel. The Gamma kernel $(\kappa_{t,b})$ is the density function of a Gamma distribution of parameters $\rho(t)_b$ and $b$. It is defined by

\begin{definition}[Gamma kernel without interior bias]\label{def:gamma}
   For $t \geq 0$ and $b > 0$, the Gamma kernel $(\kappa_{t,b})$ is defined by   
   \begin{equation}
           \kappa_{t,b}(y) = \frac{y^{\rho(t)_b - 1}e^{-y/b}}{b^{\rho(t)_b}\Gamma(\rho(t)_b)}\one_{\{y\geq 0\}} \label{eq:Gamma_kernel_1}
    \end{equation}
    where 
    \begin{equation}
    \rho(t)_b = 
    \begin{cases}
         &t/b \text{ if } t \geq 2b\\
         &\frac{1}{4}(t/b)^2 + 1  \text{ if } 0 \leq t < 2b.
    \end{cases}
    \label{eq:rho_gamma_1}
    \end{equation}
\end{definition}

The shape of the Gamma kernel for different values of $t$ is shown on Figure \ref{fig:gamma_ker}, where one can see that the Gamma kernel has a support on $\R_+$ and is asymmetric for $t$ close to $0$, unlike the Gaussian kernel (see Figure \ref{fig:gauss_ker}), which is symmetric and is defined on $\R$. 
\begin{figure}[H]
     \centering
        \begin{subfigure}{0.48\linewidth}
    \includegraphics[scale = 0.25]{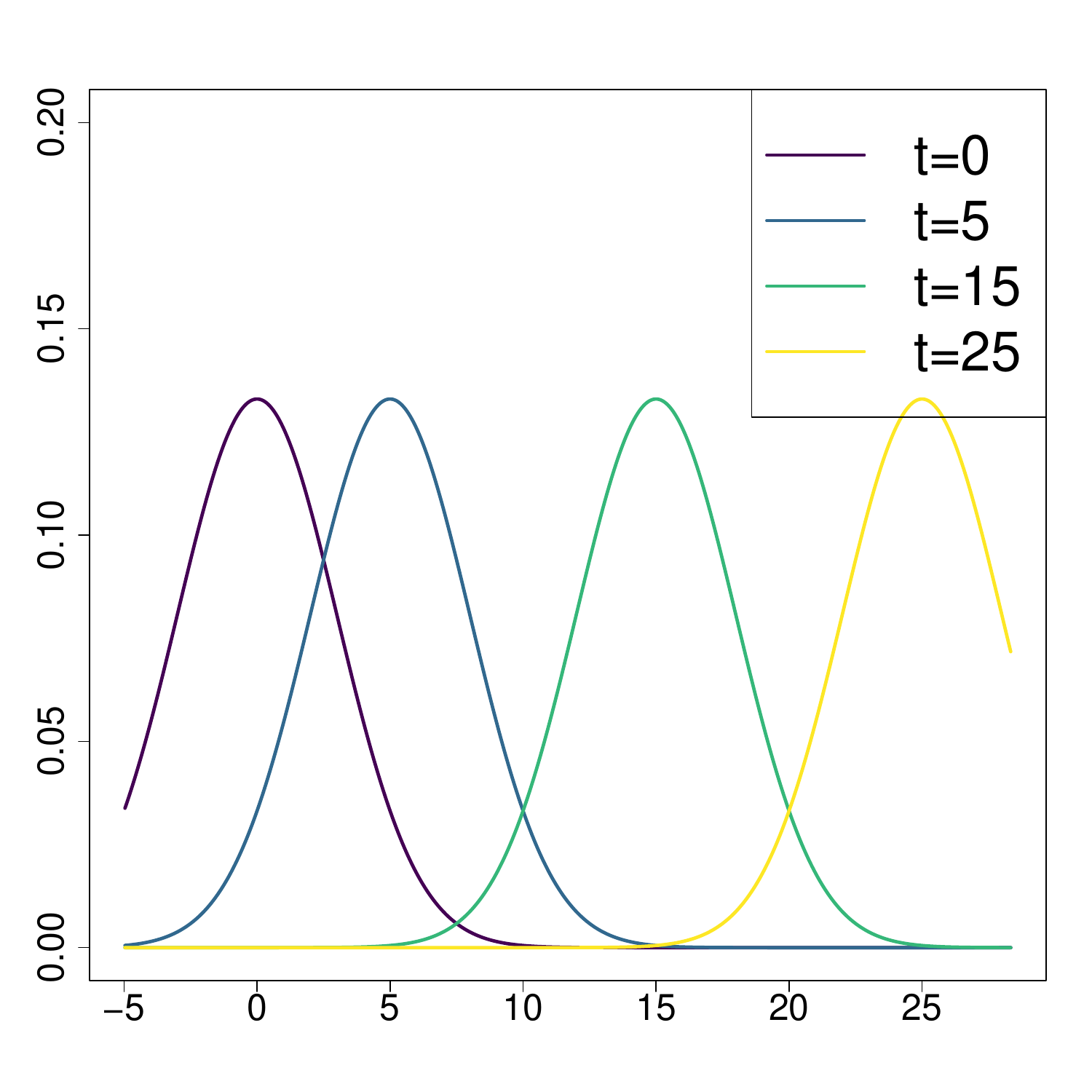}
    \caption{Shape of the Gaussian kernel for different values of $t$ with a fixed bandwidth $b=0.1$.}
    \label{fig:gauss_ker}
    \end{subfigure}
     \begin{subfigure}{0.48\linewidth}
    \includegraphics[scale = 0.25]{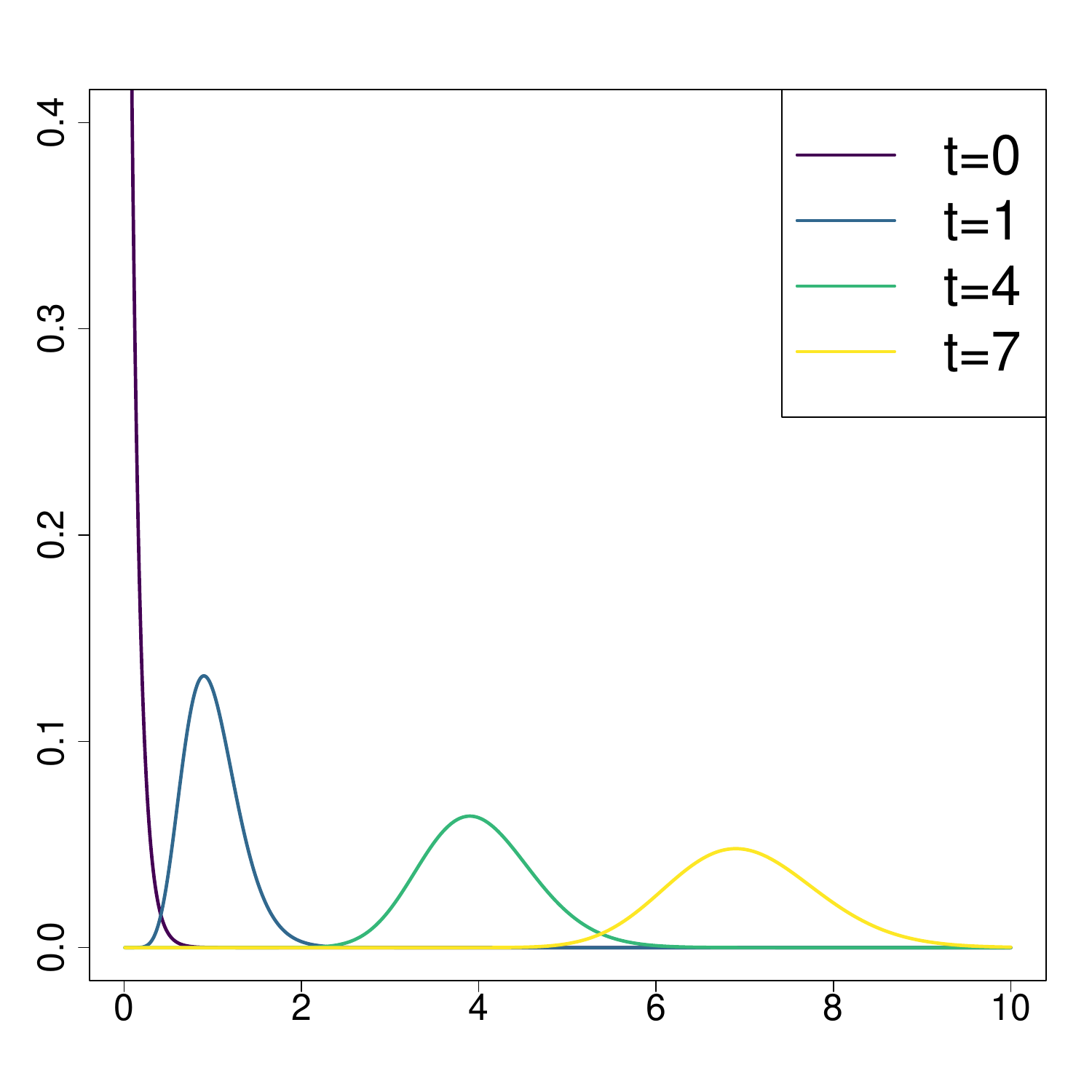}
    \caption{Shape of the Gamma kernel for different values of $t$ with a fixed bandwidth $b=0.1$.}
    \label{fig:gamma_ker}
    \end{subfigure}
\end{figure}

We can now introduce the notion of associated kernel of order $(\beta,\gamma)$, generalizing the standard definition to associated kernels. The parameter $\gamma$ is a scaling parameter that is also involved in several quantities depending on the kernel.

\begin{definition}[Associated kernel of order $(\beta,\gamma)$]
\label{def:order_beta}
Let $(\kappa_{t,b})$ be an associated kernel as defined in Definition \ref{def:cont_ass} and  $\beta > 0$, $\gamma>0$. 
The kernel  is called of order $(\beta,\gamma)$ if  for all integers $n$ such that $1 \leq n \leq \lfloor \beta \rfloor$, 
\begin{equation}\label{assump:gamma_1}
 \exists C_n \in L^{\infty}_{\text{loc}}(\R_+,\R_+^*),  \forall t \;  \in \St, \; \forall \; b \geq 0, \quad  \E[(Z_{t,b} -t)^n] \leq C_n(t) b^{\gamma \beta}.
\end{equation}
\end{definition}
\blue{Definition \ref{def:order_beta} requires that all moments $n\leq \lfloor \beta \rfloor$ of $Z_{t,b}$ have the same order $O(b^{\gamma\beta})$. In particular, when $\beta=2$, this means that   the bias $\Lambda(t,b)$ and variance  $Var(Z_{t,b})$ terms should be of order $O(b^{2\gamma})$. Similar assumptions are made in \cite{esstafa:hal-04112846} for an associated kernel density estimator when $\beta=2$, but the bias term is only assumed to be of order $O(b^{\gamma})$, which does not allow to reach the optimal asymptotic rate of convergence obtained in Theorem \ref{thm:MSE_k} and Theorem \ref{thm:MISE_k} in Section \ref{sec:conv_results}.}

\paragraph{Assumptions} For an associated kernel $(\kappa_{t,b})$ of order $(\beta,\gamma)$, we now introduce the following assumptions, which are used to prove the results of Sections \ref{sec:results} and \ref{subsec:normal}. 
We denote $L^{\infty}_{\text{loc}}(E,H)$ the set of functions from $E$ to $H$ bounded on any compact set.  The first assumption is specific to associated kernels, as it is trivially verified by classical kernels, and is needed to compute with precision the rest term in the equivalent of the Mean Square Error (MSE), but can be omitted for first order results, as in the case of density kernel estimation (\cite{esstafa:hal-04112846}). 

Note that if $(\kappa_{t,b})$ is an associate kernel of order $(\tilde{\beta},\gamma)$ verifying \ref{ass:order} with $\tilde{\beta}\geq \beta$, $(\kappa_{t,b})$ is also of order $(\beta, \gamma)$ and verifies \ref{ass:order}  for these parameters. 
\begin{assumption} \label{ass:order}
The kernel $\kappa$ verifies
\begin{equation*}
     \forall \eta >0, \quad b^{-\beta \gamma} \int_{\{|y-\E[Z_{t,b}]|> \eta\}} \kappa_{t,b}(y)|y-\E[Z_{t,b}]|^{\beta} \dd y \xrightarrow[b\rightarrow0]{}0
\end{equation*}
where $Z_{t,b}$ denotes the random variable with pdf $\kappa_{t,b}$.
\end{assumption}

\begin{assumption}\label{assump:gamma_sup}$\exists b_0 > 0, \exists C_s \in L^{\infty}_{\text{loc}}(\R_+,\R_+^*)\forall t \in \T, \forall b \leq b_0$, 
\begin{equation}
  \sup_{y \in \St}(\kappa_{t,b}(y)) \leq  C_s(t) b^{-\gamma}.
    \end{equation}
\end{assumption}

The following assumption is specific to hazard rate estimation, but is necessary even for classical kernels (see \cite{Hazard2}). However, we present here a weaker assumption than the one in \cite{Hazard2} since it requires that the condition is true for only one $\lambda$,
an assumption easier to prove for the Gamma kernel. Note that this assumption depends both on the kernel and the survival function of the hazard we are estimating \blue{and comes down to assuming that the tails of distribution of $\kappa_{t,b}$ are of the same order as $y \mapsto e^{-\int_0^y k(u) \dd u}$. }

\begin{assumption}\label{assump:compat}
   $F$ and $\kappa_{t,b}$ are compatible i.e. there exists $\lambda>0$ such that for any $t \in \R_+$,  $\exists \;  b_0 >0$ and $G(t) \in L^{\infty}_{\text{loc}}(\R_+, \R_+^*)$, such that
    \begin{equation}
    \forall b \leq b_0, \forall y \in \St, |y-t|>\lambda ,\implies \frac{\kappa_{t,b}(y)}{1-F(y)}  < G(t).\label{eq:compatibility_k}
\end{equation}  
\end{assumption}

The next assumption is needed to bound the mean integrated square error (MISE) in Theorem \ref{thm:MISE_k}, but is not necessary for pointwise results. 
\begin{assumption}\label{assump:unif_int}
   For any fixed $b$ and any compact set $I$,
   \begin{align}
       & 
       \sup_{t\in I} (\kappa_{t,b}(y)) :=  \psi_{I,b}(y) \text{  and  }  \sup_{t\in I}(\kappa_{t,b}(y)^2) := \phi_{I,b}(y)
   \end{align}
   are integrable functions. 
\end{assumption}

Finally, we introduce the following notations: 
\begin{equation}
     \alpha_b(t) :=\int_{\St} \kappa_{t,b}(y)^2 \dd y \qquad  \beta_b(t) := \int_{\St} \kappa_{t,b}(y)^3 \dd y. \label{eq:alpha_beta}
\end{equation}

\begin{assumption}\label{assump:gamma_2_inf}$\exists b_0 >0, \exists \underline{C}_3, \underline{C}_4  \in L^{\infty}_{\text{loc}}(\R_+,\R_+^*) \forall t \in \T, \forall b \leq b_0$, 
\begin{align}
  & \underline{C}_3(t)b^{-\gamma}\leq  \alpha_b(t)\\
  & \underline{C}_4(t)b^{-2\gamma}\leq  \beta_b(t)
    \end{align}
\end{assumption}

\begin{remark}\label{rem:alpha_beta_bound}
 Note that Definition \ref{def:cont_ass} and Assumption \ref{assump:gamma_sup} directly imply that for $b \leq b_0$ 
    \begin{align*}
        &\alpha_b(t) \leq C_s(t) b^{-\gamma}\\
        &\beta_b(t) \leq C_s(t)^2 b^{-2\gamma}. 
    \end{align*}
\end{remark}

The following proposition ensures that the Gamma kernel verifies the assumptions, the proof is postponed to the Appendix (Section \ref{sec:proof_2.1}).

\begin{prop}\label{prop:gamma}

    \blue{The Gamma kernel as defined in Definition \ref{def:gamma} is an associated kernel of order $(\beta, \gamma)$ with $\beta =2$, $\gamma = \frac{1}{2}$ and $\mathbb{S} = \R_+$.} Furthermore, for all $t\geq 0$, 
    \begin{equation*}
        \Lambda(t,b) = (t^2/(4b) +b) \one_{\{t\leq 2b\}}, \quad \text{Var}(Z_{t,b}) = b\one_{\{t>2b\}}+ (t^2/4+b^2) \one_{\{t\leq 2b\}},  
    \end{equation*}
  and Assumptions \ref{ass:order} to \ref{assump:gamma_2_inf}  are verified.
\end{prop}

\blue{ The previous assumptions are also verified by the RIG, Lognormal and Weibull kernels for $\beta = 2$ and $\gamma= 1 \text{ or } \frac{1}{2}$. We refer the reader to Proposition \ref{prop:other_kernels} in the Appendix. }\\

We start with a technical result which will be useful in other proofs.
\blue{
\begin{lemma} \label{lemma:conv_pointwise}
    Let $t \in \St$ and $Z_{t,b} $ be a random variable of  pdf $\kappa_{t,b}$.
    Suppose that  $(\kappa_{t,b})$ is an associated kernel of order $(\beta,\gamma)$, then $\exists \; C(t) >0$, such that for any $M >0$, 
    \begin{equation}
        \mathbb{P}(|Z_{t,b} - t| \geq M) = \int_{|y-t| \geq M } \kappa_{t,b}(y) \dd y \leq \frac{C(t)}{M^{\lfloor\beta\rfloor}}b^{\beta \gamma}.
    \end{equation}
\end{lemma}
\begin{proof}
Let $l = \lfloor \beta \rfloor$. We have by Markov's inequality
\begin{align*}
  \mathbb{P}(|Z_{t,b} - t| \geq M) \leq \frac{\E[|Z_{t,b} - t|^l]}{M^l}
   & \leq   \frac{C_l(t) b^{\beta \gamma}}{M^l}.
   \end{align*}
\end{proof}
}

\section{Convergence results}\label{sec:conv_results}
Recall that the hazard rate kernel estimator introduced in equation \eqref{eq:k_estim} is given for all $t\in \St$ by
\begin{equation*}
   \hat{k}_m(t) =  \sum_{i\geq 1} \frac{1}{m-N_{\tau_i^-}} \kappa_{t,b}(\tau_i),
\end{equation*}
with $(\kappa_{t,b})$ an associated kernel and $(N_t)_{t\geq 0}$ the counting process counting the events' occurrences.

\subsection{Convergence of the mean integrated square error}

A measure of the quality of the estimator is given by the mean integrated square error (MISE) (\cite{Bandwidth_kernel_crossval}) defined on a compact set $I \subset \St $ by
\begin{equation}
    \text{MISE}(b) = \E\left[\int_I(\hat{k}_m(t) - k(t))^2 \dd t \right] = \int_I \E[\hat{k}_m(t) - k(t)]^2+Var(\hat{k}_m(t))  \dd t \label{eq:MISE_def}
\end{equation} 
where the right hand side of \eqref{eq:MISE_def} is the decomposition of the error in the bias and variance terms, for which asymptotic equivalents will be shown.\\

We start by showing that the estimator is asymptotically unbiased, and prove  a non-asymptotic inequality on the bias of the estimator. The proofs for the results of this section are gathered in Section \ref{sec:proofs_MISE}. 

\label{sec:results}
In the following, we consider a sequence $(b_m)_{m\geq 1}$ such that 
$$b_m \xrightarrow[m \rightarrow + \infty]{}0.$$

\begin{prop}    
\label{prop:inequ}
Let $\hat{k}_m$ be defined by \eqref{eq:k_estim} with a kernel verifying Definition \ref{def:cont_ass}.

Suppose $k \in \Sigma(\beta, L)$ and assume either that $k$ is bounded, or that Assumption \ref{assump:compat} is verified. We have for $t \in \St$, 
 \begin{equation}
      \E[\hat{k}_m(t)] =\int_{\St} (1-F(y)^m)k(y) \kappa_{t,b_m}(y) \dd y \xrightarrow[m \rightarrow + \infty]{} k(t) \label{eq:cvgce_exp_0_k}. 
 \end{equation}
Hence $\hat{k}_m(t)$ is asymptotically unbiased.

 Furthermore, we have the following inequality, which is not asymptotic as it holds for any $ m \geq 1 $: for any $b>0$, $t \in \St$, 
 \begin{align}
    |\E[\hat{k}_m(t)] - k(t)|  &\leq\sum_{s=1}^{l} \frac{|k^{(s)}(t)|}{s!}\E[ (Z_{t,b}-t)^s] + \frac{L}{l!}\E[|Z_{t,b}-t|^{\beta}]+  F(t+\lambda)^m  \sup_{y\in [t-\lambda, t+\lambda]}(k(y)) + U_{b,m}(t). \label{eq:bound_nas}
\end{align}
Where
\begin{itemize}
     \item $\lambda >0 $ and $ U_{b,m}(t) = ||k||_{\infty} \mathbb{P}(|Z_{t,b} - t| \geq \lambda)$ if $k$ is bounded. 
\item $U_{b,m}(t) = \frac{G(t)}{m+1}$ if Assumption \ref{assump:compat} is verified, and $ G(t) >0$ and $\lambda$ are defined in Assumption \ref{assump:compat}. 
 \end{itemize}
\end{prop}
\color{black}

The following proposition precises the result of Proposition \ref{prop:inequ} and gives an asymptotical equivalent of the bias of the estimator. This result will also be used to prove the equivalent of the MISE further on. 
\blue{
\begin{prop}[Bias]\label{thm:cvgce_exp_k}
Suppose $k\in \Sigma(\beta, L)$ for some $\beta ,L  >0$ and $k$ is bounded. 
Let $\hat{k}_m$ be defined by \eqref{eq:k_estim} with an associated kernel $(\kappa_{t,b})$ of order $(\beta,\gamma)$, assumed to verify Assumption \ref{ass:order}. If $\exists \;  \mu >0$ such that $b_m^{\gamma}m^\mu \rightarrow +\infty$, we have the following asymptotic expansion for $t \in \St$
 \begin{equation}
     \E[\hat{k}_m(t)] -k(t) =   \sum_{s= 1}^{l } \frac{k^{(s)}(t)}{s!}|\E[(\Zt - t)^{s}]| + o(b_m^{\gamma \beta}).    \label{eq:bound_exp}
 \end{equation}
In particular, the convergence rate is of $O(b_m^{ \gamma \beta})$.
 \end{prop}
}

\begin{remark}
    The assumption that $k$ is bounded in Proposition \ref{thm:cvgce_exp_k} can be replaced by Assumption \ref{assump:compat}. In that case, the expansion of the bias has a term of order $\frac{1}{m}$ which is not negligible compared to $b_m^{\gamma \beta}$. However, as the variance is of order $\frac{1}{mb_m^{\gamma}}$ (see Proposition \ref{lemma:var_eq_k}), the compatibility assumption does not change results of Theorem \ref{thm:MSE_k} and \ref{thm:MISE_k} for the order of convergence of the MSE and MISE.
\end{remark}
\color{black}
\blue{Contrary to the case of standard symmetric kernels of order $\beta$, the terms in the sum on the r.h.s. of \eqref{eq:bound_exp} do not vanish. However, Definition \ref{def:order_beta} for associated kernels ensures that all the terms in the sum are of the same order.}\\
The assumption that $\mathbb{S} \subset \R_+$ in Definition \ref{def:cont_ass} ensures that the estimator is asymptotically unbiased for any $t \in \St$. 
    Indeed, when $\St \not \subset \R_+$, if  $\exists \mu >0$ such that $\forall b>0, \int_{\St \backslash \T} \kappa_{0,b}(y) \dd y > \mu$, and $k(0) >0$, then
    \begin{align}
         \liminf_{m\rightarrow + \infty} &|\E[\hat{k}_m(0)]-k(0)| \label{eq:bias_symp} \\
         &=   \liminf_{m\rightarrow + \infty} \Big|\int_{\R_+} (k(y)-k(0))\kappa_{0,b_m}(y) \dd y - \int_{\R_+}k(y)F(y)^m\kappa_{0,b_m}(y) \dd y - k(0)\int_{\St \backslash \T} \kappa_{0,b_m}(y) \dd y \Big| \nonumber\\ 
        &=\liminf_{m\rightarrow + \infty} k(0) \int_{\St \backslash \T} \kappa_{0,b_m}(y) \dd y \geq \mu k(0) > 0, \nonumber
    \end{align}
    as the first two terms go to $0$ (we refer the reader to the proof of Proposition \ref{thm:cvgce_exp_k} for details of the computations). Hence the asymptotic bias is strictly positive. 
    This explains why symmetric kernels defined on $\R$  can be unfit to estimate data defined on $\R_+$, a fact that is well-known. 
    As an illustration, when considering a strictly positive constant hazard rate $k$ defined on $\R_+$ and estimated with a symmetric kernel estimator, we have 
    \begin{align*}
        \E[\hat{k}_m(0)]& = k \int_{\R_+} \kappa_{0,b_m}(y) \dd y - \int_{\R_+}(1-e^{-ky})^m\kappa_{0,b_m}(y) \dd y \xrightarrow[ \substack{m\rightarrow + \infty}]{} \frac{1}{2}k \neq k. 
    \end{align*}
    Hazard rate estimation with a symmetric kernel can result in a relatively good estimation if the hazard rate vanishes near $0$.
    
Furthermore, for $t>0$ and fixed $b$, there is an extra term $k(t) \int_{\St \backslash \R_+}\kappa_{t,b}(y) \dd y$ in the bias expression (as detailed for $t=0$ in \eqref{eq:bias_symp}). Hence the more the kernel is supported outside of $\R_+$, the higher the additional error compared to kernels supported in $\R_+$. 
     When working with classical symmetric kernels where $\kappa_{t,b}$ is compactly supported on $[t-b,t+b]$, the convergence results consider $b$ small enough such that $[t-b,t+b]$ is included in the domain of the hazard rate  (see e.g. \cite{andersen_statistical_1993, recurrent_event_Bouaziz}).  This is also the argument used in \cite{esstafa:hal-04112846} (in the proof of Proposition 3.1 for example).  Note that this argument does not apply at 0 for symmetric kernels, which is the reason why the results presented in \cite{andersen_statistical_1993, recurrent_event_Bouaziz} do not apply at $0$.

\paragraph{}The next proposition states the convergence of the estimator in probability and gives an asymptotic equivalent of the variance term of the MISE.
 \begin{prop}[Variance and consistency]
\label{lemma:var_eq_k}
Let $\hat{k}_m$ be defined by \eqref{eq:k_estim} with a kernel verifying Definition \ref{def:cont_ass}. Assume that $k\in \Sigma(\beta,L)$. Under Assumptions \ref{assump:gamma_sup}, \ref{assump:compat} and \ref{assump:gamma_2_inf}, if $b_m^{\gamma}m \rightarrow +\infty$ we have for all $t \in \St$,
\begin{align}
    Var(\hat{k}_m(t))  =\frac{\alpha_{b_m}(t)}{m}\frac{k(t)}{1-F(t)}  +
    o\left(\frac{1}{mb_m^\gamma}\right)
\end{align}
with $\alpha_{b_m}(t)$ defined in equation \eqref{eq:alpha_beta}, \blue{recalling that $\alpha_{b_m}(t) = O(b_m^{-\gamma})$ by Remark \ref{rem:alpha_beta_bound}}.\\
In particular, the worst convergence rate is 
$O(\frac{1}{mb_m^\gamma})$ and $\hat{k}_m(t)$ is a consistent estimator of the hazard rate $k(t)$.  
\end{prop}

\blue{By combining the results of Propositions \ref{thm:cvgce_exp_k} and \ref{lemma:var_eq_k}, we directly obtain the following expression of the MSE. }

\blue{\begin{thm}[MSE convergence]
\label{thm:MSE_k}
Suppose $k\in \Sigma(\beta, L)$ for some $\beta> 0$ and $k$ is bounded. 
Let $\hat{k}_m$ be the hazard rate estimator defined by \eqref{eq:k_estim}, with an associated  kernel $(\kappa_{t,b})$ of order $(\beta,\gamma)$. 
Under Assumptions \ref{ass:order} to \ref{assump:compat},  and \ref{assump:gamma_2_inf} and  if  $b_m^{\gamma} m \to + \infty $ we have for all $t\in \St$,
\begin{align}
    MSE&(\hat{k}_m) = \E\left[(\hat{k}_m(t) - k(t))^2  \right] \nonumber \\=& \ \Big(\sum_{s= 1}^{l } \frac{k^{(s)}(t)}{s!}|\E[(\Zt - t)^{s}]|\Big)^2    +  \frac{\alpha_{b_m}(t)}{m}\frac{k(t)}{1-F(t)}   + o(b_m^{2 \gamma \beta}) + o\left(m^{-1}b_m^{-\gamma}\right).\label{eq:mse_k}
\end{align}
The optimal asymptotic convergence rate of $O(m^{-2\beta/(2\beta +1)})$ is achieved for $b_m^{\gamma}= Cm^{-1/(2\beta +1)}$.
\end{thm}
}
Finally, using the equivalents proved in the previous propositions, we obtain the asymptotic expansion of the MISE.

\begin{thm}[MISE convergence]
\label{thm:MISE_k}
Let $I \subset \St$ be a compact set and $\hat{k}_m$ the hazard rate estimator defined by \eqref{eq:k_estim}, \blue{with an associated  kernel $(\kappa_{t,b})$ of order $(\beta,\gamma)$}. \blue{Suppose $k\in \Sigma(\beta, L)$ for some $\beta>0$ and $k$ is bounded.}
Under \blue{Assumptions \ref{assump:gamma_sup} to  \ref{assump:gamma_2_inf} and if \ref{ass:order}} is verified  for all $t$ in the interior of $I$ and  $b_m^{\gamma} m \to + \infty $ we have,
\blue{
\begin{align}
    MISE(&\hat{k}_m) = \E\left[\int_I(\hat{k}_m(t) - k(t))^2 \dd t \right] \nonumber \\=& \int_I \Big(\sum_{s= 1}^{l } \frac{k^{(s)}(t)}{s!}|\E[(\Zt - t)^{s}]|\Big)^2 dt   \!+\! \int_I \frac{\alpha_{b_m}(t)}{m}\frac{k(t)}{1-F(t)} \dd t \! +\! o(b_m^{2 \gamma \beta})\! +\! o\left(m^{-1}b_m^{-\gamma}\right)\label{eq:mise_k}
\end{align}
The optimal asymptotic convergence rate of $O(m^{-2\beta/(2\beta +1)})$ is achieved for $b_m^{\gamma}= Cm^{-1/(2\beta +1)}$. }

\end{thm}
\blue{The asymptotic convergence rate for the MISE in Theorem \ref{thm:MISE_k} is known to be optimal in the minimax sense for estimating hazard rates in a Hölder class, as shown in \cite{Pons01011986,HUBER2003209} (see also \cite{efromovich2016minimax} for Sobolev functional classes). Attaining this rate, however, requires prior knowledge of the regularity $\beta$ of the hazard rate $k$. In Section \ref{sec:minimax}, we introduce a data-driven adaptive bandwidth selection procedure that reaches this optimal rate without prior knowledge of $\beta$.}
The following corollary states the result for the Gamma kernel as defined by Definition \ref{def:gamma}.
\begin{corollary}
\label{coro:MISE}
\blue{Suppose $k$ is in $\Sigma(2,L)$. }
For the Gamma kernel and a sequence $(b_m)_{m\geq 0}$ such that $b_m \xrightarrow[m \to + \infty]{}0$ and $mb_m^{1/2} \xrightarrow[m \to + \infty]{} + \infty$, we have
    \begin{align}
    MISE(\hat{k}_m) &= \E\left[\int_I(\hat{k}_m(t) - k(t))^2 \dd t \right] \nonumber \\&= \frac{1}{2 m \sqrt{\pi b_m}} \int_I t^{-1/2}\frac{k(t)}{1-F(t)} \dd t  + \int_I \frac{1}{4}t^2k''(t)^2 b_m^2 \dd t+ o(b_m^2) + o(m^{-1}b_m^{-1/2}).
\end{align}
In particular, the optimal convergence rate of $O(m^{-4/5})$ is achieved for $b_m = C m^{-2/5}$. 
\end{corollary}

\subsection{Asymptotic normality}
\label{subsec:normal}
We now move on to another result on the asymptotic distribution of the estimator, namely, its asymptotic normality.
Let us fix $t \in \St$. 

\begin{thm}[Asymptotic normality]\label{thm:as_nor}
Let $\hat{k}_m(t)$ be defined by \eqref{eq:k_estim} with a kernel verifying Definition \ref{def:cont_ass}. If $k$ is continuous and bounded and under Assumptions \ref{assump:gamma_sup},\ref{assump:unif_int} and \ref{assump:gamma_2_inf} if $b_m \to 0$ and $b_m^{\gamma}m \rightarrow +\infty$, we have 
\begin{equation}
    \frac{\hat{k}_m(t) - \E[\hat{k}_m(t)]}{\sqrt{Var(\hat{k}_m(t))}} \xrightarrow[m \xrightarrow{}+ \infty]{\mathcal{L}} \mathcal{N}(0,1).
\end{equation}
Where we recall the following expressions
\begin{align}
    &\E[\hat{k}_m(t)] = \int_{\T}(1-F(y)^m) \kappa_{t,b_m}(y) k(y) \dd y\\
    &\text{Var}(\hat{k}_m(t)) \underset{m\xrightarrow{}+\infty}{\sim }\frac{1}{m}\frac{k(t)}{1-F(t)}\int_{\T}\kappa_{t,b_m}(y)^2 \dd y 
\end{align}
\end{thm}

This result is shown in \cite{Tanner_Wong_HR_TCL} for classical symmetric kernels. We will generalize it to associated kernels. Note that the asymptotic equivalent of the variance gives the expected $\sqrt{m}$ order of convergence. Proving asymptotic normality includes controlling the expectations and variances of the quantities involved, which adds some work for associated kernels, as they do not depend explicitly on the bandwidth. However, the assumptions made on the kernels ensure that associated kernels have asymptotically the same behavior as classical symmetric kernels.  
The proof presented here follows closely the one presented in \cite{Tanner_Wong_HR_TCL}, to which we refer the reader for the details of the computations that are not specific to associated kernels.
The proof relies on Hajek's projection method (\cite{Hajek_68}) and can be found in the Appendix in Section \ref{sec:proof_norm}.

Finally, we state the following Corollary for the application to the Gamma kernel. 
\begin{corollary}
\blue{Suppose $k \in \Sigma(2,L)$.} Let $\hat{k}_m$ be the Gamma kernel hazard rate estimator, we have
\begin{equation}
    \frac{\hat{k}_m(t) - \E(\hat{k}_m(t))}{\sqrt{Var(\hat{k}_m(t))}} \xrightarrow[m \xrightarrow{}+ \infty]{} \mathcal{N}(0,1). \nonumber
\end{equation}
For $t>0$,
\begin{align}
     &\E[\hat{k}_m(t)] = k(t) + 
        \frac{1}{2}tk''(t)b_m + o(b_m) \label{eq:eq_exp}\\ 
    &\text{Var}(\hat{k}_m(t)) \underset{m\xrightarrow{}+\infty}{\sim}\frac{1}{m}\frac{k(t)}{1-F(t)} \times 
        \frac{1}{2\sqrt{\pi t b_m}}.\label{eq:eq_var}
\end{align}
And for $t = 0$,
\begin{align*}
&\E[\hat{k}_m(t)] = k(0) + b_mk'(0)  + o(b_m) \qquad  \text{and} \qquad \text{Var}(\hat{k}_m(0)) \underset{m\xrightarrow{}+\infty}{\sim }
    \frac{1}{m}\frac{k(0)}{2b_m}. 
\end{align*}
\end{corollary}

\begin{proof}
    The proof follows from Proposition \ref{prop:gamma} and Theorem \ref{thm:as_nor}. The computations of the equivalents of the quantities involved are detailed in the proof of Proposition \ref{prop:gamma} (see Section \ref{sec:proof_2.1}). 
\end{proof}

\subsection{Proofs of Section \ref{sec:results}}
\label{sec:proofs_MISE}
In this section, we prove the results given in Section \ref{sec:results}, namely Propositions \ref{prop:inequ}, \ref{thm:cvgce_exp_k} and \ref{lemma:var_eq_k} and Theorem \ref{thm:MISE_k}. 

In the case of classical symmetric kernels (\cite{Tanner_Wong_HR_TCL,hazard_1, Hazard2}), asymptotic expansions are obtained by using the explicit formulation of the kernel \eqref{eq:sym_ker_est} and performing changes of variables in order to carry out a Taylor expansion and factorize the bandwidth $b$. This makes it possible to manipulate quantities that do not depend on $b$, thereby obtaining the rate of convergence directly.

In our framework, the kernel is given only through a general formulation, without any explicit dependence on the variables $t$ and $b$. This lack of structure makes the analysis more difficult: this requires the introduction of several new assumptions (stated in Section \ref{sec:the_frame}), and the proof relies on several Taylor expansions, as well as a careful study of the remainder terms, in order to establish the rate of convergence. In particular, no assumption is made on the compactness of the support of the kernel, unlike what is done in a large part of the literature (see e.g. \cite{andersen_statistical_1993, recurrent_event_Bouaziz}). The compatibility  assumption (Assumption \ref{assump:compat}) between survival function and kernel also allows us to 
control the decay of the remainder terms in the integrals. 

Finally, studying the kernel estimator of the hazard rate introduces additional difficulties compared to the density case. Estimating the hazard rate requires treating the data as an ordered sequence, thereby introducing dependence. This in turn complicates the computation of the expectation and variance of the estimator (mainly resolved in \cite{Hazard2} in the classical case). In particular, unlike in the density setting, the convergence of the bias of the estimator depends not only on the bandwidth $b$ but also on the sample size $m$.

\subsubsection{Proof of Proposition \ref{prop:inequ}}

The expression of $\E[\hat{k}_m(t)]$, with $\hat{k}_m$ defined in \eqref{eq:k_estim}, can be directly computed: \blue{recalling $N$ is a counting process of intensity  $(k(t)(m -N_{t^-})) \one_{\{N_{t-} < m\}}$,
\begin{align*}
   \E[\hat{k}_m(t)] & =  \E[ \int_0^{+\infty} \frac{\kappa_{t,b_m}(s)}{m-N_{s^-}} \one_{\{N_{s-} < m \}}dN_s]  = \E[ \int_0^{+\infty}  k(s) \one_{\{N_{s-} < m \}} \kappa_{t,b_m}(s) ds]\\
    & = \int_0^{+\infty}  k(s) \mathbb P ( \max_{i=1,\dots ,  m} \tau_i > s)  \kappa_{t,b_m}(s) ds.
\end{align*}}
Let $Z_{t,b_m}$ be a random variable of pdf $\kappa_{t,b_m}$. Since the r.v  $(\tau_i)_{1\leq i \leq m}$ are i.i.d, we obtain  that
    \begin{align}
    \E[\hat{k}_m(t)] -k(t)   & = \E[k(Z_{t,b_m})(1- F(Z_{t,b_m})^m)] -k(t)
 := A_m - B_m,
    \label{eq:esp_km}
\end{align}
with 
\begin{align*}
A_m = \E[k(Z_{t,b_m})] - k(t), \quad B_m = \E[k(Z_{t,b_m})F(Z_{t,b_m})^m], 
\end{align*}
 and $F$ the cdf of the event times $(\tau_i)_{1\leq i\leq m}$.
\paragraph{Part 1} We begin by showing that $\hat{k}_m(t)$ is asymptotically unbiased. \\\blue{In the case where $k$ is bounded, since by definition, there is convergence in $L^2$ of $(Z_{t,b_m})_{m \in \N}$ towards $t$ (and thus, convergence in law) and $k$ is continuous and bounded}, $A_m$ goes to $0$ as $m$ goes to $+\infty$. 
Let us now prove that $B_m$ vanishes. Let $n \in \N^*$, we have 
\begin{align}
  B_m & \leq   \E[k(Z_{t,b_m})F(Z_{t,b_m})^m\1_{|Z_{t,b_m}-t|\geq n}] + \E[k(Z_{t,b_m})F(Z_{t,b_m})^m\1_{|Z_{t,b_m}-t|\leq n}]  \label{eq:B_equation}\\
  &\leq ||k||_{\infty} \mathbb{P}(|Z_{t,b_m} - t| \geq n) + F(t+n)^m)\nonumber \xrightarrow[m \rightarrow + \infty]{}0 \nonumber.
\end{align}

In the case where the compatibility Assumption \ref{assump:compat} is verified, 
we have for $n \geq \lambda$ with $\lambda$ defined in Assumption \ref{assump:compat}, 
\begin{align*}
    \E[k(Z_{t,b_m})\one_{\{Z_{t,b_m} \geq n\}}] = \int_{\{y\geq n\}} \kappa_{t,b_m}(y)k(y) \dd y \leq G(t)\int_{\{y\geq n\}} f(y) \dd y \xrightarrow[n \to +\infty]{}0. 
\end{align*}
Thus,  the sequence  $(k(Z_{t,b_m}))_{m \in \N}$ is uniformly integrable. As $(k(Z_{t,b_m}))_{m \in \N}$ converges in probability to $k(t)$, it follows that $A_m = \E[k(Z_{t,b_m})] - k(t) \xrightarrow[m \to + \infty]{}0$. And for $B_m$, we have
\begin{align}
   B_m =  \E[k(Z_{t,b_m})F^m(Z_{t,b_m})] &\nonumber \leq F(t+\lambda)^m \int_{\St\cap\{|y-t|\leq \lambda\} } k(y) \kappa_{t,b_m}(y) \dd y \nonumber\\
   & \qquad +G(t) \int_{\St\cap\{|y-t|>\lambda \}} k(y)F(y)^m (1-F(y)) \dd y \nonumber\\
    & \leq  F(t+\lambda)^m \sup_{y\in [t-\lambda, t+\lambda]}(k(y)) +G(t) \frac{1}{m+1}[F(y)^{m+1}]_{y=0}^{y=+\infty}\nonumber\\
    &= F(t+\lambda)^m \sup_{y\in [t-\lambda, t+\lambda]}(k(y)) + \frac{G(t) }{m+1} \xrightarrow[m \to + \infty]{}0. \label{eq:B_compat}
\end{align}
\color{black}
This proves that $\hat{k}_m(t)$ is asymptotically unbiased. 
\blue{
\paragraph{Part 2}  We now prove \eqref{eq:bound_nas} with a Taylor expansion. As \eqref{eq:bound_nas} is a non-asymptotic inequality, we fix $b>0$. Let $Z_{t,b}$ be the random variable with pdf $\kappa_{t,b}$. First,
\begin{align*}
   \forall y \in \T, \, \exists \; \zeta_{b}(y) \in [y, t],  k(y) = \sum_{s=0}^{l-1} \frac{k^{(s)}(t)}{s!} (y-t)^s + \frac{k^{(l)}(\zeta_b(y))}{l!}(y-t)^{l}
\end{align*}
Since this is true for any $y \in \R_+$, in particular it is true for $Z_{t,b}$. Thus by taking the expectation 
\begin{align*}
     \E[k(Z_{t,b})]-k(t) = \sum_{s=1}^{l-1} \E[\frac{k^{(s)}(t)}{s!} (\Zt-t)^s] + \E[\frac{k^{(l)}(\zeta_b(Z_{t,b}))-k^{(l)}(t)}{l!}(\Zt-t)^{l}] + \E[\frac{k^{(l)}(t)}{l!}(Z_{t,b}-t)^{l}]. 
\end{align*}
And, using the fact that $k \in \Sigma(\beta, L)$, we have 
\begin{align}
     |\E[k(Z_{t,b})]-k(t)| &\leq \sum_{s=1}^{l-1} \frac{|k^{(s)}(t)|}{s!}\E[|Z_{t,b}-t|^s] + \E[|\frac{k^{(l)}(\zeta_b(Z_{t,b}))-k^{(l)}(t)|}{l!}|Z_{t,b}-t|^{l}] + \frac{|k^{(l)}(t)|}{l!}\E[|Z_{t,b}-t|^{l}] \nonumber\\
     &\leq \sum_{s=1}^{l} \frac{|k^{(s)}(t)|}{s!}\E[ (Z_{t,b}-t)^s] + \frac{L}{l!}\E[|Z_{t,b}-t|^{\beta}]\label{eq:control_A}.
\end{align}
}

\blue{Under the assumption that $k$ is bounded, the upper bound on $B_m$ is given by \eqref{eq:B_equation}. If $k$ is not bounded but Assumption \ref{assump:compat}, the upper bound follows from equation \eqref{eq:B_compat}.}

\subsubsection{Proof of Proposition \ref{thm:cvgce_exp_k}}

We now move on to the proof of Proposition \ref{thm:cvgce_exp_k}, which refines the result of Proposition \ref{prop:inequ} to prove the asymptotic equivalent of the bias.
Let $Z_{t,b_m}$ be a random variable of pdf $\kappa_{t,b_m}$, and recall the expression of $\E[\hat{k}_m(t)]- k(t) =A_m -B_m$ stated in \eqref{eq:esp_km}, with 
$$A_m = \E[k(Z_{t,b_m})] - k(t), \quad \text{ and } \quad B_m= \E[k(Z_{t,b_m})F(Z_{t,b_m})^m].$$

The proof of the equivalent of $A_m$ follows similar steps to those in \cite{Bouez_gamma_haz} and \cite{Gamma_kernel} for the density kernel estimator in the specific case of the Gamma kernel.
 We introduce $U_{t,b_m}:=\frac{\Zt - \E[\Zt]}{b_m^{\gamma}}$ and the set $\mathbb{K}_{t,m} = \{u \in \mathbb{R}, ub_m^{\gamma}+\E[\Zt] \in \St\}$ and its density $f_{t,b_m}$ such that
 $$ \forall u \in \mathbb{K}_{t,m} \,,  f_{t,b_m}(u) = b_m^{\gamma} \kappa_{t,b_m}(ub_m^{\gamma}+\E[\Zt]).$$ 
\blue{
Let $l = \lfloor \beta \rfloor$, by Taylor expansion around $Z_{t,b_m}$, we have
\begin{align*}
    \E[ k(\Zt)]   & =  k(t) +\sum_{s= 1}^{l} \frac{k^{(s)}(t)}{s!}\E[(\Zt - t)^{s}] + \int_{\St} \kappa_{t,b_m}(y) \frac{k^{(l)}(v_m(y))-k^{(l)}(t)}{l!} (y - t)^{l} \dd y, 
\end{align*}
where $v_m$ is such that $\forall y \in \T, v_m(y) \in [y, t]$. 
Let $\varepsilon >0$, since by assumption $k^{(l)}$ is continuous, we can introduce $\eta >0$ be such that 
$$\forall y \in \R_+, |y-t| \leq \eta \implies |k^{(l)}(y) - k^{(l)}(t)| \leq \varepsilon.$$
Furthermore, since by assumption  $|k^{(l)}(t) - k^{(l)}(y)|\leq L|y-t|^{\beta-l}$ and $v_m(y) \in [y,t]$, we have\\
\begin{align*}
    R_m := |\int_{\St} & \kappa_{t,b_m}(y) \frac{k^{(l)}(v_m(y))-k^{(l)}(t)}{l!} (y - t)^{l} \dd y| \\
    & \leq  \frac{\varepsilon}{l!} \int_{\St \cap \{|y-t| \leq \eta\}} |y - t|^{l} \kappa_{t,b_m}(y) \dd y+ \frac{L}{l!} \int_{\St \cap\{ |y-t| > \eta\}} \kappa_{t,b_m}(y) |y - t|^{\beta} \dd y\\
    & \leq \frac{\varepsilon}{l!} C_l(t) b^{\gamma \beta} + \frac{L}{l!} \int_{\St \cap \{|y-t| > \eta\}} \kappa_{t,b_m}(y) |y - t|^{\beta} \dd y
\end{align*}
By Assumption \ref{ass:order}
we have for any $\eta$ $b_m^{-\beta \gamma}\int_{\St \cap\{ |y -\E[\Zt]| > \eta\}} (y - \E[\Zt])^{\beta} \kappa_{t,b_m}(y) \dd y \xrightarrow[m \rightarrow + \infty]{} 0 $. And since $\E[(\Zt - t)^{l}] = O(b_m^{\beta \gamma})$ by Definition \ref{def:order_beta}, $R_m = o(b_m^{\beta \gamma})$. 
Since $(\kappa_{t,b_m})$ is a kernel of order $\beta$, we have for $1\leq n \leq l$,  $\E[(\Zt - t)^{n}] = O(b_m^{\beta \gamma})$ and thus
\begin{align}
   A_m =  \sum_{s= 1}^{l } \frac{k^{(s)}(t)}{s!}|\E[(\Zt - t)^{s}]| + o(b_m^{\gamma \beta}) = O(b_m^{\beta \gamma}).
   \label{eq:A_int}
\end{align}
}

Note that the $o$ term in \eqref{eq:A_int} is not necessarily uniform with respect to $t$ on any subset of $\R_+$. \blue{However, as $R_m$ is continuous as a function of $t$, Heine's theorem implies that it is uniformly continuous on a compact set.}

As shown in the proof of Proposition \ref{prop:inequ}, part 1, by boundedness of $k$, for any $n \in \mathbb{N}^*$,  
\blue{
\begin{align*}
  B_m  \leq  ||k||_{\infty}( \mathbb{P}(|Z_{t,b_m}-t|\geq n)  +  F(t+n)^m)
\end{align*}}
\blue{
By assumption, $\exists \mu >0$ such that $ \frac{1}{m^\mu} = o(b_m^{\gamma})$. Hence, since $F(t+n) <1$ (as $k\in \Sigma(\beta, L)$),  $F(t+n)^m = o(m^{-\beta \mu }) = o(b_m^{\beta \gamma})$.
For $\varepsilon>0$ let $n_0 \in \N$ be such that $\frac{1}{n_0^l} \leq \varepsilon$. 
By Lemma \ref{lemma:conv_pointwise}, 
\begin{align*}
    \mathbb{P}(|\Zt-t|\geq n_0) \leq \frac{C(t)}{n_0^l}b_m^{\beta \gamma} \leq \varepsilon b_m^{\beta \gamma}.
\end{align*}
Which yields for $m$ large enough $|B_m b_m^{-\beta\gamma}| \leq 2\varepsilon$ and thus $B_m = o(b_m^{\beta\gamma})$, and thus finally,
\begin{align*}
     \E[\hat{k}_m(t)] -k(t) =  A_m-B_m
     =   \sum_{s= 1}^{l } \frac{k^{(s)}(t)}{s!}|\E[(\Zt - t)^{s}]| + o(b_m^{\gamma \beta}).
\end{align*}
}

\subsubsection{Proof of Proposition \ref{lemma:var_eq_k} }

We now prove Proposition \ref{lemma:var_eq_k}, which gives an equivalent of the variance, as well as the consistency of the estimator as a corollary. 

 The exact expression of $Var(\hat{k}_m(t))$ is computed  for classical symmetric kernels in Theorem 1 in \cite{Tanner_Wong_HR_TCL}.
 As no assumption on the kernel is needed, the result can be directly extended to associated kernels: 
    \begin{align}
       &  Var(\hat{k}_m(t)) = \int_{\St}\kappa_{t,b_m}(y)^2 k(y) \left(\sum_{i = 0}^{m-1} \binom{m}{i} \frac{F(y)^i(1-F(y))^{m-i}}{m-i} \right)\dd y \nonumber\\
        & + 2\int_{\St} \int_{y\leq z} 
        \Big(F(z)^m-F(y)^mF(z)^m \!-\! \frac{1-F(y)}{F(z)\!-\!F(y)}(F(z)^m\!-\!F(y)^m)\Big)\kappa_{t,b_m}(y)\kappa_{t,b_m}(z) k(y) k(z) \dd y \dd z \label{eq:var_exact_k}
    \end{align}
    An asymptotic equivalent of the variance is proved in \cite{Hazard2} (Theorem 2) for classical kernels. In the following, this proof is adapted to our more general setting. First, by Lemma \ref{lemma:var_2term1} in the Appendix \ref{sec:lemma_append}, the second term in \eqref{eq:var_exact_k} is negligible compared to $\frac{\alpha_{b_m(t)}}{m}$,  where we recall 
    $$\alpha_{b_m}(t) = \int_{\St}\kappa_{t,b_m}^2(y) \dd y.$$ 
    Thus, it remains to prove that the first term of \eqref{eq:var_exact_k} is equivalent to $\frac{\alpha_{b_m}(t)}{m} \frac{k(t)}{1-F(t)}$.\\

By Lemma \ref{lemma_var_F}, for $y$ such that $|t-y|\leq \lambda$ with $\lambda$ defined in Assumption \ref{assump:compat},
\begin{equation}
\label{eq:conv_Im}
mI_m(y) := m \sum_{i = 0}^{m-1} \binom{m}{i}\frac{F(y)^i(1-F(y))^{m-i}}{m-i} \xrightarrow[m\rightarrow + \infty]{}(1-F(y))^{-1},
\end{equation}
uniformly in $y$.
Hence,
\begin{align}
    &\Bigg| \frac{m}{\alpha_{b_m}(t)} \int_{\St} \kappa_{t,b_m}^2(y) k(y)I_m(y) \dd y - \frac{k(t)}{1-F(t)} \Bigg | \nonumber\\
    & \leq \Bigg| \frac{m}{\alpha_{b_m}(t)} \int_{\{|t-y|\leq \lambda\} \cap \St} \kappa_{t,b_m}(y)^2 k(y)I_m(y) \dd y - \frac{1}{\alpha_{b_m}(t)}\frac{k(t)}{1-F(t)}\int_{\{|t-y|\leq \lambda\}\cap \St} \kappa_{t,b_m}(y)^2\dd y \Bigg|\nonumber\\
    & \qquad + \frac{m}{\alpha_{b_m}(t)}\int_{\{|t-y|>\lambda\} \cap \St} \kappa_{t,b_m}(y)^2 k(y) I_m(y) \dd y+ \frac{1}{\alpha_{b_m}(t)} \frac{k(t)}{1-F(t)}\int_{\{|t-y|>\lambda\} \cap \St} \kappa_{t,b_m}(y)^2 \dd y \label{eq:var_decomp}
\end{align}
Let us first study the first term in \eqref{eq:var_decomp}. We have
\begin{align}
    E_m &:= \Bigg| \frac{m}{\alpha_{b_m}(t)} \int_{\{|t-y|\leq \lambda \}\cap \St} \kappa_{t,b_m}(y)^2 k(y)I_m(y) \dd y - \frac{1}{\alpha_{b_m}(t)}\frac{k(t)}{1-F(t)}\int_{\{|t-y|\leq \lambda\}\cap \St} \kappa_{t,b_m}(y)^2\dd y \Bigg| \nonumber\\
    & \leq \Bigg| \int_{\{|t-y|\leq \lambda\}\cap \St} \frac{\kappa_{t,b_m}(y)^2k(y)}{\alpha_{b_m}(t)} \Bigg(mI_m(y)-\frac{1}{1-F(y)}\Bigg) \dd y \Bigg | \nonumber \\
    & \qquad +  \Bigg | \int_{\{|t-y|\leq \lambda \}\cap \St} \frac{\kappa_{t,b_m}(y)^2}{\alpha_{b_m}(t)}\left(\frac{k(t)}{1-F(t)}-\frac{k(y)}{1-F(y)}\right)\dd y \Bigg| \nonumber\\
    & \leq \sup_{|t-y|\leq \lambda}\{|mI_m(y) - (1-F(y))^{-1}|\} \sup_{y\in [t-\lambda t+\lambda]}(k(y)) +I_{1,m}. \label{eq:E_M_1}
\end{align}
\blue{
Where 
\begin{align*}
    I_{1,m} & = \Bigg | \int_{\{|t-y|\leq \lambda \}\cap \St} \frac{\kappa_{t,b_m}(y)^2}{\alpha_{b_m}(t)}\left(\frac{k(t)}{1-F(t)}-\frac{k(y)}{1-F(y)}\right)\dd y \Bigg| \\
    & = \Big| \E[\frac{\kappa_{t,b_m}(Z_{t,b_m})}{\alpha_{b_m}(t)}\left(\frac{k(t)}{1-F(t)}-\frac{k(Z_{t,b_m})}{1-F(Z_{t,b_m})}\right)\1_{|t-Z_{t,b_m}|\leq \lambda}] \Big| 
\end{align*}
}
By \eqref{eq:conv_Im}, $\sup_{|t-y|\leq \lambda}\{|mI_m(y) - (1-F(y))^{-1}|\} \rightarrow 0$. Hence, the first term in \eqref{eq:E_M_1} goes to 0. 
Let us now control the second term in \eqref{eq:E_M_1},  $I_{1,m}$. \blue{ Recall that $Z_{t,b_m} \xrightarrow{\mathbb P} t$. 
By Assumptions \ref{assump:gamma_sup} and \ref{assump:gamma_2_inf}, $\frac{\kappa_{t,b_m}(\cdot)}{\alpha_{b_m}(t)} \leq  \frac{C_s(t)}{ \underline{C}_3(t)}$ is bounded. Hence, by continuity of the density $f = \frac{k}{1-F}$,  we also have
$$\frac{\kappa_{t,b_m}(Z_{t,b_m})}{\alpha_{b_m}(t)}\left(\frac{k(t)}{1-F(t)}-\frac{k(Z_{t,b_m})}{1-F(Z_{t,b_m})}\right)\1_{|t-Z_{t,b_m}|\leq \lambda} \xrightarrow{\mathbb P} 0.$$ 
This sequence is also uniformly bounded in $m$ ($f$ is bounded on $|t-y|\leq \lambda$), and hence also converges in $L^1$, which yields that $I_{1,m} \xrightarrow[m \rightarrow +\infty]{} 0$ and thus $E_m \xrightarrow[m \rightarrow +\infty]{} 0$. } 

For the second term in \eqref{eq:var_decomp}, we have
\begin{align*}
    \frac{m}{\alpha_{b_m}(t)}\int_{\{|t-y|>\lambda\} \cap \St} \kappa_{t,b_m}(y)^2 k(y) I_m(y) \dd y &= \frac{1}{\alpha_{b_m}(t)}\int_{\{|t-y|>\lambda\} \cap \St} \left(\frac{\kappa_{t,b_m}(y)}{1-F(y)}\right)^2 (1-F(y)) f(y) mI_m(y) \dd y 
\end{align*}
By the compatibility Assumption \ref{assump:compat}, $\frac{\kappa_{t,b_m}(y)}{1-F(y)}$ is bounded on $|y-t|\geq \lambda$. In addition, \blue{for $m$ large enough, $mI_m(y) \leq \frac{2}{1-F(y)}$}, and by Assumption \ref{assump:gamma_2_inf}, $\frac{1}{\alpha_m(t)} = O(b_m^{\gamma})$, which proves that this term converges to $0$.\\

Finally, for the last term in \eqref{eq:var_decomp} we write
\begin{align*}
     \hspace{-0.5cm}\frac{1}{\alpha_{b_m}(t)} \frac{k(t)}{1-F(t)}\int_{\{|t-y|>\lambda\} \cap \St} \kappa_{t,b_m}(y)^2 \dd y& \leq \frac{\sup_{y \in \St}(\kappa_{t,b_m}(y))}{\alpha_{b_m}(t)} \frac{k(t)}{1-F(t)}\mathbb{P}(|Z_{t,b_m}-t|\geq \lambda )\xrightarrow[m \to \infty ]{}0.
\end{align*}
This finishes to prove that \eqref{eq:var_decomp} goes to 0 i.e. that the first term in \eqref{eq:var_exact_k} converges to $\frac{\alpha_{b_m}(t)}{m} \frac{k(t)}{1-F(t)}$.

Hence for a fixed $t \in \mathbb{S}$, 
        \begin{align}
    Var(\hat{k}_m(t)) = \frac{\alpha_{b_m}(t)}{m}\frac{k(t)}{1-F(t)} + h_m(t)\frac{\alpha_{b_m}(t)}{m}.\label{eq:var_comp_k}
\end{align}
with $h_m(t)\xrightarrow[m \rightarrow +\infty]{}0$. 
In particular, as $\E[\hat{k}_m(t)] \xrightarrow[m\rightarrow + \infty]{}k(t) $ by Proposition \ref{prop:inequ}, we have 
\begin{equation}
     \hat{k}_m(t) \xrightarrow[m \rightarrow + \infty]{\mathbb{P}} k(t). \nonumber
 \end{equation}

\subsubsection{Proof of Theorem \ref{thm:MISE_k}}

Finally, let us prove the main theorem, which states the convergence and provides an asymptotic equivalent of the MISE. With the classical bias-variance decomposition, we have
$$MISE(\hat{k}_m)  = \int_I Var(\hat{k}_m(t)) +\E[\hat{k}_m(t) - k(t)]^2 \dd t.$$ 
We start by proving the integrated equivalent of the variance. To prove \eqref{eq:var_form_k}, we integrate equation \eqref{eq:var_comp_k} over the compact set $I$.
\begin{align}
    \int_I Var(\hat{k}_m(t)) \dd t &=  \int_I \frac{\alpha_{b_m}(t)}{m}\frac{k(t)}{1-F(t)} \dd t + \int_I h_m(t)\frac{\alpha_{b_m}(t)}{m} \dd t\label{eq:var_int1_k} := J_{1,m} + J_{2,m}.
\end{align}
The function $t \rightarrow \alpha_{b_m}(t) =\int_{\St} \kappa_{t,b_m}(y)^2 \dd y  $ is continuous in $t$ on the compact set $I$ by dominated convergence using \ref{assump:unif_int} and \ref{assump:gamma_1} (ii). \blue{Since it is continuous and for any $t\in \St, \alpha_{b_m}(t) \leq C_s(t)b_m^{-\gamma}$ with $C_s$ a locally bounded function of $t$ by Assumption \ref{assump:gamma_sup}, we have
\begin{equation*}
    \forall m \in \mathbb{N}, \, \forall t \in I, |\alpha_{b_m}(t)| \leq \sup_{t\in I}(C_s(t)) b_m^{-\gamma} = C b_m^{-\gamma}.
\end{equation*}}
\blue{Furthermore, $t \rightarrow Var(\hat{k}_m(t))$ as expressed in \eqref{eq:var_exact_k} is also continuous in $t$, indeed by Definition \ref{def:cont_ass} for a fixed $y$, $t \mapsto \kappa_{t,b_m}(y)$ is continuous and Assumption \ref{assump:intt} allows us to conclude by dominated convergence.}
\blue{Thus, since  $t \rightarrow Var(\hat{k}_m(t))$, $t \mapsto \alpha_{b_m}(t)$, $k$ and $(1-F)^{-1}$ are continuous in $t$ on $I$, $h_m$ is continuous and hence bounded on $I$. }
Hence, 
\begin{align*}
    &\Big| \int_I h_m(t)\frac{\alpha_{b_m}(t)}{m} \dd t\Big| \leq C m^{-1}b_m^{-\gamma} \int_I |h_m(t)| \dd t
\end{align*}
and $\int_I |h_m(t)| \dd t$ tends to $0$ by dominated convergence, thus $J_{2,m} = o(J_{1,m})$. We can therefore take the limit under the integral in \eqref{eq:var_int1_k} which yields
\begin{align}
    \int_I Var(\hat{k}_m(t)) \dd t  = \int_I\frac{\alpha_{b_m}(t)}{m}\frac{k(t)}{1-F(t)} \dd t +
    o\left(m^{-1}b_m^{-\gamma}\right).\label{eq:var_form_k}
\end{align}

Proposition \ref{thm:cvgce_exp_k} provides an equivalent of the expression under the integral for the second term, namely 
 \blue{
\begin{equation*}
     \E[\hat{k}_m(t)] -k(t) =   \sum_{s= 1}^{l } \frac{k^{(s)}(t)}{s!}|\E[(\Zt - t)^{s}]| + o(b_m^{\gamma \beta}). 
\end{equation*}
}
\blue{Note that $I$ is a compact set and $\E[\hat{k}_m(t)] -k(t) - \sum_{s= 1}^{l } \frac{k^{(s)}(t)}{s!}|\E[(\Zt - t)^{s}]| $ is continuous in $t$. Thus, as noted in the proof of Proposition \ref{thm:cvgce_exp_k}, the $o(b_m^{\gamma \beta})$ can be taken to be uniformly continuous with respect to $t$ on the compact set $I$ and we can exchange the integration and the $o$.}
Hence,
\blue{
\begin{align}
    \int_I (\E[\hat{k}_m(t)] - \!k(t))^2 \dd t &=\! \int_I \Big(\sum_{s= 1}^{l } \frac{k^{(s)}(t)}{s!}|\E[(\Zt - t)^{s}]|\Big)^2 \dd t + o(b_m^{2 \gamma \beta}).  
    \label{eq:int_bias_k}
\end{align}
}

Combining \eqref{eq:var_form_k} with \eqref{eq:int_bias_k} yields \eqref{eq:mise_k}.

\subsection{Proof of Corollary \ref{coro:MISE}}
\blue{By Proposition \ref{prop:gamma}, the Gamma kernel is of order $(2, \frac{1}{2})$. }

We have (see the proof of Proposition \ref{prop:gamma} in Section \ref{sec:proof_2.1})
\begin{equation*}
     \alpha_{b_m}(t) \substack{\sim\\ b_m \rightarrow 0} \frac{1}{2\sqrt{\pi tb_m}}.
\end{equation*}
Hence
\begin{align}
     \int_I Var(\hat{k}_m(t)) \dd t  =\frac{1}{2 m \sqrt{\pi b_m}} \int_I t^{-1/2}\frac{k(t)}{1-F(t)}\dd t  + o(m^{-1}b_m^{-1/2}) \label{eq:var_form}
\end{align}
\blue{
Since $|\Lambda(t,b_m)| = 0$ and $Var(t,b_m) = t b_m$ for $t > 2b_m$ , we have by Theorem \ref{thm:MISE_k} for $\beta = 2$
\begin{align*}
    &\int_I \Big(\sum_{s= 1}^{2} \frac{k^{(s)}(t)}{s!}|\E[(\Zt - t)^{s}]|\Big)^2 \dd t = \int_I \Big( k'(t)|\E[(\Zt - t)]| + \frac{k''(t)}{2}\E[(\Zt - t)^{2}]\Big)^2 \dd t\\
    & = \int_I \Big( k'(t)|\Lambda(t,b_m)| + \frac{k''(t)}{2}(\Lambda(t,b_m)^2 + Var(t,b_m))\Big)^2 \dd t\\
    & = \int_I \frac{k''(t)^2}{4} Var(t,b_m)^2 \dd t =    b_m^2 \int_I t^2 \frac{k''(t)^2}{4} \dd t.  
\end{align*}
}

\section{Adaptive bandwidth choice}
\label{sec:minimax}
For the sake of clarity in this section,  we change our notation and write $\hat{k}_b$ instead of $\hat{k}_m$ to refer to the $b$-dependent estimator and emphasize the dependence of the estimator on the bandwidth $b$. 
\subsection{Presentation}
In practice, a statistical study is done with a fixed number of observations $m$, and the choice of the bandwidth can significantly impact the quality of the estimator. In the 90s, Lepski developed a data-driven minimax bandwidth selection method   
(\cite{Lepski_91, Lepski_92,Lepski_93}), which was later modified for density estimation by Goldenshluger and Lepski in \cite{GL_11}.  Let us describe briefly the heuristics behind the approach (see e.g. \cite{Doumic_2012,PCO_17} for more details, for the kernel density estimator).\\
The aim is to select a bandwidth $b$ which minimizes the MSE $\E[(\hat{k}_b(t) - k(t))^2]$ (although another metric could be considered).  The ideal bandwidth which minimizes the MSE, thus being the perfect compromise between bias and variance, is called the "oracle". 
However, this quantity depends on the real hazard rate $k$ and thus cannot be directly computed. By Proposition \ref{lemma:var_eq_k}, the variance of our estimator can be tightly approached by $\frac{k(t)}{1-F(t)}\frac{\alpha_b(t)}{m}$ when the sample size $m$ is reasonably large. And we have for some positive constant $\alpha$
\begin{equation}
    \E[(\hat{k}_b(t) - k(t))^2]  \leq (\E[\hat{k}_b(t)] - k(t))^2 + \alpha e^{||k||_{\infty} t} \frac{1}{mb_m^{\gamma}}.
\end{equation}
However, the expression of the bias depends on $k$ and its derivatives and is much more complicated to approximate. In comparison, the variance is bounded only knowing an upper bound of the hazard rate and in the case of density estimation, the variance can even be bounded independently of the underlying density (\cite{GL_11}).
Given a set of bandwidths, $\mathcal{B}_m$, the idea is to approach the bias term for $\hat{k}_b(t)$ by a data driven estimator, 
\begin{equation}
    \sup_{b' \in \mathcal{B}_m}\Big\{(\hat{k}_{b'}(t) - \hat{k}_{b, b'}(t))^2 - \frac{\chi}{mb'^{\gamma}}\Big\}_+,
\end{equation}
for some constant $\chi$, and where $\hat{k}_{b,b'}$ is an estimator of the hazard rate which depends both on $b$ and $b'$ and $\{x\}_+$ denotes the positive part of $x$, $\max(0,x)$. The optimal bandwidth $\hat{b}(t)$ minimizes the sum of this estimated bias and the estimated variance i.e.
\begin{equation*}
    \hat{b}(t) = \text{argmin}_{b \in \mathcal{B}_m}\Big\{ \sup_{b' \in \mathcal{B}_m}\Big\{(\hat{k}_{b'}(t) - \hat{k}_{b, b'}(t))^2 - \frac{\chi}{mb'^{\gamma}}\Big\}_+ +\frac{\chi}{mb^{\gamma}} \Big\}.
\end{equation*}
Thus defined, $\hat{b}(t)$ is (hopefully) such that for all $b \in \mathcal{B}_m$
\begin{equation}
\label{eq:oracle_exp}
     \E[(\hat{k}_{\hat{b}(t)}(t) - k(t))^2] \leq C\E[(\hat{k}_{b}(t) - k(t))^2] +R_m
\end{equation}
where $R_m \rightarrow 0 $ and $C>0$, which is an oracle-type inequality, as the selected bandwidth does better than all of the bandwidths considered, and is thus the closest to the oracle bandwidth. 

In the theory of kernel estimation, the functional $\hat{k}_{b,b'}$ used in the vast majority of cases is
\begin{equation*}
    \hat{k}_{b,b'}(t) =  \frac{1}{b}\kappa(./b) * \hat{k}_{b'}(t) =  \frac{1}{b'}\kappa(./b') * \hat{k}_{b}(t) = \hat{k}_{b',b}(t).
\end{equation*}
Indeed, in the case of symmetric kernels, the estimator itself is defined as a convolution of the kernel and the empirical hazard/density, making it compatible with a convolution-based definition of $\hat{k}_{b,b'}$. However this is not the case in our general framework, which leads us to use another criterion mentioned in \cite{Doumic_2012,Lepski_22} and used for classical kernels in e.g. \cite{LACOUR20163774,BERTIN2017115,Kerky_anisotropic}, namely 
\begin{equation*}
    \hat{k}_{b,b'}(t) = \hat{k}_{b \vee b'}(t),
\end{equation*}
where $\vee$ denotes the maximum operator.
 The existing results on adaptive minimax bandwidth choice for kernel hazard rate estimation (\cite{recurrent_event_Bouaziz}) consider a kernel defined on a bounded support only,
  which significantly simplifies the proofs as this allows to have, on the support of the kernel, a bounded hazard rate as well as a survival function with a strictly positive lower bound.
  In our case however, we allow the kernel to have an infinite support and, unlike in density estimation (\cite{REYNAUDBOURET2011115}), the assumption of compact support cannot be transferred to the hazard rate as considering a  hazard rate on a compact support implies that it is unbounded and that the survival function tends to $0$ on that support. This entails further assumptions on the kernel to ensure that is decreases sufficiently quickly compared to the survival function. Furthermore, proving oracle type inequalities necessitates the use of concentration inequalities (\cite{PCO_17}), which apply to sums of independent random variables. To that effect, studying the hazard rate estimator as opposed to the density estimator involves introducing an intermediate pseudo-estimator which is a sum of independent terms and studying the difference between the initial estimator and the intermediate one as is done in \cite{recurrent_event_Bouaziz}.\\
In the following, we present the results for both a pointwise bandwidth selection procedure, where a different bandwidth is selected at each point of estimation, and a global procedure where a single bandwidth is selected for an estimation interval. 

\subsection{Pointwise adaptive bandwidth selection}
\label{sec:local_minimax}
In this subsection, we fix $t \geq 0$. We denote $\kinf : = \sup_{|y-t| \leq \lambda} (k(y))$.
We consider a finite set of bandwidth, $\mathcal{B}_m$.
 We define
\begin{align}
V_0(b,t) = \frac{\kappa_0 \log(m)}{mb^{\gamma}}e^{\kinf (t+\lambda)} \kinf C_s(t),\label{eq:v_0}
\end{align}
with $\kappa_0$ a numerical constant \blue{and with  $C_s(t)$ and $\lambda$ as introduced in  Assumptions \ref{assump:gamma_sup} and \ref{assump:compat}}, and
\begin{align*}
     &A_0(b,t) = \sup_{b' \in \mathcal{B}_m}\Big \{ (\hat{k}_{b'}(t) - \hat{k}_{b'\vee b}(t))^2 - V_0(b',t) \Big \}_+.
\end{align*}
The adaptive bandwith choice and kernel estimator are formally defined by:
\begin{align*}
    &\hat{b}(t) = \text{argmin}_{b' \in \mathcal{B}_m}(A_0(b',t) + V_0(b',t) )\\
    &\Check{k}(t) = \hat{k}_{\hat{b}(t)}(t).
\end{align*}
We introduce the following assumption, which is a stronger version of the compatibility assumption \ref{assump:compat}. This assumption is not necessary when estimating with kernels defined on a bounded support, as one can assume that $1/(1-F)$ stays bounded on the support of the kernel. 

\begin{assumption} \label{assump_ad_strong}
 The kernel $\kappa_{t,b}$ and $F$ are strongly compatible, i.e. there exists $\lambda >0$ such that for any fixed $t \in \St$, $\exists b_0 >0 , \, \exists G(t)>0, B(t) >0$, 
     \begin{equation}
    \begin{aligned}
    &\forall b \leq b_0, \forall y \in \St, |y-t|>\lambda ,\implies \frac{\kappa_{t,b}(y)}{1-F(y)}  < G(t)e^{-B(t)/b^{\gamma}}. \label{eq:compatibility_strong}
     \end{aligned}
   \end{equation} 
\end{assumption}

\blue{
\begin{remark}
    Assumption \ref{assump_ad_strong} can be replaced by a weaker version, where the upper bound is expressed as $G(t)f_t(b)$ for some function $f_t$ such that $f_t(b) \xrightarrow[b\to 0]{}0$. In that case, the bandwidths in the bandwidth set need to verify $ \sup_{b_m \in \mathcal{B}_m}(m^2 f_t(b_m)) \leq C$ for some constant $C$ independent of $m$. Thus, the quicker $f_t$ goes to $0$ at $0$, the higher the upper bound on $\mathcal{B}_m$ is. 
\end{remark}}

We state the following proposition, which ensures that the previous assumptions apply to the Gamma kernel. We postpone the proof to the Appendix (see Section \ref{sec:proof_2.2}). 
\begin{prop}\label{prop:gamma_2}
    The Gamma kernel without interior bias defined in \ref{def:gamma} verifies Assumption \ref{assump_ad_strong} with $\gamma = 1/2$. 
\end{prop}
Then we have 
\begin{thm}[Pointwise adaptive bandwith estimation] \label{thm:oracle_point}
Let $\hat{k}_b(t)$ be defined by \eqref{eq:k_estim} with a kernel verifying Definition \ref{def:cont_ass}. \blue{Suppose $k \in \Sigma(\beta, L)$
for some unknown $\beta > 0$. 
Assume further that $(\kappa_{t,b})$ is a
kernel of order $(\tilde{\beta}, \gamma)$ with $\tilde{\beta} \geq \beta$,
satisfying Assumptions \ref{ass:order}, \ref{assump:gamma_sup},
\ref{assump:unif_int} and \ref{assump_ad_strong}.}
Consider a finite set of bandwidths $\mathcal{B}_m$ such that $\text{Card}(\mathcal{B}_m) \leq m$ and 
 $$ \forall \;  b \in \mathcal{B}_m, \quad  \max(\frac{1}{m}, \kappa_1 \frac{\log(m)}{m} ) \leq b^{\gamma} \leq \min(\frac{B(t)}{\log(m)},b_0^{\gamma}),$$ 
 
with $B(t)$ and $b_0$ defined in Assumption \ref{assump_ad_strong}, and $\kappa_1 = \frac{16}{9\kinf C_s(t)} (G(t) + C_s(t)e^{\kinf (t+\lambda)})^2$. \\
Let $\displaystyle \mathcal{S}(\mathcal{B}_m) = \sum_{b \in \mathcal{B}_m} \frac{1}{mb^{\gamma}}$. Then, provided $\kappa_0 \geq 80$, there exists constants $C_0, C_1$ and $C_2$  such that $\forall \;  b \in \mathcal{B}_m$
\begin{align}\label{eq:oracle_point_1}
 \E[(\Check{k}(t) - k(t))^2]&\leq  3\E[(\hat{k}_b(t) - k(t))^2] + \blue{C_0b^{2\beta\gamma}} + 6 V_0(b,t) + \frac{\log(m)}{m}(C_1 + C_2\mathcal{S}(\mathcal{B}_m)). 
\end{align}
\end{thm}
\blue{The first term in \eqref{eq:oracle_point_1} is the MSE of the estimator, which we can control by Theorem \ref{thm:MSE_k}. The second term is of same order as the square bias by Theorem \ref{thm:cvgce_exp_k}. By definition of $V_0(b,t)$ in \eqref{eq:v_0} and Proposition \ref{lemma_var_F}, the third term has the same order as the variance of the estimator, penalized by a factor $\log(m)$. The last term in $\frac{\log(m)}{m}$ is a penalization term. Thus, under the assumptions of Theorem \ref{thm:MSE_k}, the right hand side of \eqref{eq:oracle_point_1} is of order $b^{\blue{2\beta \gamma}} + \frac{\log(m)}{m b^{\gamma}} + \frac{\log(m)}{m}\mathcal{S}(\mathcal{B}_m)$.  If the bandwidth set  contains a bandwidth $b$ of order $(\frac{\log(m)}{m})^{\frac{1}{\gamma(2\beta + 1)}}$ and  $\mathcal{S}(\mathcal{B}_m)$ is of order at most $(\frac{m}{\log(m)})^{\frac{1}{\gamma(2\beta + 1)}}$, \blue{the optimal minimax convergence rate of $m^{-\frac{2\beta}{2\beta + 1}}$ (as shown in \cite{HUBER2003209,Pons01011986}) is almost achieved by the local adaptive bandwidth choice procedure.} There is  a logarithmic loss compared to the theoretical optimal asymptotic rate of convergence for the pointwise estimation, as it is the case in \cite{recurrent_event_Bouaziz, Comte_multidim_2013}. This is not the case in the global setting (see Section \ref{sec:global_minimax}).}

\begin{remark}

 The fact that the support of the kernel is unbounded leads to assume that the tails of distribution of jumping times vanish quickly enough for any bandwidth in $\mathcal{B}_m$, which translates mathematically into a condition on the upper bound of $\mathcal{B}_m$, which goes to $0$, a similar assumption to the ones in \cite{GL_11,Comte_multidim_2013}. 
\end{remark}

Considering the estimator under the form \eqref{k:other_form} allows us to rely on a similar proof strategy as in \cite{recurrent_event_Bouaziz}, in the case of recurrent event intensity estimation. Briefly, the proof strategy is to study the error of the adaptive bandwidth estimator by using the triangle inequality to bound it with several terms. These terms can then be handled either using the results from Section \ref{sec:results} or via concentration inequalities applied to the following pseudo-estimator:
\begin{align}
   &\Tilde{k}_b(t) = \frac{1}{m}\sum_{i=1}^m \frac{\kappa_{t,b}(\tau_i)}{1-F(\tau_i)}. \label{def:k_tilde}
\end{align}

We begin with the following technical Lemma, whose proof is presented in the Appendix \ref{sec:lemma_append} and which is similar to Lemma 2 in \cite{recurrent_event_Bouaziz}. \blue{For ease of notations, we consider in the following  that $b_0 \leq 1$. }
\begin{lemma} \label{lemma:emp_distrib}
For $c_F(t) >0 $ a given constant and $c_0 >0$, define the following events
\begin{align*}
    &\Omega_{t}^{*} = \{ \omega : \forall x, F(x) - \hat{F}_m(x) \geq -c_F(t)\}\\
    & \Omega_{c_0}^{*} = \{ \omega : \forall x,| F(x) - \hat{F}_m(x)| \leq c_0\sqrt{m^{-1}log(m)}\}\\
    & \Omega_{c_0,t} = \Omega_{t}^{*} \cap \Omega_{c_0}^{*}.
\end{align*}
For any $l \in \mathbb{N}^*$ and any $t \in \R_+$, if $c_0 \geq \max(\sqrt{l/2}, 1/m)$, $\exists c_l, \tilde{c}_l >0$, 
\begin{equation*}
    \mathbb{P}(\Omega_{c_0,t}^c) \leq(c_l +\frac{\tilde{c}_l}{c_F(t)^{2l}}) m^{-l}.
\end{equation*}
\end{lemma}

The following Lemma provides control on the difference between $\hat{k}_b(t)$ and $\tilde{k}_b(t)$ on the event $\Omega_{c_0,t}^c$ and the bounds it provides will be useful in the proof of Theorem \ref{thm:oracle_point}. 

\begin{lemma}\label{lemma:event_1}
Suppose that $(\kappa_{t,b})$ is an associated kernel, verifying Assumptions \ref{assump:gamma_sup} and \ref{assump:compat}. Suppose $\forall \; b \in \mathcal{B}_m, \; m b^{\gamma} \geq 1$  and $\text{Card}(\mathcal{B}_m) \leq m$. 
Then, there exists \blue{$C(t) \in \L^{\infty}_{\text{loc}}(\R_+, \R_+^*)$} such that for any $c_0 \geq \max(\sqrt{13/2}, 1/m)$, 
\begin{align*}
    \E[\sup_{b\in \mathcal{B}_m}(\hat{k}_b(t)-\Tilde{k}_b(t))^2 \1_{\Omega_{c_0,t}^c}] \leq C(t) m^{-2}
\end{align*}   
where $\Omega_{c_0,t}$ is defined in Lemma \ref{lemma:emp_distrib}, with $c_F(t) = (1-F(t+\lambda))/2$.
\end{lemma}

\begin{proof}

We have
\begin{align*}
    \E[(\hat{k}_b(t)-\Tilde{k}_b(t))^2\1_{\Omega_{c_0,t}^c}] &= \frac{1}{m^2}\E\Big[\Big(\sum_{i=1}^m \frac{\kappa_{t,b}(\tau_i)(\hat{F}_m(\tau_i) - F(\tau_i))}{(1-\hat{F}_m(\tau_i))(1-F(\tau_i))}\Big)^2\1_{\Omega_{c_0,t}^c}\Big] \leq \E\Big[\Big(\sum_{i=1}^m \frac{\kappa_{t,b}(\tau_i)}{1-F(\tau_i)}\Big)^2\1_{\Omega_{c_0,t}^c}\Big],
\end{align*}
since $1-\hat{F}_m(\tau_i) \geq m^{-1}$ for all $1 \leq i \leq m$ and $|\hat{F}_m - F| \leq 1$. By applying the Cauchy-Schwarz inequality two times and by independence of $(\tau_i)_{1\leq i \leq m}$, we obtain that
\blue{\begin{align*}
     \E[(\hat{k}_b(t)-\Tilde{k}_b(t))^2\1_{\Omega_{c_0,t}^c}]  &\leq m\E\Big[\1_{\Omega_{c_0,t}^c}\sum_{i=1}^m \frac{\kappa_{t,b}^2(\tau_i)}{(1-F(\tau_i))^2}\Big]\\
    &\leq  m^2 \sqrt{\mathbb{P}(\Omega_{c_0,t}^c)} \Big[ \int_0^{+\infty} \frac{\kappa_{t,b}(y)^4 f(y) }{(1-F(y))^4}\dd y\Big]^{1/2}. 
\end{align*}}
 Furthermore, by Assumptions  \ref{assump:gamma_sup} and \ref{assump:compat}, and recalling that $f = k(1-F)$, 
\begin{align*}
  \int_0^{+\infty} \frac{\kappa_{t,b}(y)^4 f(y) }{(1-F(y))^4}\dd y  & =   \int_{\{|y-t|\leq \lambda\}} \frac{\kappa_{t,b}(y)^4 f(y)} {(1-F(y))^4}\dd y + \int_{\{|y-t|> \lambda\}} \frac{\kappa_{t,b}(y)^4 f(y)} {(1-F(y))^4}\dd y \\
  & \leq  G(t)^4 +\kinf \frac{b^{-3\gamma} C_s(t)^3}{(1-F(t+\lambda))^{3}}. 
\end{align*}
Combining this result with  Lemma \ref{lemma:emp_distrib}, we obtain for all $l \in \mathbb{N}^*$ and $c_0 \geq \max(\sqrt{l/2}, 1/m)$: 
\begin{align*}
  \E[ \sup_{b' \in \mathcal{B}_m} (\hat{k}_{b'}(t)-\Tilde{k}_{b'}(t))^2\1_{\Omega_{c_0,t}^c}] &\leq \sum_{b' \in \mathcal{B}_m}\E[(\hat{k}_{b'}(t)-\Tilde{k}_{b'}(t))^2\1_{\Omega_{c_0,t}^c}] \\
  & \leq  \sum_{b' \in \mathcal{B}_m}  m^{2-l/2} \sqrt{(c_l +\frac{\tilde{c}_l}{c_F(t)^{2l}})} \Big( G(t)^2 +\sqrt{\kinf}\frac{b'^{-3\gamma/2} \sqrt{C_s(t)^3}}{(1-F(t+\lambda))^{3/2}} \Big). 
\end{align*}
By assumption $\forall b' \in \mathcal{B}_m, b'^{-\gamma} \leq m$, and $\text{Card}(\mathcal{B}_m) \leq m$. Thus, \blue{for a fixed $l\geq 13$, }
$$\E[ \sup_{b' \in \mathcal{B}_m} (\hat{k}_{b'}(t)-\Tilde{k}_{b'}(t))^2\1_{\Omega_{c_0,t}^c}] \leq C(t)m^{-2}.$$
\blue{With $C(t) = \sqrt{(c_l +\frac{\tilde{c}_l}{c_F(t)^{2l}})} \Big( G(t)^2 +\sqrt{\kinf}\frac{ \sqrt{C_s(t)^3}}{(1-F(t+\lambda))^{3/2}} \Big) \in L^{\infty}_{\text{loc}}(\R_+, \R_+)$, as $C_s$ and $G$ are also locally bounded functions of $t$ by Assumptions \ref{assump:gamma_sup} and \ref{assump:compat}}.

\end{proof}

We now move on to another Lemma, where we control the difference between $\hat{k}_b(t)$ and $\tilde{k}_b(t)$ on $\Omega_{c_0,t}$. Recall that  $ \mathcal{S}(\mathcal{B}_m) = \sum_{b \in \mathcal{B}_m} \frac{1}{mb^{\gamma}}$. 

\begin{lemma}\label{lemma:event_2}
Suppose $(\kappa_{t,b})$ and $F$ verify the strong compatibility Assumption \ref{assump_ad_strong},  and $(\kappa_{t,b})$ is an associated kernel verifying \ref{assump:gamma_sup}. For all $b \in \mathcal{B}_m$,  suppose also that $\frac{1}{m}\leq b^{\gamma} \leq \min\Big(\frac{B(t)}{\log(m)}, b_0^{\gamma} \Big)$.\\
\blue{Then, for any fixed $c_0 \geq \max(\sqrt{13/2}, 1/m)$,} there exists $A_1(t), A_2(t) \in L_{\text{loc}}^{\infty}(\R_+, \R_+^*)$ such that 
    \begin{align*}
    \E[\sup_{b'\in \mathcal{B}_m}&(\hat{k}_{b'}(t)-\Tilde{k}_{b'}(t))^2\1_{\Omega_{c_0,t}}]  \leq \frac{\log(m)}{m}(A_1(t) + A_2(t) S(\mathcal{B}_m)).
\end{align*}
where $\Omega_{c_0,t}$ is defined in Lemma \ref{lemma:emp_distrib} with $c_F(t) = (1-F(t+\lambda))/2$, and $\tilde{k}_b$ introduced in \eqref{def:k_tilde}.
\end{lemma}
\begin{proof}

We split the expectation into two subevents, and use the fact that $| F(x) - \hat{F}_m(x)| \leq c_0\sqrt{m^{-1}log(m)}$ on $\Omega_{c_0,t}$.  
\begin{align}
    \E[\sup_{b'\in \mathcal{B}_m}&(\hat{k}_{b'}(t)-\Tilde{k}_{b'}(t))^2\1_{\Omega_{c_0,t}}] \nonumber&\\
    &\leq \frac{c_0^2 \log(m)}{m}\E\Big[\sup_{b'\in \mathcal{B}_m}\frac{1}{m^2}\Big(\sum_{i=1}^m \frac{\kappa_{t,b'}(\tau_i)}{(1-\hat{F}_m(\tau_i))(1-F(\tau_i))}(\one_{\{|t-\tau_i| \geq \lambda\}} + \one_{\{|t-\tau_i| < \lambda\}})\1_{\Omega_{c_0,t}}\Big)^{2}\Big] \nonumber\\
    & \leq \frac{c_0^2 \log(m)}{m}\E\Big[\sup_{b'\in \mathcal{B}_m}\frac{2}{m^2}\Big(\sum_{i=1}^m \frac{\kappa_{t,b'}(\tau_i)}{(1-\hat{F}_m(\tau_i))(1-F(\tau_i))} \one_{\{|t-\tau_i| < \lambda\}}\1_{\Omega_{c_0,t}}\Big)^{2} \nonumber\\
    & \qquad +\sup_{b'\in \mathcal{B}_m} \frac{2}{m^2}\Big(\sum_{i=1}^m \frac{\kappa_{t,b'}(\tau_i)}{(1-\hat{F}_m(\tau_i))(1-F(\tau_i))}\one_{\{|t-\tau_i| \geq \lambda\}} \1_{\Omega_{c_0,t}}\Big)^{2}\Big].\label{eq:events_sep}
\end{align}
Let us control the first term. 
 With $c_F(t) = (1-F(t+\lambda))/2$, we have on $\Omega_{c_0,t}\cap (|\tau_i-t| < \lambda)$, 
$$1-\hat{F}_m(\tau_i) = 1- F(\tau_i)  + F(\tau_i)-\hat{F}_m(\tau_i)  \geq 1 - F(\tau_i) -(1-F(t+\lambda))/2 \geq (1-F(t+\lambda))/2.$$
Hence, 
\begin{align*}
   \E&\Big[\sup_{b'\in \mathcal{B}_m} \frac{1}{m^2}2\Big(\sum_{i=1}^m \frac{\kappa_{t,b'}(\tau_i)}{(1-\hat{F}_m(\tau_i))(1-F(\tau_i))} \one_{\{|t-\tau_i| < \lambda\}}\1_{\Omega_{c_0,t}}\Big)^{2}\Big]\\
   &\leq \frac{8}{(1-F(t+\lambda))^4} \E\Big[\one_{\{|t-\tau_i| < \lambda\}}\sup_{b'\in \mathcal{B}_m}\frac{1}{m^2}\Big(\sum_{i=1}^m \kappa_{t,b'}(\tau_i) \Big)^{2}\Big] \\
    & \leq \frac{16}{(1-F(t+\lambda))^4} \left(\E\Big[\one_{\{|t-\tau_i| < \lambda\}}\sup_{b'\in \mathcal{B}_m}\Big(\frac{1}{m}\sum_{i=1}^m \kappa_{t,b'}(\tau_i) - \E[\kappa_{t,b'}(\tau_1)] \Big)^{2}\Big] +  \sup_{b'\in \mathcal{B}_m}\E[\one_{\{|t-\tau_i| < \lambda\}}\kappa_{t,b'}(\tau_1)] ^2 \right)\\
    &\leq \frac{16}{(1-F(t+\lambda))^4}\left(\sum_{b'\in \mathcal{B}_m} Var\Big(\one_{\{|t-\tau_i| < \lambda\}}\frac{1}{m}\sum_{i=1}^m \kappa_{t,b'}(\tau_i)\Big) +\sup_{b'\in \mathcal{B}_m}\E[\one_{\{|t-\tau_i| < \lambda\}}\kappa_{t,b'}(\tau_1)] ^2 \right).
\end{align*}
Recalling that $\tau_1$ has the pdf $f = k(1-F) \leq k$,  we have $\E[\kappa_{t,b'}(\tau_1) \one_{\{|t-\tau_1| < \lambda\}}] = \E[f(Z_{t,b'})\one_{\{|t-Z_{t,b'}| < \lambda\}}] \leq \kinf $, with $Z_{t,b'}$ a random variable of pdf $\kappa_{t,b'}$. Furthermore, 
$$Var\Big(\frac{1}{m}\sum_{i=1}^m \kappa_{t,b'}(\tau_i)\Big) \leq \frac{1}{m} \E[ \kappa_{t,b'}(\tau_1)^2] \leq \frac{\kinf \alpha_b'(t) }{m},$$
where $\alpha_b(t) = \int_{\St} \kappa_{t,b}(x)^2 \dd x \leq C_s(t)b^{-\gamma}$,  by Assumption \ref{assump:gamma_sup} and Remark \ref{rem:alpha_beta_bound}. Thus,
\begin{align*}
    \E\Big[\sup_{b'\in \mathcal{B}_m}\frac{1}{m^2}2\Big(\sum_{i=1}^m &\frac{\kappa_{t,b'}(\tau_i)}{(1-\hat{F}_m(\tau_i))(1-F(\tau_i))} \one_{\{|t-\tau_i| < \lambda\}}\1_{\Omega_{c_0,t}}\Big)^{2}\Big] \\
    & \leq \frac{16}{(1-F(t+\lambda))^4}\sum_{b'\in \mathcal{B}_m} \frac{\kinf  \alpha_{b'}(t)}{m} +\frac{16\kinf^2}{(1-F(t+\lambda))^4}\\
    &\leq \frac{16\kinf}{(1-F(t+\lambda))^4}(C_s(t) S(\mathcal{B}_m) + \kinf).
\end{align*}

For the second term of \eqref{eq:events_sep}, by Assumption \ref{assump_ad_strong} and since $\forall b \in \mathcal{B}_m, b^{\gamma} \leq \min\Big( b_0^{\gamma}, \frac{B(t)}{\log(m)}\Big)$, we have
\begin{align*}
    \E\Big[\sup_{b'\in \mathcal{B}_m}&\frac{1}{m^2} 2\Big(\sum_{i=1}^m \frac{\kappa_{t,b'}(\tau_i)}{(1-\hat{F}_m(\tau_i))(1-F(\tau_i))}\one_{\{|t-\tau_i| \geq \lambda\}} \Big)^{2}\1_{\Omega_{c_0,t}}\Big]
    \leq \E\Big[\sup_{b'\in \mathcal{B}_m} 2\Big(\sum_{i=1}^m \frac{\kappa_{t,b'}(\tau_i)}{(1-F(\tau_i))}\one_{\{|t-\tau_i| \geq \lambda\}} \Big)^{2}\Big]\\
    &\leq 2m^2 G(t)^2\sup_{b'\in \mathcal{B}_m} e^{-2B(t)/b'^{\gamma}}  \leq 2m^2 G(t)^2e^{-2\log(m)} \leq 2G(t)^2. 
\end{align*}
\blue{This yields the announced result with $A_1(t) = \frac{16 c_0 \kinf^2}{(1-F(t+\lambda))^4}+2 c_0 G(t)^2 \in L_{\text{loc}}^{\infty}(\R_+, \R_+^*) $ and $A_2(t) =  \frac{16 c_0 \kinf C_s(t)}{(1-F(t+\lambda))^4} \in L_{\text{loc}}^{\infty}(\R_+, \R_+^*) $ by Assumptions \ref{assump:gamma_sup} and \ref{assump:compat}. }
\end{proof}

We now move on to the proof of Theorem \ref{thm:oracle_point}, the pointwise oracle inequality. The proof follows the steps of what is done in \cite{recurrent_event_Bouaziz} for the ratio kernel estimator and classical symmetric kernels on a bounded support. 
\begin{proof}[Proof of Theorem \ref{thm:oracle_point}]
In this proof, $C_1(t)$ and $C_2(t)$ denote positive $t-$dependent constants for which the value can change from line to line. 

\textbf{Step 1} We start by decomposing $(\Check{k}(t) - k(t))^2$, where $\check{k}(t)$ is defined by \eqref{eq:v}. For all $b \in \mathcal{B}_m$, we have
\begin{align*}
    (\Check{k}(t) - k(t))^2 &\leq 3(\Check{k}(t) - \hat{k}_{b\vee \hat{b}}(t))^2 +3(\hat{k}_b(t)-\hat{k}_{b\vee \hat{b}}(t))^2+ 3(\hat{k}_b(t) - k(t))^2\\
    & \leq  3V_0(\hat{b}(t), t) + 3A_0(b,t)+ 3V_0(b,t) + 3A_0(\hat{b},t) +3(\hat{k}_b(t) - k(t))^2\\
    & \leq 6A_0(b,t) + 6V_0(b,t) + 3(\hat{k}_b(t) - k(t))^2
\end{align*}
Taking the expectation in the previous inequality, the only unknown term is $\E[A_0(b,t)]$ as $V_0$ is deterministic and an equivalent expression of $\E[(\hat{k}_b(t) - k(t))^2]$ is given by Theorem \ref{thm:MSE_k}. 

We start with the following decomposition,
\begin{align*}
    \E[A_0&(b,t)] = \E\Big[\sup_{b' \in \mathcal{B}_m}\Big \{ (\hat{k}_{b'}(t) - \hat{k}_{b'\vee b}(t))^2 - V_0(b',t) \Big \}_+\Big]\\
    & \leq 5\E[\sup_{b' \in \mathcal{B}_m}(\hat{k}_{b'}(t) - \tilde{k}_{b'}(t))^2] +  5\E\Big[\sup_{b' \in \mathcal{B}_m}\Big \{ (\tilde{k}_{b'}(t) - \E[ \tilde{k}_{b'}(t)])^2 - V_0(b',t)/10 \Big \}_+\Big] \\
     & + 5\sup_{b' \in \mathcal{B}_m}\Big \{ (\E[ \tilde{k}_{b'}(t)] -\E[ \tilde{k}_{b' \vee b}(t)])^2 \Big \}+5\E\Big[\sup_{b' \in \mathcal{B}_m}\Big \{ (\tilde{k}_{b\vee b'}(t) - \E[\tilde{k}_{b\vee b'}(t)])^2 -V_0(b',t)/10 \Big\}_+\Big] \\
    &+  5\E\Big[\sup_{b' \in \mathcal{B}_m}\Big \{ (\tilde{k}_{b\vee b'}(t) - \hat{k}_{b\vee b'}(t))^2 \Big\}\Big] \\
     & \leq 10 \E[\sup_{b' \in \mathcal{B}_m}\Big \{ (\tilde{k}_{ b'}(t) - \hat{k}_{b'}(t))^2 \Big\}] +  10\E\Big[\sup_{b' \in \mathcal{B}_m}\Big \{ (\tilde{k}_{b'}(t) - \E[ \tilde{k}_{b'}(t)])^2 - V_0(b',t)/10 \Big \}_+\Big] \\
   & +  5\sup_{b' \in \mathcal{B}_m}\Big \{ (\E[ \tilde{k}_{b'}(t)] -\E[ \tilde{k}_{b' \vee b}(t)])^2 \Big \}.
   \end{align*}

\textbf{Step 2} For $T_1 :=  10 \E[\sup_{b' \in \mathcal{B}_m}\Big \{ (\tilde{k}_{ b'}(t) - \hat{k}_{b'}(t))^2 \Big\}] $, we proceed by decomposing the expectation along $\Omega_{t,c_0}$ as defined in Lemma \ref{lemma:emp_distrib} for some $c_0 \geq \sqrt{13/2}$. 
We have by Lemmas \ref{lemma:event_1} and \ref{lemma:event_2},
\begin{align}
   T_1 \leq C \frac{\log(m)}{m}(A_1(t) + A_2(t) S(\mathcal{B}_m)). \label{eq:T1}
\end{align}
for some constant $C$. 

\textbf{Step 3} 
We move on to controlling 
\begin{align*}
    T_2 := 5\sup_{b' \in \mathcal{B}_m}\Big \{ (\E[ \tilde{k}_{b'}(t)] -\E[ \tilde{k}_{b' \vee b}(t)])^2 \Big \} &= 5\sup_{b' \leq b}\Big \{ (\E[ \tilde{k}_{b'}(t)] -\E[ \tilde{k}_{b}(t)])^2 \Big \}.
\end{align*}
\blue{We have $\E[\tilde{k}_b(t)] = \E[k(Z_{t,b})]$. Thus, by equation \eqref{eq:control_A} in the proof of Proposition \ref{prop:inequ}, and using equation \eqref{assump:gamma_1}, we have,
\begin{align*}
    (\E[ \tilde{k}_{b'}(t)] -\E[ \tilde{k}_{b}(t)])^2 \leq 2(\E[ \tilde{k}_{b'}(t)]- k(t))^2 + 2 (k(t)-\E[ \tilde{k}_{b}(t)])^2 \leq 2 C(t) (b^{2\gamma \beta} + b'^{2\gamma \beta}).
\end{align*}
And thus,
\begin{align}
    T_2 &\leq 10 C(t) (b^{ 2\beta \gamma}+ \sup_{b'\leq b}{b'^{2\beta \gamma}}) \leq 20 C(t) b^{2\beta \gamma}, \label{eq:T2}
\end{align}}
where $C(t) = \sum_{n=1}^{l} \frac{|k^{(n)}(t)|}{n!}C_n(t) + \frac{L}{l!} C_{\beta}(t)$.
\textbf{Step 4} 
Let us move on to $T_3 := 10\E\Big[\sup_{b' \in \mathcal{B}_m}\Big \{ (\tilde{k}_{b'}(t) - \E[ \tilde{k}_{b'}(t)])^2 - V_0(b',t)/10 \Big \}_+\Big] $.
We want to apply Bernstein's inequality to $S_m = m\Tilde{k}_b(t) = \sum_{i=1}^m \frac{\kappa_{t,b}(\tau_i)}{1-F(\tau_i)}$. We rely on similar arguments as in the proof of Theorem 2 in \cite{recurrent_event_Bouaziz}.
Bernstein's inequality applied to $S_m$ (see \cite{Massart2007} p.26) reads as follows 
\begin{equation*}
    \forall x \in \R_+, \mathbb{P}(|S_m - \E[S_m]|^2 \geq x) \leq 2\exp\Big(-\frac{x}{2m(w(b) + h(b)\sqrt{x}/3)}\Big), 
\end{equation*}

for all $w(b)$ and $h(b)$ verifying $\frac{\kappa_{t,b}(\tau_i)}{1-F(\tau_i)} \leq h(b)$ and $\text{Var}(\frac{\kappa_{t,b}(\tau_i)}{1-F(\tau_i)}) \leq w(b)$. Here, we take:
\begin{align}
    &h(b) := G(t) + \frac{C_s(t)}{b^{\gamma}(1-F(t+\lambda))}, \quad \text{ and } w(b) := G(t)^2 + \frac{\kinf C_s(t)}{b^{\gamma}(1-F(t+\lambda))},  \label{eq:h_w(b)}
\end{align}
which verify the above conditions by  Assumptions \ref{assump:gamma_sup} and \ref{assump_ad_strong} and Remark \ref{rem:alpha_beta_bound}.

Using the inequalities $1/(a+b) \geq \min(1/2a, 1/2b)$  and  $\sqrt{a+b} \geq (\sqrt{a} + \sqrt{b})/\sqrt{2}$,  the Bernstein inequality can be rewritten as
\begin{align*}
    \mathbb{P}(|\Tilde{k}_b(t) - \E[ \tilde{k}_b(t)]|^2 \geq V_0(b,t)/10 + x)&\leq 2\exp\Big(-m\frac{V_0(b,t)/10 + x}{2(w(b) + h(b)\sqrt{V_0(b,t)/10 + x}/3) }\Big) \\
    & \leq  2 \max\Big(\exp(-m\frac{V_0(b,t) + 10x}{40w(b)}), \exp(-m\frac{3\sqrt{V_0(b,t)} + 3\sqrt{10x}}{4\sqrt{20}h(b)})\Big).
\end{align*}

Let us compute a bound for the previous equation. 

First, recall that $V_0(b,t) = \frac{\kappa_0 \log(m)}{mb^{\gamma}}e^{\kinf (t+\lambda)} ||k||_{\infty} C_s(t)$. 
\blue{Since $1-F(t+\lambda) = e^{-\int_0^{t+\lambda} k(u) \dd u} \geq e^{-\kinf (t+\lambda)}$, we have  $(1-F(t+\lambda))e^{\kinf  (t+\lambda)} \geq 1$.} Furthermore, as $\kappa_0\geq 80$ and $b^\gamma \leq B(t)/\log(m)$:
\begin{align}
    \frac{mV_0(b,t)}{40w(b)}
    & \geq \log(m) \frac{\kappa_0}{40} \Big(1- \frac{(1-F(t+\lambda)) b^{\gamma} G(t)^2}{ C_s(t) \kinf + b^{\gamma}G(t)^2 (1-F(t+\lambda)) }\Big)
  \nonumber \\
   &\geq 2\log(m) -  2\frac{B(t)(1-F(t+\lambda)) G(t)^2}{C_s(t) \kinf}.\label{eq:bound_V01}
\end{align}
Furthermore, since $b^{\gamma} \geq \kappa_1 \frac{\log(m)}{m}$ and $b^{\gamma}  \leq 1$, 
\begin{align*}
    \frac{3m\sqrt{V_0(b,t)}}{4\sqrt{20}h(b)} 
    & \geq  3\sqrt{m\log(m)b^{\gamma}}\frac{\sqrt{\kappa_0}}{4\sqrt{20}}\frac{\sqrt{ e^{\kinf (t+\lambda)} \kinf C_s(t)}}{G(t)b^{\gamma} + \frac{C_s(t)}{(1-F(t+\lambda))}}\\
    & \geq \frac{3\sqrt{80}}{4\sqrt{20}}\log(m) \frac{\sqrt{\kappa_1 \kinf C_s(t)}}{G(t) + \frac{C_s(t)}{1-F(t+\lambda)}}\\
    & \geq \frac{3}{2}\log(m) \frac{\sqrt{\kappa_1 \kinf C_s(t)}}{G(t) + C_s(t)e^{\kinf(t+\lambda)}}. 
\end{align*}
Thus, as by assumption   $\kappa_1 \kinf C_s(t) = \frac{16}{9} (G(t) + C_s(t)e^{\kinf(t+\lambda)})^2$
we have 
\begin{equation}
     \frac{m\sqrt{V_0(b,t)}}{40h(b)} \geq 2\log(m).\label{eq:low_bound_V0}
\end{equation}

Finally, using the lower bounds on $V_0(b,t)$ provided by equations \eqref{eq:bound_V01} and \eqref{eq:low_bound_V0}, we have 
\begin{align*}
    \mathbb{P}(|\Tilde{k}_b(t) - \E[ \tilde{k}_b(t)]| \geq \sqrt{V_0(b,t)/10 + x})& \leq 2m^{-2}\max\Big(e^{-mx/4w(b) + \frac{2B(t)G(t)^2}{\kinf C_s(t)}} , e^{-3m\sqrt{x}/4\sqrt{2}h(b)}\Big). 
\end{align*}
By integrating the previous inequality, using the expressions of $h(b)$ and $w(b)$ in equations \eqref{eq:h_w(b)}, as well as the assumption that $\forall b \in \mathcal{B}_m, b^{\gamma} \geq \frac{1}{m}$, we obtain
\begin{align*}
    \E[\{|\Tilde{k}_b(t) - \E[ \tilde{k}_b(t)]|^2 -V_0(b,t)/10 \}_+] &\leq 2m^{-2}\max\Big(e^{\frac{2B(t)G(t)^2}{\kinf C_s(t)}}\frac{4w(b)}{m}, \frac{64h(b)^2}{9m^2} \Big) \\
    & \leq m^{-2} C \max\Big(\frac{1}{mb^{\gamma}}, \frac{1}{m^2b^{2\gamma}} \Big)\leq C m^{-2}
\end{align*}
for some constant $C$. This inequality yields finally,
\begin{align*}
    T_3 &\leq 10 \sum_{b' \in \mathcal{B}_m}  \E[\{|\Tilde{k}_{b'}(t) - \E[ \tilde{k}_{b'}(t)]|^2 -V_0(b',t)/10 \}_+]  \leq  C \text{Card}(\mathcal{B}_m)m^{-2} \leq  C m^{-1}.
\end{align*} 

Finally, putting all of the bounds together, this yields equation \eqref{eq:oracle_point_1}. 
\end{proof}
\subsection{Global adaptive bandwidth selection}
\label{sec:global_minimax}

In this section, we present the global adaptive bandwidth selection procedure. 
Consider a finite interval $I = [T_1, T_2] \subset \St$. We will write $||f|| = (\int_I f(x)^2 \dd x)^{1/2}$ and we denote $\kinfI : = \sup_{t \in I} \Big(\sup_{|y-t| \leq \lambda} (k(y))\Big)$. As in the pointwise case, we introduce
\begin{align}
V(b) = \frac{\kappa_2}{mb^{\gamma} }\int_I G(t)^2 + \kinfI C_s(t)e^{\kinfI (t+\lambda)} \dd t, \label{eq:v}
\end{align}
with $\kappa_2$ a strictly positive numerical constant, and 
\begin{align*}
     &A(b) = \sup_{b' \in \mathcal{B}_m}\Big \{ ||\hat{k}_{b'} - \hat{k}_{b'\vee b}||^2 - V(b') \Big \}_+\\
    &\hat{b} = \text{argmin}_{b' \in \mathcal{B}_m}(A(b') + V(b') )\\
    &\check{k} = \hat{k}_{\hat{b}}.
\end{align*}

We introduce the following assumption,
\begin{assumption}\label{assump:intt}
    For any interval $I \subset \St$, there exists $b_0 >0$ such that for any $b \leq b_0$,
    \begin{equation*}
       \exists R_1 >0, \forall y \geq 0, \int_I \kappa_{t,b}(y) \dd t \leq R_1, 
    \end{equation*}
        \begin{equation*}
       \exists R_2>0, \eta \geq 0, \forall y \geq 0, \int_I \kappa_{t,b}(y)^2 \dd t \leq R_2\frac{1}{b^{\gamma(1+\eta)}}. 
    \end{equation*}
\end{assumption}

\begin{remark}
        In the case of classical symmetric kernels, the roles of $t$ and $y$ are interchangeable, making Assumption \ref{assump:intt} redundant with the assumption on the integral of the kernel over $y$. In that case, one therefore has $\eta = 0$, but this is not generally verified by associated kernels. For example, $\eta = 1$ for the Gamma kernel. Assumption \ref{assump:intt} is similar to Assumption $(A^{\alpha}_2)$ in \cite{esstafa:hal-04112846}, in the case of density estimation. 
\end{remark}

The proof of the following proposition can be found in the Appendix (Section \ref{sec:proof_2.2}). 

\begin{prop}\label{prop:gamma_2_global}
    The Gamma kernel without interior bias defined by \ref{def:gamma} verifies Assumption \ref{assump:intt} with $\gamma = 1/2$ and $\eta = 1$.
\end{prop}

Then we have 
\begin{thm}[Global adaptive bandwith estimation]\label{thm:oracle_global}
Let $\hat{k}_b(t)$ be defined by \eqref{eq:k_estim} with a kernel verifying Definition \ref{def:cont_ass}. \blue{Suppose $k \in \Sigma(\beta,l)$ and $(\kappa_{t,b})$ is an associated kernel of order $(\beta,\gamma)$, with $\tilde{\beta} \geq \beta$,  verifying  Assumptions \ref{ass:order}, \ref{assump:gamma_sup}, \ref{assump:unif_int} and \ref{assump_ad_strong} for all $t \in I$, as well as \ref{assump:intt}. 
Consider a finite set of bandwidths $\mathcal{B}_m$ such that $\text{Card}(\mathcal{B}_m) \leq m$ and}
\begin{equation}
\label{hyp:bglobal}
     \forall b \in \mathcal{B}_m, \;  1  \leq mb^{\gamma(1+\eta)}  \text{ and }  b^\gamma \leq \min(B/\log(m),b_0^{\gamma}),
\end{equation}
with $B = \inf_I(B(t))$ as defined in Assumption \ref{assump_ad_strong}, which we suppose strictly positive and $\eta$ defined in Assumption \ref{assump:intt}. \\
Suppose also that $\exists \; A >0$ such that for any constant $C>0$,
\begin{equation}
     \sum_{b\in \mathcal{B}_m} e^{-\frac{C}{\sqrt{b^{\gamma}}}}   \leq A \log(m). \label{eq:assump_sum_exp}
\end{equation}
Then, provided $\kappa_2 \geq 20$, we have $ \forall \; b \in \mathcal{B}_m,$
\begin{align}\label{eq:oracle_int_2}
\E[||\Check{k} - k||^2]&\leq 3\E[||\hat{k}_b - k||^2] +\blue{\Tilde{C}_0b^{2 \beta \gamma}} + 6 V(b) +  \frac{\log(m)}{m}(\Tilde{C}_1 + \Tilde{C}_2S(\mathcal{B}_m)), 
\end{align}
for some nonnegative constants $\Tilde{C}_0, \Tilde{C}_1$ and $\Tilde{C}_2$ and with $S(\mathcal{B}_m) = \sum_{b \in \mathcal{B}_m} \frac{1}{mb^{\gamma}}$.\end{thm}
\blue{Similarly as what was noted for Theorem \ref{thm:oracle_point}, the first term in  \eqref{eq:oracle_int_2} is the MISE of the estimator, which we can control by Theorem \ref{thm:MISE_k}. The second term is of same order as the square bias, and by definition of $V(b)$, the third term has the same order as the variance of the estimator. The term in $\frac{\log(m)}{m}$ is a penalization term. Thus, under the assumptions of Theorem \ref{thm:MISE_k}, \eqref{eq:oracle_int_2} is of order $b^{2\beta \gamma} + \frac{1}{m b^{\gamma}} + \frac{\log(m)}{m}\mathcal{S}(\mathcal{B}_m)$. Thus, if the bandwidth set contains a bandwidth of order $m^{-\frac{1}{\gamma(2\beta +1)}}$ and if $\mathcal{S}(\mathcal{B}_m)$ is of order at most $\frac{m^{1/(2 \beta + 1)}}{\log(m)}$, then the optimal minimax rate of convergence of $m^{\frac{-2\beta}{2\beta + 1}}$ is achieved by the global adaptive bandwidth choice procedure.}
\begin{remark}
\begin{itemize}
\item Unlike for classical kernels, the non explicit dependence of the kernel in $t$ and $b$ prevents from proving a result on $\R_+$, mainly due to Assumption \ref{assump_ad_strong} where $\inf_{\R_+}B(t) >0$ is in general not true. However, as estimations are conducted on an interval in practice, it is not a major drawback.
\item Assumption \eqref{eq:assump_sum_exp} can be understood as an assumption on the distribution of bandwidths inside the bandwidth set, to ensure that they are not all too close to the upper bound, in which case the sum would be larger than $O(\log(m))$. In \cite{Doumic_2012}, no assumption is made on the sum but the bandwidth set is assumed to be a subset of $\{1/i, i=1,..., \delta m\}$ for some positive $\delta$, which ensures the convergence of the sum, and thus that is stays controlled by $\log(m)$. 
\end{itemize}
\end{remark}

We will now prove the global oracle inequality \ref{thm:oracle_global}. The proof follows partly what is done in \cite{recurrent_event_Bouaziz} and \cite{Doumic_2012}. It is quite similar to the proof of Theorem \ref{thm:oracle_point} for the most part, as the bounds can be uniformly integrated over the segment $I$. Only the concentration inequality differs and allows to get a sharper bound without the $\log(m)$ in factor. In the following, we recall that $||.||$ denotes the $L^2$ norm on $I$. In order to apply Talagrand's inequality, we also recall that for any function $f \in L^2(I)$, and if $\mathcal{B}$ denotes the unit ball in $L^2(I)$ and $\mathcal{A}$ is a dense countable subset of $\mathcal{B}$, we have the following representation of the $L^2$ norm, 
\begin{align*}
    ||f|| &= \sup_{a \in \mathcal{B}}\int_I a(t) f(t) \dd t  = \sup_{a \in \mathcal{A}}\int_I a(t) f(t) \dd t. 
\end{align*}

We begin by proving some bounds that will be useful to apply Talagrand's inequality.
\blue{For ease of notations, we consider in the following  that $b_0 \leq 1$. }

\begin{lemma}\label{lemma:comput}
For any $b \in \mathcal{B}_m$ and $t \in \T$ we denote
\begin{align*}
     \xi_{b}(t) :=\Tilde{k}_{b}(t)- \E[ \tilde{k}_b(t)] =  \sum_{i=1}^m (\zeta_{t,b}(\tau_i) - \E[\zeta_{t,b}(\tau_i)])  \qquad \text{with} \qquad   \zeta_{t,b}(\tau_i) = \frac{1}{m}\frac{\kappa_{t,b}(\tau_i)}{1-F(\tau_i)}.
\end{align*}
Under Assumptions \ref{assump:gamma_sup}, \ref{assump_ad_strong} and \ref{assump:intt}, there exist $C_e,C_v, C_h \geq 0$ such that
    \begin{align*}
    &\E[||\xi_{b}||]  \leq \frac{1}{\sqrt{mb^{\gamma}}}C_e,\\
    &v := m\sup_{a \in \mathcal{A}}Var\Bigg[\int_I a(t) \zeta_{t,b}(\tau_1)   \dd t \Bigg] \leq \frac{1}{m} C_v,\\
    & h :=\sup_{y \in \St}||\zeta_{t,b}(y) - \E[\zeta_{t,b}(\tau_1)] || \leq \frac{C_h}{m\sqrt{b^{\gamma(1+\eta)}}},
\end{align*}
where $\eta$ is defined in Assumption \ref{assump:intt}.
\end{lemma}
\begin{proof}
\textbf{Step 1} By applying Jensen's inequality, using the independence of the random variables $(\zeta_{t,b}(\tau_i))$  we have 
\begin{align*}
     \E[||\xi_{b}||]
     & \leq \Big( \int_I \sum_{i=1}^m \E\big[ (\zeta_{t,b}(\tau_i) - \E[\zeta_{t,b}(\tau_i)])^2  \big] \dd t \Big)^{1/2}\\
     &\leq \sqrt{m}\Big(\int_I \E[\zeta_{t,b}(\tau_1)^2] \dd t \Big)^{1/2} = \frac{1}{\sqrt{m}} \Big(\int_I \int_{\mathbb S} \frac{\kappa_{t,b}(y)^2}{(1-F(y))^2}f(y)  \dd y \dd t \Big)^{1/2}.
\end{align*}
By Assumption \ref{assump_ad_strong} and Remark \ref{rem:alpha_beta_bound}, we obtain
\begin{align*}
      \E[||\xi_{b}||]  & \leq \frac{1}{\sqrt{m}} \Big(\int_I G(t)^2 + \frac{\kinfI C_s(t)}{b^{\gamma}(1-F(t+\lambda))} \dd t \Big)^{1/2} \\
      &\leq \frac{1}{\sqrt{mb^{\gamma}}} \Big(\int_I G(t)^2 + \frac{\kinfI C_s(t)}{(1-F(t+\lambda))} \dd t \Big)^{1/2} :=  \frac{1}{\sqrt{mb^{\gamma}}}C_e,
\end{align*}
where we used the fact that $b^{\gamma} \leq 1$.\\
\textbf{Step 2}  By the Cauchy Schwarz inequality and Assumption \ref{assump_ad_strong}, we have since $I = [T_1, T_2]$:
    \begin{align*}
    v &\leq \frac{1}{m} \sup_{a \in \mathcal{A}} \E\Big[\int_I \frac{\kappa_{t,b}(\tau_1)}{1-F(\tau_1)}\dd t \int_I a(t)^2\frac{\kappa_{t,b}(\tau_1)}{1-F(\tau_1)}\dd t\Big]\\
    & \leq \frac{1}{m}\sup_{a \in \mathcal{A}} \E\Big[\Big(\int_IG(t) \dd t + \frac{1}{1-F(T_2+\lambda)} \int_I \kappa_{t,b}(\tau_1) \dd t  \Big) \int_I a(t)^2\frac{\kappa_{t,b}(\tau_1)}{1-F(\tau_1)}\dd t\Big]
    \end{align*}
Furthermore,  $\int_I \kappa_{t,b}(\tau_1) \dd t \leq R_1$ by Assumption \ref{assump:intt},  and thus
    \begin{align*}
   v  & \leq \frac{1}{m} \Big(\int_I G(t) \dd t + \frac{1}{1-F(T_2+\lambda)} R_1 \Big)  \sup_{a \in \mathcal{A}} \E\Big[\int_I a(t)^2\frac{\kappa_{t,b}(\tau_1)}{1-F(\tau_1)}\dd t\Big]\\
    & \leq \frac{1}{m} \Big(\int_I G(t) \dd t + \frac{1}{1-F(T_2+\lambda)} R_1\Big)  \sup_{t\in I} \E\Big[\frac{\kappa_{t,b}(\tau_1)}{1-F(\tau_1)}\Big]\\
    & \leq \frac{1}{m} \Big(\int_I G(t) \dd t + \frac{1}{1-F(T_2+\lambda)} R_1\Big) (\sup_{t\in I} (G(t)) + \kinfI) := \frac{1}{m} C_v.
\end{align*}
\textbf{Step 3} We have 
\begin{align*}
    h &= \sup_{y \in \St}||\zeta_{t,b}(y) - \E[\zeta_{t,b}(\tau_1)] || \leq \sup_{y \in \St}||\zeta_{t,b}(y)|| + ||\E[\zeta_{t,b}(\tau_1)] ||.
\end{align*}
By step 1, $||\E[\zeta_{t,b}(\tau_1)] || \leq \frac{C}{m\sqrt{b^\gamma}}$. Finally,  Assumptions \ref{assump_ad_strong} and \ref{assump:intt} yield that,
\begin{align*}
   h & \leq  \frac{1}{m}\sup_{y \in \St}\Big(\int_I G(t)^2 \dd t + \frac{1}{(1-F(T_2 + \lambda))^2}\int_I \kappa_{t,b}^2(y) \dd t\Big)^{1/2} + \frac{C}{m\sqrt{b^\gamma}}\\ 
    &  \leq \frac{1}{m} \Big(\int_I G(t)^2 \dd t + \frac{1}{b^{\gamma(1+\eta)}(1-F(T_2 + \lambda))^2} R_2 \Big)^{1/2} +  \frac{C}{m\sqrt{b^\gamma}}\\
    & \leq \frac{C_h}{m\sqrt{b^{\gamma(1+\eta)}}},
\end{align*}
 with $R_2$ and $\eta$ such that $\int_I \kappa_{t,b}(t)^2 \dd t \leq R_2 b^{-\gamma(1+\eta)}$.
\end{proof}

We now move on to proof of the global oracle inequality. 

\begin{proof}[Proof of Theorem \ref{thm:oracle_global}]
    We proceed similarly as in the proof of Theorem \ref{thm:oracle_point}. 
   We have 
    \begin{align*}
    ||\Check{k} - k||^2 \leq 6A(b) + 6V(b) + 3||\hat{k}_b - k||^2
\end{align*}
and
    \begin{align*}
    \E[A(b)]  \leq &10\E\Big[\sup_{b' \in \mathcal{B}_m}\Big \{ ||\tilde{k}_{b'} - \E[ \tilde{k}_{b'}]||^2 - V(b')/10 \Big \}_+\Big] + 10\E[\sup_{b' \in \mathcal{B}_m}||\tilde{k}_{b'} - \hat{k}_{b'}||^2] \\
    &+ 5\sup_{b' \in \mathcal{B}_m}\Big \{ ||\E[ \tilde{k}_{b'}] -\E[ \tilde{k}_{b' \vee b}]||^2 \Big \}.
\end{align*}

\blue{ For $U_1 :=  10\E[\sup_{b' \in \mathcal{B}_m}||\tilde{k}_{b'} - \hat{k}_{b'}||^2]$ and $U_2 :=  5\sup_{b' \in \mathcal{B}_m}\Big \{ ||\E[ \tilde{k}_{b'}] -\E[ \tilde{k}_{b' \vee b}]||^2 \Big \} $, we have by Jensen and Fubini, 
\begin{align*}
   U_1 \leq  10\E[\sup_{b' \in \mathcal{B}_m}\int_I(\tilde{k}_{b'}(t) - \hat{k}_{b'}(t))^2 \dd t]\leq 10 \int_I \E[\sup_{b' \in \mathcal{B}_m}(\tilde{k}_{b'}(t) - \hat{k}_{b'}(t))^2 ] \dd t = 10\int_I T_1 \dd t.
\end{align*}
And a similar result holds for $U_2$.As the constants involved in \eqref{eq:T1} and \eqref{eq:T2} are locally bounded in $t$ by Lemmas \ref{lemma:event_1} and \ref{lemma:event_2}, we can integrate \eqref{eq:T1} and \eqref{eq:T2}, which yields
\begin{align*}
    &U_1 \leq \frac{\log(m)}{m} (C + C' \mathcal{S}(\mathcal{B}_m) )\qquad \text{ and } \qquad 
   U_2 \leq C b^{2\beta\gamma} .
\end{align*}}

Let us now turn to $U_3 := 10\E\Big[\sup_{b' \in \mathcal{B}_m}\Big \{ ||\tilde{k}_{b'} - \E[ \tilde{k}_{b'}]||^2 - V(b')/10 \Big \}_+\Big] = 10\E\Big[\sup_{b' \in \mathcal{B}_m}\Big \{ ||\xi_{b'}||^2 - V(b')/10 \Big \}_+\Big]$,  using the notations introduced in Lemma \ref{lemma:comput}. We will use a similar proof strategy as what is used in Lemma 1 in \cite{Doumic_2012}. We start by noticing that for all $M_b>0$,
\begin{align}
\label{eq:U3}
    \E[\sup_{b\in \mathcal{B}_m} \!\{||\xi_b||^2\! - \!M^2_b\}_+]\! &\leq\! \sum_{b \in \mathcal{B}_m} \int_{\R_+}\mathbb{P}(||\xi_b||\! \geq \!\sqrt{ M^2_b \!+ \!y} ) \dd y  \leq \sum_{b \in \mathcal{B}_m} \int_{\R_+}\mathbb{P}(||\xi_b||\! \geq \!\frac{1}{\sqrt{2}}(M_b\! +\! \sqrt{y} )) \dd y.
\end{align}
Using the same notations as in Lemma \ref{lemma:comput},
\begin{align*}
    ||\xi_{b}|| &= \sup_{a \in \mathcal{A}}\int_I a(t) \xi_{b}(t) \dd t  = \sup_{a \in \mathcal{A}}\sum_{i=1}^m \int_I a(t) (\zeta_{t,b}(\tau_i) - \E[\zeta_{t,b}(\tau_i)])  \dd t.
\end{align*}
This expression of the $L^2$ norm allows us to apply Talagrand's inequality (see \cite{Massart2007} p.170). For all $\epsilon, x >0$, 
\begin{equation*}
    \mathbb{P}(||\xi_{b}|| \geq (1 + \epsilon) \E[||\xi_{b}||] + \sqrt{2vx} + c(\epsilon)hx) \leq e^{-x},
\end{equation*}
where $c(\epsilon) = \frac{1}{3} + \epsilon^{-1}$ and $v$ and $h$ are defined in Lemma \ref{lemma:comput}.
The bounds for $\E[||\xi_b||], v, h$  obtained in Lemma \ref{lemma:comput} yield that 
\begin{equation*}
     \mathbb{P}\Big(||\xi_{b}|| \geq (1 + \epsilon)\frac{C_e}{\sqrt{mb^{\gamma}}} +\frac{\sqrt{2C_vx}}{\sqrt{m}} + c(\epsilon)\frac{C_h}{m\sqrt{b^{\gamma(1+\eta)}}}x\Big) \leq e^{-x}.
\end{equation*}
Furthermore, for some $L_b >0$ to be determined later, and by setting $x = u + L_b$, the previous inequality can be rewritten as
\begin{equation}\label{eq:talagrandbis}
     \mathbb{P}\Big(||\xi_{b}|| \geq C_b +\frac{\sqrt{2C_vu}}{\sqrt{m}} + c(\epsilon)\frac{C_hu}{m\sqrt{b^{\gamma(1+\eta)}}}\Big) \leq e^{-u}e^{-L_b}, 
\end{equation}
with  $C_b =(1 + \epsilon)\frac{C_e}{\sqrt{mb^{\gamma}}} +\frac{\sqrt{2C_vL_b} }{\sqrt{m}} + c(\epsilon)\frac{C_h}{m\sqrt{b^{\gamma(1+\eta)}}}L_b$.\\
Taking $M_b=\sqrt{2}C_b$ in \eqref{eq:U3}, and using the change of variables $y = 2(\frac{\sqrt{2C_v u} }{\sqrt{m}} + c(\epsilon)\frac{C_h}{m\sqrt{b^{\gamma(1+\eta)}}}u)^2$, we obtain by \eqref{eq:talagrandbis} that there is a constant $C_{\epsilon}$ such that:
\begin{align*}
     \E[\sup_{b\in \mathcal{B}_m} \{||\xi_b||^2 - M_b^2\}_+]  &\leq \sum_{b \in \mathcal{B}_m} \int_0^{+\infty} e^{-L_b - u}4\Big(\frac{\sqrt{2C_v u} }{\sqrt{m}} + c(\epsilon)\frac{C_h}{m\sqrt{b^{\gamma(1+\eta)}}}u\Big)\times \Big(\frac{\sqrt{C_v} }{\sqrt{2mu}} + c(\epsilon)\frac{C_h}{m\sqrt{b^{\gamma(1+\eta)}}}\Big) \dd u \\
     & \leq \sum_{b \in \mathcal{B}_m} e^{-L_b} \int_0^{+\infty}  e^{- u} u^{-1}4\Big(\frac{\sqrt{2C_v u} }{\sqrt{m}} + c(\epsilon)\frac{C_h}{m\sqrt{b^{\gamma(1+\eta)}}}u\Big)^2 \dd u \\
     &\leq C_{\epsilon} \sum_{b \in \mathcal{B}_m} e^{-L_b} (2m^{-1} C_v  + C_h^2 m^{-2}b^{-(1+\eta)\gamma}),
\end{align*}
where we have used the inequality $(a+b)^2 \leq 2a^2 + 2b^2$. 

For $\theta >0$, we set 
\begin{equation*}
    L_b = \frac{C_e^2\theta^2}{2C_v\sqrt{b^{\gamma}}} = \frac{C_{L,\theta}}{\sqrt{b^{\gamma}}}.
\end{equation*}
By Assumption \eqref{hyp:bglobal},  $mb^{\gamma(1+\eta)}\geq 1 $, and combining this with Assumption \eqref{eq:assump_sum_exp}, we obtain that  
\begin{align*}
     \E[\sup_{b\in \mathcal{B}_m} \{||\xi_b||^2 - M_b^2\}_+]  & \leq \frac{C_{\epsilon}}{m} \sum_{b\in \mathcal{B}_m} e^{-\frac{C_{L,\theta}}{\sqrt{b^{ \gamma}}}} (2C_v + C_h^2 \frac{1}{mb^{(1+\eta)\gamma}})\\
     & \leq (2C_v + C_h^2) \frac{C_{\epsilon}}{m} \sum_{b\in \mathcal{B}_m} e^{-\frac{C_{L,\theta}}{\sqrt{b^{\gamma}}}}  \leq  \frac{C\log(m)}{m}.
\end{align*}
Furthermore, 
\begin{align*}
    M_b  &= \sqrt{2}\Big((1+\epsilon)\frac{C_e}{\sqrt{mb^{\gamma}}} + \frac{C_e \theta }{b^{\gamma/4}\sqrt{m}} + c(\epsilon) \frac{C_h \theta^2C_e^2}{2C_v m \sqrt{b^{\gamma}} \sqrt{b^{(1+\eta)\gamma}}} \Big)\\
    & \leq\frac{\sqrt{2}C_e}{\sqrt{mb^{\gamma}}}  \Big(1+ \epsilon + b^{\gamma/4}\theta+ c(\epsilon)\frac{C_h \theta^2C_e}{2C_v\sqrt{mb^{(1+\eta)\gamma}}}\Big).
\end{align*}

Recall that $mb^{\gamma(1+\eta)} \geq 1 $ and $b^{\gamma} \leq 1$. Hence,  for $\theta$ and $\epsilon$  small enough
we have $M_b \leq  \frac{\sqrt{2}C_e}{\sqrt{mb^{\gamma}}} \sqrt{\frac{\kappa_2}{20}} \leq  \sqrt{V(b)/10}$, \blue{as $V(b) \geq \kappa_2 \frac{C_e^2}{mb^{\gamma}}$ and since $\kappa_2 > 20$. }
Finally, this leads to 
\begin{align*}
    &U_3 = 10 \E[\sup_{b\in \mathcal{B}_m} \{||\xi_b||^2 -V(b)/10\}_+] \leq 10 \E[\sup_{b\in \mathcal{B}_m} \{||\xi_b||^2 - M_b^2\}_+] \leq 10 C\log(m)m^{-1}, 
\end{align*}
which achieves the proof. 

\end{proof}

\section{Hazard rate estimation with the Gamma kernel}
\label{sec:num_simul}
We now present some numerical illustrations of the previous results. The associated kernel we consider is the Gamma kernel without interior bias first introduced in \cite{Gamma_kernel} and defined in \eqref{eq:Gamma_kernel_1}, \eqref{eq:rho_gamma_1}. As mentioned in Section \ref{sec:the_frame}, this kernel is particularly adapted to estimating data on a support bounded by one end (for example, the positive real line). It is notably very asymmetric at the end of the support, thus improving the boundary bias, especially for hazard rates that do not vanish near 0.
As per Propositions \ref{prop:gamma}, \ref{prop:gamma_2} and \ref{prop:gamma_2_global}, the Gamma kernel verifies the assumptions needed for our results with $\gamma = 1/2$. In this section, we provide several illustrations of hazard rate estimations with the Gamma kernel, using different bandwidth choice methods, and compare the results with other kernel estimators. We also provide an example on real data. \blue{ Note that the assumptions on the kernel can be checked numerically instead of analytically for practical purposes. Definition \ref{def:gamma} for example, can be verified by computing the moments for several values of the bandwidth and comparing their behavior to the bandwidth.}

All of the code used to generate the results and figures of this section is available at \url{https://github.com/luce-breuil/non_param_estim_assoc}. The simulated data is generated using the package IBMPopSim \cite{IBMPopSim}.

\subsection{Estimation on simulated data}

Although the theoretical expressions of $V_0$ \eqref{eq:v_0} and $V$ \eqref{eq:v} are $t$-dependent and involve a lot of constants, we will use a much simpler expression in the numerical implementation, similarly to what is done in \cite{recurrent_event_Bouaziz}, by only keeping the same asymptotic order as the theoretical results:
\begin{align*}
    V_0(b) = \frac{\kappa_0 \log(m) \kinf}{m b^{1/2}} \text{ and } \blue{V(b) = \frac{\kappa_2 \kinfI}{m b^{1/2}}}.
\end{align*}
The constants are taken here as $\kappa_0 = 0.03$, $\kappa_2 = 20$ and $\kinf$ and $\kinfI$ are estimated by taking the maximum of the estimated hazard rate \blue{with the smallest bandwidth in the bandwidth set. As mentioned in \cite{recurrent_event_Bouaziz}, estimators of $\kinf$ and $\kinfI$ could be directly considered in the proofs, which would result in a similar inequality as the error done when estimating these quantities is of order $b^{2\gamma\beta} + b^{-\gamma}m^{-1}$.}\\
The choice of an optimal constant is one of the difficulties linked to the implementation of adaptive estimators (see e.g. the discussion in \cite{LACOUR20163774}). Since the theoretical framework of associated kernels involves additional approximations in the upper bounds compared to the classical case, it is not surprising that the resulting constants are also suboptimal, and depend on how sharp the inequalities in the assumptions on the kernels are.
\blue{The constants were chosen by tuning the bandwidth choice method on simulated data from several different hazard rates. In particular, the chosen values of $\kappa_0$ and $\kappa_2$ are much smaller than the theoretical values, and the dependence in $t$ of the variance term is removed. Without doing so, the variance term would systematically impose the choice of largest bandwidth, most particularly for large $t$. Given the broad framework of our study (e.g. the order of the dependence in $b$ of $\kappa_{t,b}$ is asymptotic and not exact as it is for classical kernels), the theoretical bound of $V_0$ and $V$ is unsurprisingly suboptimal. In particular, the exponential dependence in $t$ comes from relatively rough upper bounds and removing it allows $V_0$ and $V$ to stay close to the known theoretical order of the variance. The values chosen result in a bandwidth choice procedure that does not systematically choose the over-smoothing estimator. Note that the same is done in \cite{recurrent_event_Bouaziz}.} \blue{We also follow a strategy similar to that of \cite{LACOUR20163774}, who established a minimal penalty phenomenon for this method for density estimation with standard kernels. We  decreased each constant until the selected bandwidth was systematically the smallest value of $\mathcal{B}_m$ (minimization of  the bias only), and chose a value above this empirical critical point.}
Note that  the empirical penalization functions $V$ and $V_0$ are not directly proportional to the variance of the estimator and are only asymptotically of the same order.\\
The bandwidth sets we consider are as follows: 
\begin{align*}
    &\text{Local } \mathcal{B}_m=\{400(\log(m)/m)^2 \} \cup \{i \log(m)^2/m, 1\leq i \leq 10\log(m), (i-1)\equiv 0 [4]\} \\
    & \qquad \qquad \qquad  \cap \{ b^{1/2} \leq \min(1, 6/\log(m)) \}\\
     &\text{Global } \mathcal{B}_m = \{ i/m^{2/3}, 1 \leq i \leq 10m^{1/2}, i \equiv 0 [10]  \} \cap \{ b^{1/2} \leq \min(1, 6/\log(m))\}.
\end{align*}
\blue{The local bandwidth set is of size at most $10\log(m) + 1$, ensuring that it has size smaller than $m$ and the bandwidths in that set are such that $\min(\frac{\log(m)}{\sqrt{m}}, 20 \frac{\log(m)}{m})\leq b^{1/2} \leq\min(1, \frac{6}{\log(m)} ,\sqrt{10} \frac{\log(m)^3/2}{\sqrt{m}})$. As we tested values of $m \geq 20$, $\gamma = 1/2$ and $b_0 = 1$ in the case of the Gamma kernel (see the proof of Proposition \ref{prop:gamma}), this ensures that $\mathcal{B}_m$ verifies the theoretical hypothesis, up to approximations on the values of $\kappa_1$ and $B(t)$, for which we used a respectively smaller and larger value to allow for a larger range of bandwidths in our set. The set is such that $S(\mathcal{B}_m) = O(1)$. \\
Similarly, the global bandwidth set is of size at most $10m^{1/2}$ and the bandwidths inside verify $\frac{1}{m^{1/3}} \leq b^{1/2} \leq \min(1, \frac{6}{\log(m)}) $. Thus, the bandwidths inside also verify the theoretical assumptions, once again for a larger value of $B$ than the one that we were able to prove analytically. It is such that $\sum_{b\in \mathcal{B}_m} e^{-C/b^{1/4}} \leq me^{-10 Cm^{24}} \leq A \log(m)$. The size of the bandwidth set being one of the main causes for duration of computations, a good way to improve computation time if it is a limitation would be, for example, to restrict the bandwidth set to a geometric progression. }

\paragraph{Comparison of estimators and kernels.} 
We begin by comparing different kernels and estimators on several  hazard rates 
for a sample size of $2000$ observations.
Figure \ref{fig:gamma_ker_comp} shows a comparison of the kernel estimation with several methods on two different  hazard rates. 
The Gamma kernel estimator is compared with the cross-validation bandwidth Gaussian kernel estimator, the local plugin bandwidth choice with Epanechnikov kernel (as described in \cite{bound_ker_hazard}), and the ratio estimator with lognormal kernel as defined in \cite{Salha02092023}. \blue{For the lognormal estimator, the bandwidth was chosen by testing several values and choosing the minimizer of the empirical MISE.}
Firstly, the Gaussian kernel estimator with cross-validation bandwidth shows estimations which are highly biased and underestimated at 0, where it completely fails to capture the magnitude of the hazard, especially for the example presented on Figure \ref{fig:gamma_ker_comp1_f}.\\
The log-normal estimator does capture the peak near $0$, but is very noisy and does not perform as well in the rest of the support. In particular, the log-normal kernel ratio estimator underestimates the hazard rate after $t = 150$ on both figures. This could be improved by choosing a smaller bandwidth, but it would render the estimation near 0 even noisier than it already is. Additionally, as the definition of this estimator involves the ratio of two estimators and integration, it is not as computationally efficient as the others. 
\blue{The local plugin bandwidth Epanechnikov kernel estimator presents an overall good estimation of the second hazard (Figure \ref{fig:gamma_ker_comp1_l})), but significantly oversmoothes the estimation on Figure \ref{fig:gamma_ker_comp1_f} . } \\
The Gamma kernel seems to provide the overall least biased estimation for both hazard shapes on Figures \ref{fig:gamma_ker_comp1_f} and \ref{fig:gamma_ker_comp1_l}. In particular, the use of the local bandwidth choice on Figure \ref{fig:gamma_ker_comp1_l} allows for a good estimation in spite of the high variations in the hazard rate, whereas the global adaptive bandwidth is most adequate for the first hazard on Figure \ref{fig:gamma_ker_comp1_f} to provide a smoother estimation. Overall, the Gamma kernel allows for an estimation which is unbiased near 0, precise inside the support and relatively smooth at the same time. 
\begin{figure}[h!]
\begin{subfigure}{0.98\textwidth}
    \centering
    \includegraphics[scale=0.26]{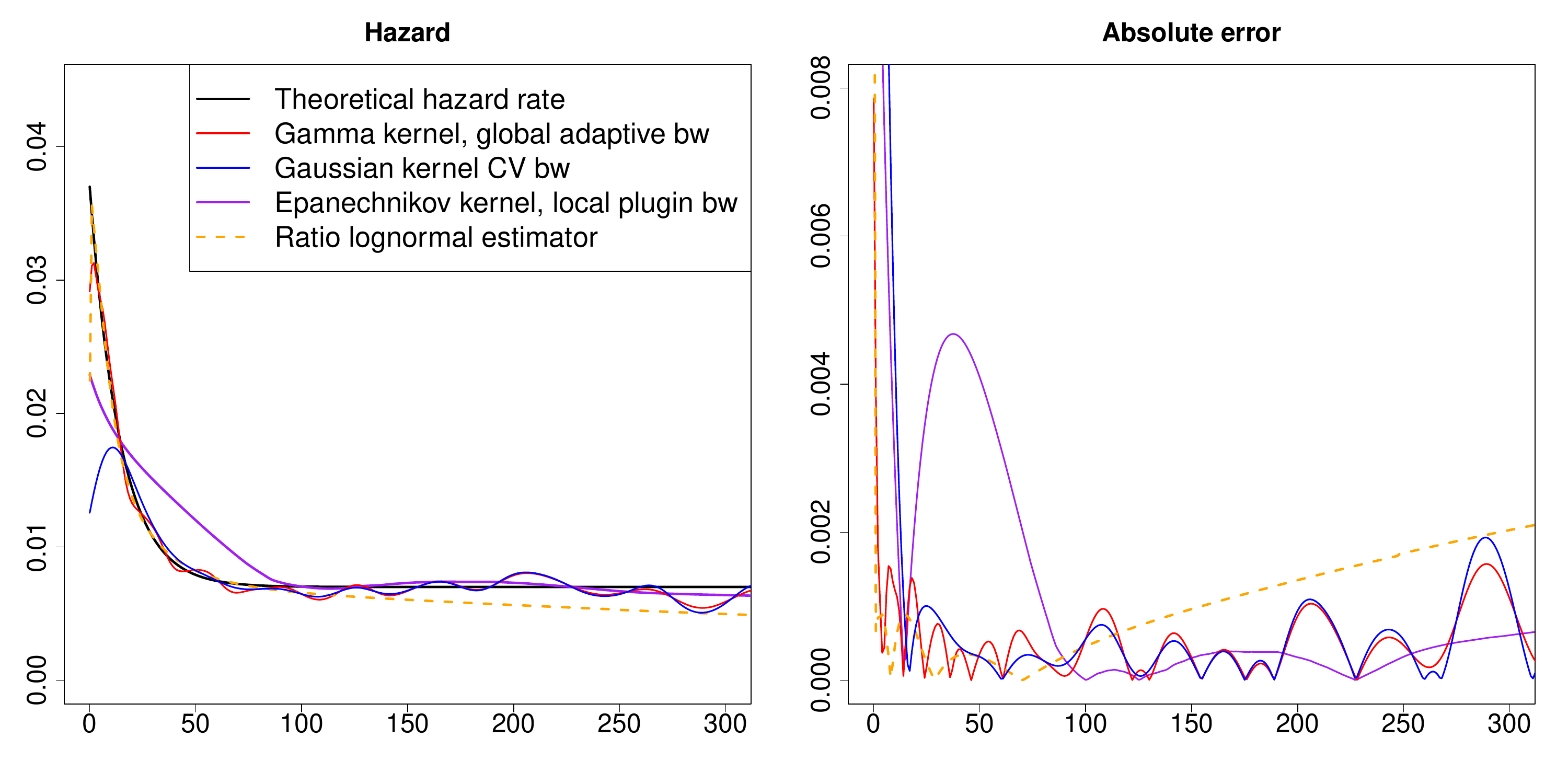}
    \caption{Hazard is  $k(t) = a+c\cdot e^{-dt}$, $a = 7\cdot 10^{-3}, c = 3 \cdot 10^{-2}, d = 7 \cdot 10^{-2}$ Bandwidths - (Gam)adaptive global $0.57$ (Gaus) 10 (LN)$0.5$.}
    \label{fig:gamma_ker_comp1_f}
    \end{subfigure}
    \begin{subfigure}{0.98\textwidth}
    \centering
    \includegraphics[scale=0.26]{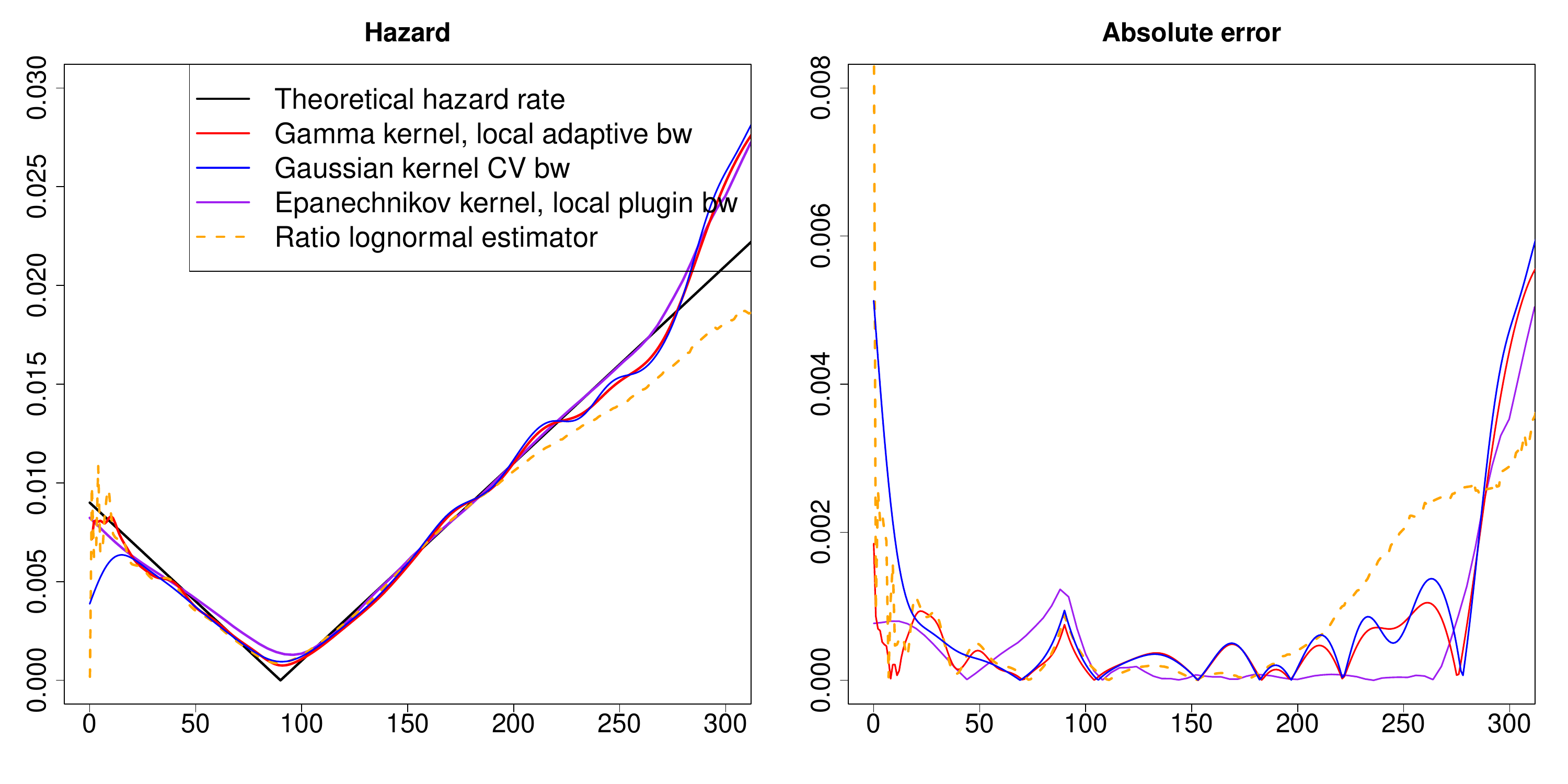}
    \caption{Hazard is $k(t) = b|t/3-30|$ with $b = 3\cdot 10^{-4}$. Bandwidths - (Gam)adaptive local bandwidth (Gaus) $10$ (LN) $0.01$.}
    \label{fig:gamma_ker_comp1_l}
    \end{subfigure}
    
    \caption{\textbf{Comparison of the kernel estimation on two hazard rates}. Estimation methods are Gamma (Gam), Gaussian with cross-validation bandwidth (Gaus), Epanechnikov kernel with local plugin bandwidth and log-normal ratio (LN) for the specified values of bandwidth and a sample size of $2000$.}
    \label{fig:gamma_ker_comp}
\end{figure}

The MISE and MSE at 0 for the hazard of Figure \ref{fig:gamma_ker_comp1_f} are presented in Table \ref{table:MSE_CV_gaussian} for different kernels, and for the Gamma kernel with different bandwidth choice methods in Table \ref{table:MSE_CV_adapt}. 
The error for the estimations with the Gaussian kernel is comparable to the Gamma kernel for the nearest neighbor bandwidth choice (see \cite{k_nearest_choice}) on the entire interval, but 
consistently higher than the Gamma kernel for the cross-validation choice and at $0$ for the nearest neighbor bandwidth. In particular, the error at $0$ for the cross-validation bandwidth and Gaussian kernel does not decrease with an increasing sample size, showing the asymptotically biased nature of the Gaussian kernel estimation. Similarly, the lognormal ratio estimator has a MISE and a MSE at $0$ almost systematically greater than the Gamma kernel estimator. The standard deviation of the MSE at $0$ is high, showing the high instability of the lognormal ratio estimator at $0$. 

 \begin{table}
\centering
\begin{tabular}{ c c c c c c c }
     & \multicolumn{3}{c}{Gaussian kernel} & \makecell{Epanechnikov \\ kernel}&\makecell{Lognormal \\ ratio estimator} & Gamma kernel \\ \hline \hline \noalign{\smallskip}
     &\makecell{CV \\bandwidth}  &\multicolumn{2}{c}{\makecell{Nearest\\ neighbor}} & \makecell{Local \\ Plugin }& Bandwidth $b = 0.5$ & Global adaptive \\ \hline \hline \noalign{\smallskip}
     Size (m) & \makecell{ MISE \\MSE at $0$} &k&\makecell{ MISE \\MSE at $0$}& \makecell{ MISE \\MSE at $0$}& \makecell{ MISE \\MSE at $0$} & \makecell{ MISE \\MSE at $0$} \\ \hline \hline \noalign{\smallskip}
    $500$  & \makecell{$8.30\cdot 10^{-3}$ \\${\scriptstyle (4.78 \cdot 10^{-3})}$\\$8.86\cdot 10^{-4}$\\ ${\scriptstyle (8.92 \cdot 10^{-5})}$}& $25$ & \makecell{$\mathbf{1.55\cdot 10^{-3}}$ \\ ${\scriptstyle (6.91 \cdot 10^{-4})}$\\$3.77 \cdot 10^{-4}$ \\${\scriptstyle (1.19 \cdot 10^{-4})}$} &  \makecell{$4.28\cdot 10^{-3}$ \\ ${\scriptstyle (6.22\cdot 10^{-3})}$\\$2.31\cdot 10^{-4}$ \\${\scriptstyle (7.36 \cdot 10^{-5})}$} &\makecell{$3.09 \cdot 10^{-3}$ \\${\scriptstyle (1.09 \cdot 10^{-3})}$\\$4.82\cdot 10^{-4} $\\ ${\scriptstyle (7.61 \cdot 10^{-4})}$} & \makecell{$3.27\cdot 10^{-3}$ \\${\scriptstyle (2.15 \cdot 10^{-3})}$\\$\mathbf{4.00\cdot 10^{-5}}$\\ ${\scriptstyle (5.13 \cdot 10^{-5})}$} \\\hline \noalign{\smallskip}
    $1000$ & \makecell{$7.52\cdot 10^{-3}$ \\${\scriptstyle (4.04 \cdot 10^{-3})}$\\ $8.72 \cdot 10^{-4}$ \\${\scriptstyle (1.20 \cdot 10^{-4})}$} & $40$& \makecell{$\mathbf{1.04 \cdot 10^{-3}}$  \\${\scriptstyle (3.24 \cdot 10^{-4})}$ \\$3.65 \cdot 10^{-4}$ \\${\scriptstyle (1.19 \cdot 10^{-4})}$}  & \makecell{$2.88 \cdot 10^{-3}$  \\${\scriptstyle (2.05 \cdot 10^{-3})}$ \\$4.93 \cdot 10^{-4}$ \\${\scriptstyle (5.01 \cdot 10^{-5})}$}  & \makecell{$2.74\cdot 10^{-3}$ \\${\scriptstyle (5.76 \cdot 10^{-4})}$\\ $2.37 \cdot 10^{-4}$ \\${\scriptstyle (3.29 \cdot 10^{-4})}$} & \makecell{$1.89\cdot 10^{-3}$ \\${\scriptstyle (1.27 \cdot 10^{-3})}$\\ $\mathbf{2.32 \cdot 10^{-5}}$ \\${\scriptstyle (2.78 \cdot 10^{-5})}$} \\\hline \noalign{\smallskip}
     $2000$  &\makecell{$7.14 \cdot 10^{-3}$ \\${\scriptstyle (3.87\cdot 10^{-3})}$\\$8.82 \cdot 10^{-4} $ \\${\scriptstyle (1.13 \cdot 10^{-4})}$}&$60$&\makecell{$\mathbf{8.20 \cdot 10^{-4}}$  \\${\scriptstyle (2.29 \cdot 10^{-4})}$\\$3.68 \cdot 10^{-4}$  \\${\scriptstyle (6.95 \cdot 10^{-5})}$} &\makecell{$1.88 \cdot 10^{-3}$  \\${\scriptstyle (6.11 \cdot 10^{-4})}$\\$4.93 \cdot 10^{-4}$  \\${\scriptstyle (4.46 \cdot 10^{-5})}$} & \makecell{$2.63\cdot 10^{-3}$ \\${\scriptstyle (3.60 \cdot 10^{-4})}$\\$1.22 \cdot 10^{-4}$\\${\scriptstyle (2.26\cdot 10^{-4})}$}   & \makecell{$1.09\cdot 10^{-3}$ \\${\scriptstyle (6.08 \cdot 10^{-4})}$\\$\mathbf{2.08 \cdot 10^{-5}}$\\${\scriptstyle (2.53\cdot 10^{-5})}$}  \\\hline \noalign{\smallskip}
    $4000$ &\makecell{$6.45\cdot 10^{-3}$ \\${\scriptstyle (2.02 \cdot 10^{-3})}$\\ $8.70 \cdot 10^{-4}$  \\${\scriptstyle (1.07 \cdot 10^{-4})}$}& $80$&\makecell{$7.11\cdot 10^{-4} $ \\${\scriptstyle (1.56 \cdot 10^{-4})}$\\$3.47 \cdot 10^{-4}$  \\${\scriptstyle (5.91 \cdot 10^{-5})}$} &\makecell{$1.05 \cdot 10^{-3} $ \\${\scriptstyle (3.07 \cdot 10^{-4})}$\\$7.10 \cdot 10^{-5}$  \\${\scriptstyle (2.19 \cdot 10^{-5})}$} & \makecell{$2.52\cdot 10^{-3}$\\${\scriptstyle (2.37\cdot 10^{-4})}$\\$6.85\cdot 10^{-5}$ \\${\scriptstyle (9.73 \cdot 10^{-5})}$} & \makecell{$\mathbf{5.35\cdot 10^{-4}}$\\${\scriptstyle ( 2.55\cdot 10^{-4})}$\\$\mathbf{1.23 \cdot 10^{-5}}$ \\${\scriptstyle (1.36 \cdot 10^{-5})}$} \\ \hline \hline
\end{tabular}
\caption{\textbf{Comparison of the MISE and the MSE at 0  for the Gaussian, Epanechnikov and Gamma kernels and lognormal ratio estimators} on the hazard rate $k(t) = a+c\cdot e^{-dt}$ with $a = 7\cdot 10^{-3}, c = 3 \cdot 10^{-2}, d = 7 \cdot 10^{-2}$ on a grid from $0$ to $600$. 
The MISE and MSE are computed with 200 simulations, standard deviation is shown in parenthesis. }
\label{table:MSE_CV_gaussian}
\end{table}

\paragraph{Comparison of bandwidth choice methods}
We now compare bandwidth choice methods for the Gamma kernel estimator on the exponentially decreasing hazard rate of Figure \ref{fig:gamma_ker_comp1_f}. The methods considered will be: local and global adaptive choice as in Sections \ref{sec:local_minimax} and \ref{sec:global_minimax}, cross-validation choice (see \cite{Patil01011994}) and  a variable nearest neighbor bandwidth as was used with the Gaussian kernel on Figure \ref{fig:gamma_ker_comp} (see \cite{Orava2011KnearestNK}). We provide the empirical MISE on the interval $[0,600]$ and the MSE at $0$, computed on 200 simulations for several sample sizes in Table \ref{table:MSE_CV_adapt}.
Although it is the most commonly used, the cross-validation choice of the bandwidth tends to choose an over-smoothing bandwidth which results in an overall reasonable MISE as the estimator performs well between 50 and 600 where the hazard is mainly smooth, but yields a high bias at 0 (see Figure \ref{fig:bw_choice_comp} in the appendix). The nearest neighbor bandwidth choice performs well in terms of integrated error but performs very poorly at 0 for all sample sizes. This is due to its high variations and the fact that it is too data-dependent thus lacking robustness (as seen on Figure \ref{fig:bw_choice_comp}). Furthermore, none of these two methods show a real improvement of the boundary bias with increasing sample size.\\
The local adaptive bandwidth choice performs well for high sample sizes, but its bias at 0 indicates high variations and a tendency to overfit similarly to the nearest neighbor bandwidth choice. However the boundary bias improves for increasing sample sizes. Overall, for this hazard rate, the global adaptive bandwidth choice is the one with the lowest boundary bias and outperforming the other methods for high sample sizes. As this method chooses a bandwidth among a set of bandwidths which are rather small (at least smaller than 1), it works best for relatively high sample sizes for which these bandwidth values are more adapted. Smaller sample sizes might  necessitate larger bandwidth sizes.
This makes the global adaptive bandwidth choice a relevant data-driven way to select a bandwidth which performs better than the most commonly used cross-validation.

\begin{table}[h!]
\centering
\begin{tabular}{ c c c c c c}
     & Global adaptive &CV bandwidth &Local adaptive &\multicolumn{2}{c}{Nearest neighbor}\\ \hline \hline \noalign{\smallskip}
     Size (m) & \makecell{ MISE \\MSE at $0$} & \makecell{ MISE \\MSE at $0$} & \makecell{ MISE \\MSE at $0$}&k&\makecell{ MISE \\MSE at $0$} \\ \hline \hline \noalign{\smallskip}
    $500$  & \makecell{$3.27\cdot 10^{-3}$ \\${\scriptstyle (2.15 \cdot 10^{-3})}$\\$\mathbf{4.00\cdot 10^{-5}}$\\ ${\scriptstyle (5.13 \cdot 10^{-5})}$}&   \makecell{$2.80\cdot 10^{-3}$ \\ ${\scriptstyle (2.37 \cdot 10^{-3})}$\\$2.68 \cdot 10^{-4}$ \\${\scriptstyle (1.92 \cdot 10^{-4})}$} &\makecell{$6.82 \cdot 10^{-3}$ \\${\scriptstyle (1.11 \cdot 10^{-2})}$\\ $5.35 \cdot 10^{-4}$ \\ ${\scriptstyle (9.72 \cdot 10^{-4})}$}&$25$& \makecell{$\mathbf{1.56 \cdot 10^{-3}}$  \\${\scriptstyle (7.87\cdot 10^{-4})}$\\$3.77 \cdot 10^{-4}$  \\${\scriptstyle (1.12 \cdot 10^{-4})}$} \\\hline \noalign{\smallskip}
    $1000$ & \makecell{$1.89\cdot 10^{-3}$ \\${\scriptstyle (1.27 \cdot 10^{-4})}$\\ $\mathbf{2.32 \cdot 10^{-5}}$ \\${\scriptstyle (2.78 \cdot 10^{-5})}$}&  \makecell{$2.05\cdot 10^{-3}$ \\${\scriptstyle (1.08\cdot 10^{-3})}$\\$2.46 \cdot 10^{-4}$ \\${\scriptstyle (1.93 \cdot 10^{-4})}$} &\makecell{$2.82 \cdot 10^{-3}$  \\${\scriptstyle (2.24 \cdot 10^{-3})}$ \\$4.13 \cdot 10^{-4}$ \\${\scriptstyle (6.40 \cdot 10^{-4})}$} &$40$ &\makecell{$\mathbf{1.03 \cdot 10^{-3}}$  \\${\scriptstyle (3.09\cdot 10^{-4})}$\\$3.66 \cdot 10^{-4}$  \\${\scriptstyle (9.56 \cdot 10^{-5})}$}\\\hline \noalign{\smallskip}
     $2000$  & \makecell{$1.09\cdot 10^{-3}$ \\${\scriptstyle (6.08 \cdot 10^{-4})}$\\$\mathbf{2.08 \cdot 10^{-5}}$\\${\scriptstyle (2.53\cdot 10^{-5})}$} &\makecell{$1.71\cdot 10^{-3}$\\${\scriptstyle (1.07 \cdot 10^{-3})}$\\$2.35 \cdot 10^{-4}$ \\${\scriptstyle (1.86 \cdot 10^{-4})}$} &\makecell{$1.62 \cdot 10^{-3}$ \\${\scriptstyle (1.11 \cdot 10^{-3})}$\\$3.24 \cdot 10^{-4} $ \\${\scriptstyle (5.96 \cdot 10^{-4})}$}&$60$&\makecell{$\mathbf{8.11 \cdot 10^{-4}}$  \\${\scriptstyle (1.96 \cdot 10^{-4})}$\\$3.71 \cdot 10^{-4}$  \\${\scriptstyle (7.24 \cdot 10^{-5})}$} \\\hline \noalign{\smallskip}
    $4000$& \makecell{$\mathbf{5.35\cdot 10^{-4}}$\\${\scriptstyle ( 2.55\cdot 10^{-4})}$\\$\mathbf{1.23 \cdot 10^{-5}}$ \\${\scriptstyle (1.36 \cdot 10^{-5})}$}& \makecell{$1.36\cdot 10^{-3}$\\${\scriptstyle (1.05 \cdot 10^{-3})}$\\$2.08 \cdot 10^{-4}$ \\${\scriptstyle (1.79 \cdot 10^{-4})}$}&\makecell{$1.02\cdot 10^{-3}$ \\${\scriptstyle (6.83 \cdot 10^{-4})}$\\ $2.61 \cdot 10^{-4}$  \\${\scriptstyle (3.50 \cdot 10^{-4})}$}& $80$&\makecell{$7.29\cdot 10^{-4}$  \\${\scriptstyle (1.61 \cdot 10^{-4})}$\\$3.59 \cdot 10^{-4}$  \\${\scriptstyle (6.84 \cdot 10^{-5})}$} \\ \hline \hline
\end{tabular}
\caption{\textbf{Comparison of the MISE and MSE at 0 for different bandwidth choice methods with the Gamma kernel} for the hazard rate $k(t) = a+c\cdot e^{-dt}$ with $a = 7\cdot 10^{-3}, c = 3 \cdot 10^{-2}, d = 7 \cdot 10^{-2}$ on a grid from $0$ to $600$. 
The MISE and MSE are computed with 200 simulations, standard deviation is shown in parenthesis.}
\label{table:MSE_CV_adapt}
\end{table}

The effect of the hazard shape on the local bandwidth choice is shown on Figure \ref{fig:oracle_steep}, where the adaptive estimator is shown for different widths of the peaks in the hazard rate. As shown by the bandwidth plots, the chosen bandwidth is small near 0 and 150, especially for the second hazard rate where the peaks are even narrower, thus allowing the estimator to pick up the rapid variations in these regions, while it is much greater on the rest of the interval, allowing for a smooth estimation of the quasi-constant phase of that hazard (the small bandwidths chosen at the end are due to lack of data). This illustrates the relevance of having a local choice of the bandwidth for hazard rates with a lot of variations, whereas Table \ref{table:MSE_CV_adapt} indicates that for somewhat smoother hazard rates, a global bandwidth choice can be better. 
\begin{figure}[h!]
    \centering
    \includegraphics[scale = 0.35]{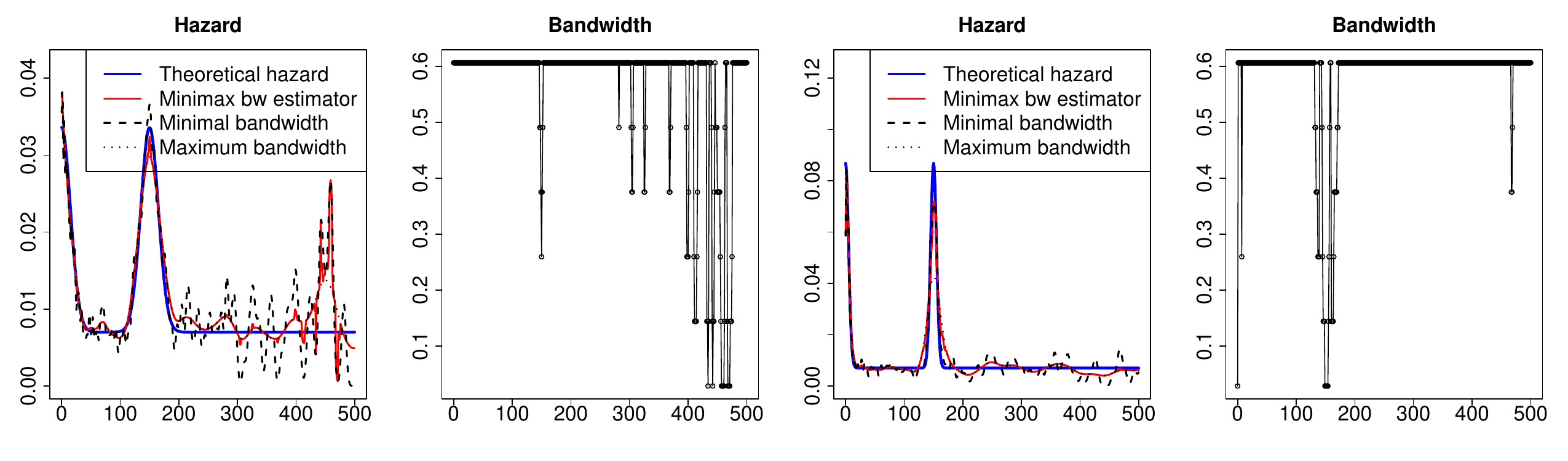}
    \caption{\textbf{Local adaptive bandwidth estimator and close-up of the chosen bandwidth} for $m=2000$, on two hazard rates 
     of the form  $k(t) = a + f_1(t) + f_2(t)$ with $a = 7\cdot 10^{-3}$ and $f_1$ and $f_2$ Gaussian densities of same sd $15$ (left plot) and $5$ (right plot), centered around $0$ and $150$ respectively.}
     \label{fig:oracle_steep}
\end{figure}

\subsection{Test on experimental data} \label{sec:data}
Finally, we test the local adaptive bandwidth choice procedure with the Gamma kernel on experimental data from \cite{tricoire_new_2015}. They study a 2-phases aging model (first introduced in \cite{rera_intestinal_2012}) in drosophila: before dying, all flies first enter a senescent "Smurf" phase at a certain rate $k_S$. This phase change is detected with a blue food die which permeates the entire body through the intestine when flies turn Smurf. Once Smurf, flies die shortly (but not immediately) at a rate $k_D$. The model is summed up on Figure \ref{fig:drawing_model}.

\begin{figure}[!ht]
\centering
\resizebox{0.8\textwidth}{!}{%
\begin{circuitikz}
\tikzstyle{every node}=[font=\normalsize]

\draw [->, >=Stealth] (2.5,9.2) -- (13.25,9.2);
\draw  (2.5,9) -- (2.5,9.2);
\draw  (7.5,8.3) -- (7.5,8.5);
\draw [->, >=Stealth] (7.5,8.5) -- (13.25,8.5);
\draw  (2.5,10.25) rectangle (5,9.75);
\draw  (7.5,10.25) rectangle (9,9.75);
\node [font=\large] at (12.5,10) {\text{dead}};

\node [font=\large] at (3.75,10) {non-Smurf};
\node [font=\large] at (8.25,10) {Smurf};
\node [font=\large] at (6,10.25) {$k_S(t)$};
\node [font=\large] at (10.5,10.25) {$k_D(a)$};
\draw [->, >=Stealth] (5,10) -- (7.5,10);
\draw [->, >=Stealth] (9,10) -- (12,10);
\node [font=\normalsize] at (2.75,8.7) {t = 0};
\node [font=\normalsize] at (7.5,8) {a = 0};
\node [font=\normalsize] at (14,9.2) {\textit{time t }};
\node [font=\normalsize] at (14.25,8.5) {\textit{Smurf age a }};
\end{circuitikz}
}%
\caption{Schematic representation of the two-phased model.}
\label{fig:drawing_model}
\end{figure}

Mathematically speaking, this means that the rate of death for smurf flies, $k_D$, is high for small smurf ages $a$ as the smurf phase is a strong predictor of death. In particular, it is key to correctly estimate $k_D$ near $0$. To that effect, the use of the Gamma kernel, particularly adapted to hazard rates which are non-zero at $0$, along with the adaptive bandwidth procedure is relevant. 
More precisely, the data we estimate the hazard rate on consists in the time spent smurf of 1159 independent drosophila. 

\begin{figure}
    \centering
    \includegraphics[scale = 0.25]{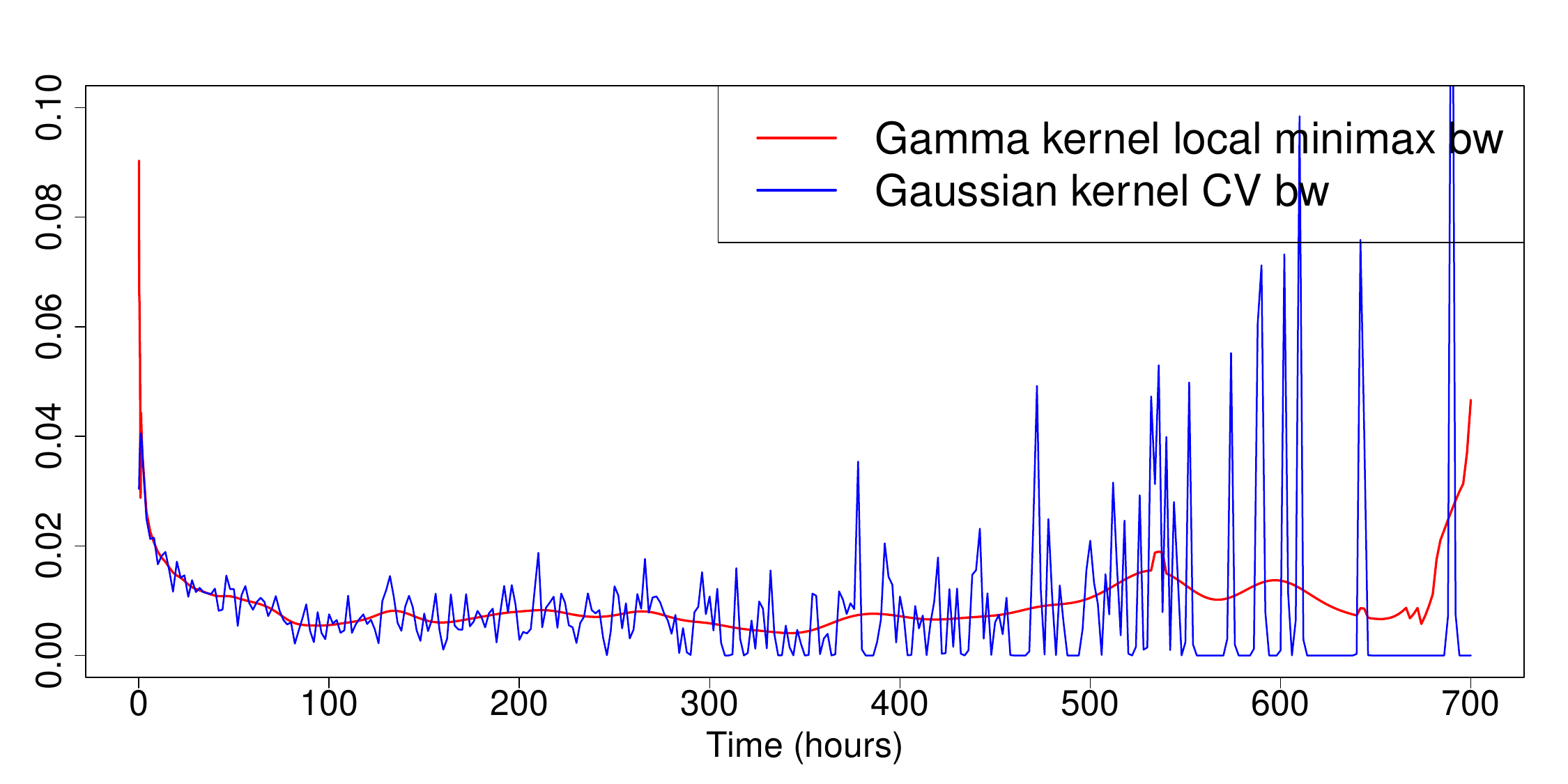}
    \caption{Kernel estimator of the death hazard rate estimation in smurf flies with the Gamma kernel with the adaptive bandwidth procedure and the Gaussian kernel with cross-validation bandwidth.}
    \label{fig:oracle_smurf}
\end{figure}

The result of the estimation is shown on Figure \ref{fig:oracle_smurf}. The adaptive bandwidth choice selects the smallest bandwidth near 0, where the hazard rate decreases drastically and the largest bandwidth is then selected, as the hazard rate is quasi-constant after the 50 first hours. 

In comparison, the Gaussian kernel estimator with cross-validation bandwidth (the chosen bandwidth is $0.8$) both underestimates the initial peak, and overfits the rest of the hazard. The Gamma kernel allows to fully capture the unusual behavior of the death rate which is particularly high at $0$, while still providing a readable estimation of the more constant part of the hazard. 
Quantitatively, the initial peak with the Gaussian kernel estimator is of $0.04$, significantly smaller than the $0.09$ yielded by the Gamma kernel with adaptive bandwidth choice. The use of an associated kernel, namely the Gamma kernel, allows to capture the height of the initial peak, which translates the fact that the smurf phenotype is a strong predictor of death, and that the transition to this phenotype is accompanied by a particularly high chance of death in the first hours. Biologically, it shows that at one key point in their lives, fruit flies undergo a drastic decrease of several health indicators, which is accompanied by an extremely high risk of biological failure and thus of impending death. However, not all flies die immediately after this transition, and some are capable of surviving up to a few weeks, even with such decreased capacities.

\section{Concluding remarks}

In this work, we present convergence results in the general framework of kernel hazard rate estimation with associated kernels. Associated kernels are still an active field of research, but few results exist in all generality, and none in the case of hazard rate estimation. We both prove the convergence of the MISE and asymptotic normality. Furthermore, we provide results on a data-driven adaptive bandwidth selection method. The general formulation of the kernel and its implicit dependence in time, as well as its infinite support, result in several theoretical difficulties which are solved through the introduction of several assumptions, most of them being trivially verified by classical kernels. 
We show that all of our theoretical results apply to the Gamma kernel, and provide several simulations showing both the relevance of the Gamma kernel and of the adaptive bandwidth choice when estimating hazard rates, especially when they are non-0 at 0. We also use the Gamma kernel estimator on experimental data, which shows that using an associated kernel allows to capture the real behavior of the hazard, and thus of the underlying biological mechanisms. 

Our results, along with the assumptions on the kernels on which they rely, provide some guidelines as to which properties a kernel should have to provide good estimations. This could help in designing new kernels, but also highlights the key properties of existing kernels. This article tackles the problem of estimating a hazard from independent observations, further perspectives could include extending it to dependent data. \blue{A very natural extension of this work would be to study an associated kernel
estimator of the hazard rate for censored data, or more generally for the
estimation of counting process intensities in the multiplicative intensity
model, a framework in which standard kernel estimators have been studied (see e.g. \cite{Ramlau_haz}).
We believe that the results of this paper could be extended to this setting,
although additional technical difficulties may arise.} Finally, we  only consider positive kernels, but similar results for non-positive kernels could be studied. 
\section*{Code availability}
The code used in this work can be found at \url{https://github.com/luce-breuil/non_param_estim_assoc}.

\section*{Acknowledgments}
The authors would like to thank Marie Doumic for her contribution in the mathematical formulation of the problem, as well as for her many valuable suggestions and comments on this work. The authors also thank Michaël Rera, who is at the origin of the biological
questions motivating this work, for fruitful discussions and for providing the
data analysed in Section 5.

\appendix
\section{Appendix}
\label{sec:appendix}
\subsection{Verification of the assumptions for specific kernels.}
\label{sec:proof_2.1}
\begin{proof}[Proof of Proposition \ref{prop:gamma}]
    
We first prove that the unbiased Gamma kernel defined by \eqref{eq:Gamma_kernel_1} and \eqref{eq:rho_gamma_1} falls under Definitions \ref{def:cont_ass} and \ref{def:order_beta} for $\beta = 2$ as well as verifies Assumptions \ref{ass:order},\ref{assump:compat}, \ref{assump:unif_int} and \ref{assump_ad_strong} (for a bounded hazard rate). We recall the following equivalent for the Gamma function. 
\begin{equation}
    \Gamma(z) \underset{z \to +\infty}{\sim}\sqrt{2\pi}z^{z-1/2}e^{-z}. \label{eq:equiv_gamma}
\end{equation}
\textbf{Definitions \ref{def:cont_ass} and \ref{def:order_beta}.}
We have $\St = \mathbb{R}_+$.
For $t > 2b$,
\begin{align*}
    \Lambda(t,b) = \rho(t)_b b -t = 0 \text{ and }Var(Z_{t,b}) = \rho(t)_b b^2 = tb \xrightarrow[b \rightarrow 0]{} 0.
\end{align*}  
Furthermore, for $t \leq 2b$, 
\begin{align*}
    \Lambda(t,b) =  \frac{1}{4}\frac{t^2}{b} +b - t \leq 2b  \xrightarrow[b \rightarrow 0]{} 0 \text{ and }
    Var(Z_{t,b}) =  \frac{1}{4}t^2 +b^2 \leq 2b^2 \xrightarrow[b \rightarrow 0]{} 0.
\end{align*} 
By the previous computations,  $\Lambda(t,b) = O(b)$ and  $Var(Z_{t,b}) = O(b)$. Hence, Assumption \ref{assump:gamma_1} holds with $\gamma= \frac{1}{2}$.

\textbf{Assumption \ref{ass:order}.} Let $t>0$. Since we are looking to prove an asymptotic result, we will consider that $b\leq t/2$. We introduce $f_{t,b}(u) = b^{1/2}\kappa_{t,b}(b^{1/2}u+E[Z_{t,b}])$. Since for the Gamma kernel with $b \leq t/2$, $\E[Z_{t,b}] = t$, we have for $\eta >0$,
\begin{equation*}
    b^{-1}\int_{\St \cap \{|y -\E[Z_{t,b}]| > \eta\}} (y - \E[Z_{t,b}])^2 \kappa_{t,b}(y) \dd y = \int_{[-t/\sqrt{b}, +\infty) \cap \{|u| > \eta b^{-1/2}\}}f_{t,b}(u)u^2 \dd u.
\end{equation*}
We have for all $u  \geq-\frac{t}{\sqrt{b}}$ and by \eqref{eq:equiv_gamma}
\begin{align}
\label{eq:dom_f}
   0 < f_{t,b}(u) &=\frac{1}{\sqrt{2\pi t}} \frac{(1+\frac{u\sqrt{b}}{t})^{t/b-1} e^{-\frac{u}{\sqrt{b}}}}{1+o(b)}  \leq C\frac{1}{\sqrt{2\pi t}}\exp\Bigg(\Big(\frac{t}{b}-1\Big)\log\Big(1+\frac{u\sqrt{b}}{t}\Big)-\frac{u}{\sqrt{b}} \Bigg).
\end{align}

Let us first consider $\int_{\{u>\eta b^{-1/2}\}}u^2 f_{t,b}(u) \dd u $. Let us fix $0< \alpha \leq 1$.
If $\eta \geq \sqrt{b}\frac{t-\sqrt{b}}{1 - \sqrt{b}}$, which holds for $b$ small enough, the function $u \mapsto \Big(\frac{t}{b}-1\Big)\log\Big(1+\frac{u\sqrt{b}}{t}\Big)-\frac{u}{\sqrt{b}} + \alpha u$ is decreasing on $[\eta b^{-1/2}, + \infty[$. Combining this with \eqref{eq:dom_f}, we have for  $u \geq \frac{\eta}{\sqrt{b}}$,
\begin{align*}
  \frac{\sqrt{2\pi t}}{C} f_{t,b}(u) e^{\alpha u} 
    &\leq  \exp\Bigg(\Big(\frac{t}{b}-1\Big)\log\Big(1+\frac{\eta}{t}\Big)-\frac{\eta}{b} + \alpha \frac{\eta}{\sqrt{b}}\Bigg)\\
    & \leq  \exp\Bigg(\frac{t}{b}\Big(\log\Big(1+\frac{\eta}{t}\Big) - \frac{\eta}{t}\Big)  + \alpha \frac{\eta}{\sqrt{b}}\Bigg)  < 1,
\end{align*}
for $b$ small enough, since $\log\Big(1+\frac{\eta}{t}\Big) - \frac{\eta}{t} < 0$.
Hence $\one_{\{u>\eta b^{-1/2}\}}u^2 f_{t,b}(u) \leq C u^2  e^{-\alpha u}/\sqrt{2\pi t}$ for $b$ small enough, and hence by the dominated convergence theorem, $\int_{\{u>\eta b^{-1/2}\}}u^2 f_{t,b}(u) \dd u \xrightarrow[b\rightarrow 0]{}0$.

Let us now consider the case $-\frac{t}{\sqrt{b}} \leq u \leq  -\frac{\eta}{\sqrt{b}}$. As previously, we have
\begin{align*}
  u^2f_{t,b}(u) &\leq \frac{C}{\sqrt{2\pi t}} u^2\exp\Bigg(\Big(\frac{t}{b}-1\Big)\log\Big(1+\frac{u\sqrt{b}}{t}\Big)-\frac{u}{\sqrt{b}}\Bigg) \leq \frac{Ct}{\sqrt{2\pi t}b}\exp\Bigg(\Big(\frac{t}{b}-1\Big)\log\Big(1+\frac{u\sqrt{b}}{t}\Big)-\frac{u}{\sqrt{b}}\Bigg). 
\end{align*}

The function $u \mapsto (\frac{t}{b}-1)\log\Big(1+\frac{u\sqrt{b}}{t}\Big)-\frac{u}{\sqrt{b}}$ is increasing on $]-\frac{t}{\sqrt{b}} , -\frac{\eta}{\sqrt{b}}]$ provided $\eta \geq b$, which is the case for $b$ small enough. Hence for $-\frac{t}{\sqrt{b}} \leq u \leq  -\frac{\eta}{\sqrt{b}}$,
\begin{align*}
    &u^2f_{t,b}(u)  \leq \frac{C\sqrt{t}}{\sqrt{2\pi }b}\exp\Bigg(\Big(\frac{t}{b}-1\Big)\log\Big(1-\frac{\eta}{t}\Big)+\frac{\eta}{b}\Bigg) \leq \frac{C\sqrt{t}}{\sqrt{2\pi}b}\exp\Bigg(\frac{t}{b}\Big(\log\Big(1-\frac{\eta}{t}\Big)+\frac{\eta}{t}\Big) \Bigg).
\end{align*}
Hence 
\begin{align*}
    \int_{-\frac{t}{\sqrt{b}}}^{-\frac{\eta}{\sqrt{b}}}u^2f_{t,b}(u) \dd u \leq \frac{C\sqrt{t}(t-\eta)}{b^{3/2}} \exp\Bigg(\frac{t}{b}\Big(\frac{\eta}{t}+\log\Big(1-\frac{\eta}{t}\Big)\Big) \Bigg) \xrightarrow[b\rightarrow 0]{}0.
\end{align*}
Finally, 
\begin{equation*}
    \int_{|u| > \eta b^{-1/2}}u^2 f_{t,b}(u) \dd u  \xrightarrow[b\rightarrow 0]{}0
\end{equation*}
and Assumption \ref{ass:order} is verified for any $t>0$.
For $t = 0$, 
\begin{align*}
    \int_{u\geq \eta b^{-1/2}} u^2 f_{0,b}(u)\dd u  &=  \int_{u\geq \eta b^{-1/2}}\frac{u^2}{\sqrt{b}}e^{-u/\sqrt{b}-1} \dd u = \sqrt{b}\int_{v\geq \eta/b}v^2e^{-v-1} \dd v \xrightarrow[b\to 0]{}0. 
\end{align*}

\textbf{Assumption \ref{assump:gamma_sup}} Let $t>0$, and $b\leq 1$. \blue{As $b \mapsto ||\kappa_{t,b}||_{\infty} $ is continuous in $b$ for $b \in [t/2, 1]$, we have $\sup_{b \in [t/2, 1]} (||\kappa_{t,b}||_{\infty}) < + \infty$. }Thus, it suffices to show the result for $b < t/2$. 
For $b < t/2$, by differentiating w.r.t $y$, we find that for $y \geq 0$, $\kappa_{t,b}'(y) = 0 \iff y = t-b$. And the maximum is achieved at $t-b$. Thus,
\begin{align*}
    \max_{y \in \mathbb{R}_+} \kappa_{t,b}(y) &=  \kappa_{t,b}(t-b) = \frac{(t-b)^{t/b -1} e^{-t/b}}{b^{t/b} \Gamma(t/b)} \substack{\sim\\ b \rightarrow 0} \frac{e^{-1}}{\sqrt{2\pi b (t-b)}} = O(b^{-1/2}).
\end{align*}

Obviously, $\max_{y \in \mathbb{R}_+} \kappa_{0,b}(y) = 1/b$.\\

\textbf{Assumptions \ref{assump:compat} and \ref{assump_ad_strong}} As these assumptions depend both on the kernel and the hazard rate, we will consider the case of a bounded hazard rate (which is a reasonable hypothesis in survival analysis, and one that we make throughout this article). 
    Suppose $b < \min(1, 1/(4||k||_{\infty})$  and $\lambda > 2e$. \\
    \textbf{Case 1} $t>2b$. Note that $\lambda>2e$ implies that for $y \geq 0$ and $t < 2e$, $|t-y| > \lambda \implies y > t+\lambda > 1$. 
 We have for $y \geq 0$,
 \begin{align*}
    \kappa_{t,b}(y)\frac{1}{1-F(y)} &= \frac{y^{{t/b-1}}e^{-y/b}}{b^{t/b} \Gamma(t/b)} e^{\int_0^y k(u) \dd u}\\
    & \leq \frac{y^{{t/b-1}}e^{-y(\frac{1}{b}-||k||_{\infty})}}{b^{t/b}\Gamma(t/b)} \\
    &\leq C \frac{\sqrt{t}}{\sqrt{2b\pi}}\exp(\frac{1}{b}(t(\log(y)-\log(t))-y +t)+||k||_{\infty} y-\log(y))\\
    &\leq 
    \begin{cases}
    &C\frac{\sqrt{t}}{\sqrt{2b\pi}}\exp(\frac{1}{b}(t(\log(y)-\log(t))-\frac{y}{2} +t) -(\frac{1}{2b}-||k||_{\infty}) y) \text{ if $y \geq 1$}\\
     &  C\frac{\sqrt{t}}{\sqrt{2b\pi}}\exp(\frac{t}{b}(1 - \log(t)) +||k||_{\infty}) \text{ else }
    \end{cases}.
\end{align*}
Let us first consider the case $y >1$. The map $ y \mapsto t(\log(y)-\log(t))-\frac{1}{2}y +t$ goes to $-\infty$ as $y$ goes to $+\infty$ and is decreasing on $[2t, + \infty]$. Furthermore, for $y \in \R_+, |y-t|> 5t \iff y > 6t \implies t(\log(y)-\log(t))-\frac{1}{2}y +t < t(\log(6) - 2) <0 $. Hence for $y \geq \max(6t,1)$ such that $|y-t| \geq \lambda$ and since $\frac{1}{4b}-||k||_{\infty} >0$, 
\begin{align*}
    \kappa_{t,b}(y)\frac{1}{1-F(y)} &\leq C\frac{\sqrt{t}}{\sqrt{2b\pi}}\exp(-\frac{t}{b}(2-\log(6)) -(\frac{1}{4b}-||k||_{\infty}) y - \frac{1}{4b} y) \\
    &\leq C\frac{\sqrt{t}}{\sqrt{b}}\exp(-\frac{t}{b}(2-\log(6)) - \frac{1}{4b}) \leq C e^{- \frac{1}{4b}} < C_1.
\end{align*}
For $y$ such that $|y-t| \geq \lambda$ and $1 \leq y \leq 6t$, we have 
 \begin{align*}
    \kappa_{t,b}(y)\frac{1}{1-F(y)} 
    & \leq C \frac{\sqrt{t}}{\sqrt{2b\pi}}\exp(\frac{1}{b}(t(\log(y)-\log(t))-y +t)  +||k||_{\infty} y).
\end{align*}
The map $y \mapsto t(\log(y)-\log(t))-y +t $ has a maximum of $0$ at $y=t$ and is increasing on $[0,t]$ and decreasing on $[t, +\infty]$ hence $B_1(t) = -\min( t(\log(t-\lambda)-\log(t))+\lambda, t(\log(t+\lambda)-\log(t)) -\lambda) > 0$ is such that if $|y-t| \geq \lambda$ and $1 \leq y < 6t $, 
\begin{align*}
    \kappa_{t,b}(y)\frac{1}{1-F(y)} 
    & \leq C \frac{\sqrt{t}}{\sqrt{2\pi b}}\exp(-B_1(t)\frac{1}{b} +6||k||_{\infty} t)\leq C\frac{\sqrt{t}}{\sqrt{2\pi}}\exp(-B_1(t)\frac{1}{2b} +6||k||_{\infty} t) \leq C_2(t).
\end{align*}
And for any $T >0, \exists A >0$, $t \leq T \implies B_1(t) > A$.

Finally in the case $y < 1$, since $\lambda >2e$, we have $|y-t| >\lambda \implies t \geq 2e$ and 
\begin{align*}
    \kappa_{t,b}(y)\frac{1}{1-F(y)} \leq C\frac{\sqrt{t}}{\sqrt{2b\pi}}\exp(-\frac{t}{b}\log(2)  +||k||_{\infty}) < Ce^{-e\log(2)/2b} < C_3.
\end{align*}
Let $G(t) = \max(C_1,C_2(t),C_3)$, $G(t)$ verifies Assumption \ref{assump:compat}, and $B(t) = \min(e\log(2)/2, B_1(t)/2,1/4)$ verifies Assumption \ref{assump_ad_strong}. 

 \textbf{Case 2}  If $t \leq 2b \leq 2$, we have $|y-t| \geq \lambda \implies y \geq t + \lambda \geq 2e > 2b$ and since $b \leq \frac{1}{2||k||_{\infty}}$,
 \begin{align*}
    \kappa_{t,b}(y)\frac{1}{1-F(y)} &= \frac{y^{{1/4(t/b)^2}}e^{-y/b}}{b^{1/4(t/b)^2 + 1}\Gamma(1/4(t/b)^2+1)} e^{\int_0^y k(u) \dd u}\\
    & \leq \frac{1}{b} \Big( \frac{y}{b}\Big)^{1/4(t/b)^2}e^{-y/b+||k||_{\infty} y}  \leq \frac{1}{2e}\Big(\frac{y}{b}\Big)^2 e^{-y/4b} e^{-e/2b} \leq C_4.
\end{align*}
And $B(t) = e/2$ verifies Assumption \ref{assump_ad_strong}.\\

\textbf{Assumption \ref{assump:unif_int}}
Consider a compact set $I$. 

For $t \in I \cap [2b,L]$ (such an $L$ exists as $I$ is compact), since $ \inf_{u \in [2,+\infty[}(\Gamma(u)) = \Gamma(2) > 0$, we have
\begin{align} \label{eq:dom_kappa_1}
   \kappa_{t,b}(y) & = \frac{y^{\frac{t}{b}-1}e^{-y/b}}{b^{\frac{t}{b}} \Gamma\left(\frac{t}{b}\right)}  \leq \begin{cases}
       &\frac{y^{L/b-1} e^{-y/b}}{(b^{L/b} + b^2)\Gamma(2)} \quad \text{if $y \geq 1$}\\
       & \frac{y e^{-y/b}}{(b^{L/b} + b^2)\Gamma(2)} \quad \text{else}.
   \end{cases}
\end{align}
 The function defined by \eqref{eq:dom_kappa_1} is integrable on $\R_+$ as $L/b - 1 \geq 0$.
For $t \in I \cap [0, 2b]$, let $B = \inf_{u \in [1,2]}(\Gamma(u))$, we have 
\begin{align}\label{eq:dom_kappa_2}
   \kappa_{t,b}(y) & = \frac{y^{\frac{1}{4}(\frac{t}{b})^2}e^{-y/b}}{b^{\frac{1}{4}(\frac{t}{b})^2-1} \Gamma\left(\frac{1}{4}(\frac{t}{b})^2+1\right)} \leq \begin{cases}
       &\frac{y e^{-y/b}}{(1 + b^2)B} \quad \text{if $y \geq 1$}\\
       & \frac{e^{-y/b}}{(1 + b^2)B} \quad \text{else}.
   \end{cases}
\end{align}
The uniform bound can be taken to be $\eqref{eq:dom_kappa_2} + \eqref{eq:dom_kappa_1}$.\\

\textbf{Assumption \ref{assump:gamma_2_inf}}
Recall that $\gamma=\frac{1}{2}$ for the Gamma kernel. We have for $t > 2b$ and $r = 2,3$,
\begin{align}
    &\int_{\mathbb{R}_+} \kappa_{t,b}(y)^r\dd y  = \frac{r^{1-rt/{b}}\Gamma(rt/b -r+1)}{b^{r-1} \Gamma^r(t/b)} \substack{\sim\\ b \rightarrow 0} \frac{1}{(2\pi)^{\frac{r-1}{2}}r^{\frac{1}{2}}(tb)^{\gamma(r-1)}}. \label{eq:int_gamma_eq}
\end{align}
The result is straightforward for $t=0$. 
\end{proof}
We now present the proof of Propositions \ref{prop:gamma_2} and \ref{prop:gamma_2_global}. 
\begin{proof}[Proof of Propositions \ref{prop:gamma_2} and \ref{prop:gamma_2_global}]
\label{sec:proof_2.2}
The proof for Assumption \ref{assump_ad_strong} follows from the proof of Assumption \ref{assump:compat} in the proof of Proposition \ref{prop:gamma}.
We now verify that the Gamma kernel verifies Assumption \ref{assump:intt}. As the expression of the Gamma kernel differs for $t \geq 2b$ and $t < 2b$, we study the integrals by splitting them accordingly. Let us start with the integral on $[2b, +\infty]$. 

\textbf{Case 1} $y \geq 1$. Using the equivalent of the Gamma function \eqref{eq:equiv_gamma}, it follows that for $b \leq 1$,  
\begin{align*}
    \int_{2b}^{+\infty} \kappa_{t,b}(y) \dd t &= \frac{e^{-y/b}}{y}\int_{2b}^{+\infty} \Big(\frac{y}{b}\Big)^{t/b}\frac{1}{\Gamma(t/b)} \dd t  = \frac{be^{-y/b}}{y}\int_{2}^{+\infty} \Big(\frac{y}{b}\Big)^{u}\frac{1}{\Gamma(u)} \dd u \\
    & \leq C \frac{be^{-y/b}}{y} \int_{2}^{+\infty} \exp( -u( \log(u) - \log(y/b) -1))\sqrt{u} \dd u.
\end{align*}
Since $y \geq 1$, we have on $[e^2y, +\infty)$
\begin{align*}
    C \frac{be^{-y/b}}{y}\int_{e^2 y/b}^{+\infty} \exp( -u( \log(u) - \log(y/b) -1))\sqrt{u} \dd u   \leq C  be^{-1/b} \int_{e^2 y/b}^{+\infty} e^{-u}\sqrt{u} \dd u \leq C'  be^{-1/b}  \leq C'' .
\end{align*}
If $y/b > 2e^{-2}$,
since $u \mapsto u (\log(u) - \log(y/b) -1)$ has a global minimum of $-y/b$ at $y/b$, on $[2,e^2y]$, we have
\begin{align*}
     C \frac{be^{-y/b}}{y}\int_2^{e^2 y/b}\exp( -u( \log(u) - \log(y/b) -1))\sqrt{u} \dd u \leq C \frac{b}{y}\int_2^{e^2 y/b}\sqrt{u}\dd u  \leq Ce^4 \sqrt{\frac{y}{b}} \leq Ce^3 \sqrt{2}.
\end{align*}

\textbf{Case 2} For $y \leq 1$, we have
\begin{align*}
     \frac{e^{-y/b}}{y} \Big(\frac{y}{b}\Big)^{t/b}\frac{1}{\Gamma(t/b)} \leq   e^{-y/b} \frac{1}{b^{t/b}}\frac{1}{\Gamma(t/b)}
\end{align*}
and the result can be proved by similar computations as shown for $y \geq 1$. 

We now study the integral on $[0,2b]$.
For $t \in [0,2b]$, we have
\begin{align*}
    \int_0^{2b} \kappa_{t,b}(y) \dd t &=  e^{-y/b}\int_0^{2} \Big(\frac{y}{b}\Big)^{u^2/4} \frac{1}{\Gamma(u^2/4+ 1)} \dd u  \leq  e^{-y/b}2  \Big(\frac{y}{b}+1\Big) \leq C.
\end{align*} 
In any case, the integral of the kernel over $t$ is bounded by a constant independent of $y$ and $b$.

For $\int_I \kappa_{t,b}(y) ^2 \dd t $, we have by the proof of Assumption \ref{assump:gamma_sup} in the proof of Proposition \ref{prop:gamma} that for some constant $C$ and $t \geq 2b$, $\sup_{y \in \R_+} \kappa_{t,b}(y)\leq \frac{C}{\sqrt{2\pi b t}} $. Hence $\sup_{t \geq 2b} \sup_{y \in \R_+} \kappa_{t,b}(y)\leq \frac{C}{2b\sqrt{\pi }}$. Similarly for $t < 2b$, $\kappa_{t,b}(y) \leq e^{-y/b}(y/b+1)$. 
Thus $\int_I \kappa_{t,b}(y) ^2 \dd t \leq C/b $. Hence the Gamma kernel verifies Assumption \ref{assump:intt} with $\eta = 1$. 
\end{proof}

\begin{prop} \label{prop:other_kernels}
    The Reciprocal inverse Gaussian, lognormal and Weibull kernels verify Definitions \ref{def:cont_ass} and \ref{def:order_beta} and Assumptions \ref{ass:order} to \ref{assump:intt}. 
\end{prop}
\begin{proof}
\textbf{Reciprocal inverse Gaussian (RIG) kernel.} The RIG kernel is defined for $t,y \in \R_+$ and $b >0$ by,
\begin{align*}
    \kappa_{t,b}(y) = \frac{1}{\sqrt{2\pi b y}} e^{-\frac{t-b}{2b}(\frac{y}{t-b} - 2 + \frac{t-b}{y})}. 
\end{align*}
\begin{itemize}
    \item Definition \ref{def:cont_ass} and Assumption \ref{ass:order} : As shown in \cite{RIG_kernel}, the kernel integrates to 1 and 
    \begin{align*}
    \E[Z_{t,b}] = t \text{ and } Var(Z_{t,b}) = b(t-b) + 2b^2. 
\end{align*}
So the RIG kernel verifies Definition \ref{def:cont_ass} and Assumption \ref{ass:order} with $\gamma = 1/2$ for $\beta$ up to 2. Furthermore, for $b\leq \min(1,t)$ and $A$ large enough such that $A >t$ and $\sqrt{\frac{\sqrt{b}A+t}{t-b}} - \sqrt{\frac{t-b}{\sqrt{b}A+t}}  > \frac{1}{2}\sqrt{\frac{\sqrt{b}A+t}{t-b} }> 1 $, 
\begin{align*}
    \int_{|y-t| > \eta} &(y-t)^2\frac{1}{\sqrt{2\pi yb}b} e^{-\frac{t-b}{2b}(\frac{y}{t-b} - 2 + \frac{t-b}{y})} \dd y \\
    & =   \int_{|u| > \eta/\sqrt{b}} \frac{u^2}{\sqrt{2\pi u}}  e^{-\frac{t-b}{2b}\big( \sqrt{\frac{\sqrt{b}u+t}{t-b}} - \sqrt{\frac{t-b}{\sqrt{b}u+t}}\big)^2} \dd u \\
    & \leq   \int_{|u| > A/\sqrt{b}} \frac{u^{3/2}}{\sqrt{2\pi}}  e^{-\frac{1}{4b}(\sqrt{b}u+t)} \dd u +  \int_{A/\sqrt{b} > |u| > \eta/\sqrt{b} } \frac{u^{3/2}}{\sqrt{2\pi}}   e^{-\frac{t-b}{2b}\big( \sqrt{\frac{\sqrt{b}u+t}{t-b}} - \sqrt{\frac{t-b}{\sqrt{b}u+t}}\big)^2} \dd u \\
    & \leq   \int_{|u| > A/\sqrt{b}} \frac{u^{3/2}}{\sqrt{2\pi}}  e^{-\frac{1}{4\sqrt{b}}u} \dd u + (A-\eta)\frac{A^{3/2}}{b^2\sqrt{2\pi}}   e^{-\frac{t-b}{2b} N}  \xrightarrow[b \to 0]{}0, 
\end{align*}
where $N = \sup_{A/\sqrt{b}> |u| > \eta/\sqrt{b} } \big( \sqrt{\frac{\sqrt{b}u+t}{t-b}} - \sqrt{\frac{t-b}{\sqrt{b}u+t}}\big)^2 > 0$. 
Thus, Assumption \ref{ass:order} is verified.
\item Assumption \ref{assump:gamma_sup}. We have 
\begin{align*}
    &\max_{y \in \R_+} (\kappa_{t,b}(y)) = \frac{1}{(t-b)\sqrt{2\pi b (1+\sqrt{b})}} e^{-\frac{t-b}{2b}((t-b)(1+\sqrt{b}) - 2 + \frac{1}{(t-b)(1+\sqrt{b})})} \leq C_s(t) b^{-1/2}. 
\end{align*}
\item Assumptions \ref{assump:compat} and \ref{assump_ad_strong} can be shown similarly as for the Gamma kernel for a bounded hazard, using the fact that for $y \gg 1$, the RIG kernel verifies $ \kappa_{t,b}(y) \sim  \frac{1}{\sqrt{2\pi b y}} e^{-\frac{y}{2b}}$ and the result holds for $1/b $ large enough compared to $||k||_{\infty}$. 
\item Assumption \ref{assump:unif_int} is straightforward to verify as $\kappa_{t,b}(y)$ is a continuous and bounded function of $t$. 
\item Assumption \ref{assump:gamma_2_inf}. We have 
\begin{align*}
    &\int_{\R_+} \kappa_{t,b}(y)^2 \dd y = \frac{1}{2\sqrt{b\pi}(t-b)},\\
    & \int_{\R_+} \kappa_{t,b}(y)^3 \dd y = \frac{1}{\pi \sqrt{3} (t-b) b}.
\end{align*}
Where the result on $\int_{\R_+} \kappa_{t,b}(y)^2 \dd y $ is shown by noticing that $\int \frac{1}{y} e^{-\alpha y - \beta/y} \dd y = - \frac{\dd}{\dd \beta} \int e^{-\alpha y - \beta/y} \dd y$ and using 3.324 in \cite{Int_calcul} and the result on $\int_{\R_+} \kappa_{t,b}(y)^3 \dd y $ is shown by using Glasser's master theorem. 
\item The first equation of Assumption \ref{assump:intt} holds as 
\begin{align*}
    \int_{\R_+} \frac{1}{\sqrt{2\pi b y}} e^{-\frac{t-b}{2b}(\frac{y}{t-b} - 2 + \frac{t-b}{y})} \dd t = \frac{1}{\sqrt{2\pi b y}} e^{-\frac{3y}{8b}}\int_{\R_+} e^{-\frac{1}{2b} (\frac{t-b}{\sqrt{y}} - \frac{\sqrt{y}}{2})^2} \dd t \leq  e^{-\frac{3y}{8b}}. 
\end{align*}
Furthermore, if $I$ is such that $t \in I \implies t -b > \mu$ for some $\mu >0$ and $\alpha>0$ is sufficiently small such that for $t > b+ \mu$ and $y< \alpha$,  $\frac{t-b}{\sqrt{y}} - \frac{\sqrt{y}}{2} > \frac{\mu}{\sqrt{y}}> 1$ 
\begin{align*}
    &\int_{I} \frac{1}{2\pi b y}  e^{-\frac{1}{b} (\frac{t-b}{\sqrt{y}} - \frac{\sqrt{y}}{2})^2}  \dd t \leq  \frac{1}{\sqrt{2\pi b y}} \int_{I} \frac{1}{\sqrt{2\pi b y}} e^{-\frac{1}{2b} (\frac{t-b}{\sqrt{y}} - \frac{\sqrt{y}}{2})^2}  \dd t \leq \frac{1}{\sqrt{2\pi b y}} e^{-\frac{3y}{8b}}.
\end{align*}
Thus, the RIG kernel verifies Assumption \ref{assump:intt} with $\eta = 0$. 
If $0 \in I$, the RIG kernel verifies Assumption \ref{assump:intt} with $\eta = 1$. 
\end{itemize}
\textbf{Lognormal kernel.} The lognormal kernel is defined for $y,t \in \R_+$ and $b>0$ by 
\begin{align*}
    \kappa_{t,b}(y) = \frac{1}{yb\sqrt{2\pi}}e^{-\frac{1}{2}\Big(\frac{\log(y/t)}{b} - b\Big)^2}.
\end{align*}
\begin{itemize}
\item Definition \ref{def:order_beta} and Assumption \ref{ass:order}: 
As shown in \cite{esstafa:hal-04112846}, the kernel integrates to 1 and 
\begin{align*}
    \E[Z_{t,b}] = t + (e^{3b^2/2} - 1)t \text{ and } Var(Z_{t,b}) = t^2 e^{3b^2}(e^{b^2} - 1).
\end{align*}
So the lognormal kernel verifies Definition \ref{def:cont_ass}  with $\gamma = 1$ for $\beta$ up to 2. 
Furthermore, Assumption \ref{ass:order} can be checked as follows. By noticing that for $b$ small enough, $\log((\eta + \E[Z_{t,b}])/t) > 0$ and $\log((-\eta + \E[Z_{t,b}])/t) < 0$, we have
\begin{align*}
   &\int_{|y - \E[Z_{t,b}]| > \eta}\frac{(y-\E[Z_{t,b}])^2}{yb^3} e^{-\frac{1}{2}\big(\frac{\log(y/t)}{b} - b\big)^2} \dd y  = \int_{|u| > \eta/b}\frac{u^2}{bu +\E[Z_{t,b}]} e^{-\frac{1}{2}\big(\frac{\log((bu +\E[Z_{t,b}])/t)}{b} - b\big)^2} \dd u \\
   & = \int_{v > A/b}\frac{v^2}{b^3(v +\E[Z_{t,b}])} e^{-\frac{1}{2}\big(\frac{\log((v +\E[Z_{t,b}])/t)}{b} - b\big)^2} \dd v + \int_{A/b^2 > u > \eta/b}\frac{u^2}{bu +\E[Z_{t,b}]} e^{-\frac{1}{2}\big(\frac{\log((bu +\E[Z_{t,b}])/t)}{b} - b\big)^2} \dd u\\
   & \qquad + \int_{ -\E[Z_{t,b}]/b< u < -\eta/b}\frac{u^2}{bu +\E[Z_{t,b}]} e^{-\frac{1}{2}\big(\frac{\log((bu +\E[Z_{t,b}])/t)}{b} - b\big)^2} \dd u\\
   & \leq\int_{v > A/b}\frac{v^2}{(v +\E[Z_{t,b}])} e^{-\frac{1}{4}\big(\frac{\log((v +\E[Z_{t,b}])/t)}{b} - b\big)^2} \dd v +  \frac{A^3}{b^6(\eta +\E[Z_{t,b}])} e^{-\frac{1}{2}\big(\frac{\log((\eta +\E[Z_{t,b}])/t)}{b} - b\big)^2} \\
   & \qquad + \frac{\E[Z_{t,b}]^2}{b^2} \int_{ z <  \log(-\eta+ \E[Z_{t,b}])/b}  e^{-\frac{1}{2}\big(z - b\big)^2 }\dd z\\
    & \leq\int_{v > A/b}\frac{v^2}{(v +\E[Z_{t,b}])} e^{-\frac{1}{4}\big(\frac{\log((v +\E[Z_{t,b}])/t)}{b} - b\big)^2} \dd v +  \frac{A^3}{b^6(\eta +\E[Z_{t,b}])} e^{-\frac{1}{2}\big(\frac{\log((\eta +\E[Z_{t,b}])/t)}{b} - b\big)^2} \\
   & \qquad + C\E[Z_{t,b}]^2 \int_{ z <  \log(-\eta+ \E[Z_{t,b}])/b}  e^{-\frac{1}{4}\big(z - b\big)^2 }\dd z \xrightarrow[b \to 0]{}0. 
\end{align*}

\item Assumption \ref{assump:gamma_sup} is verified as $\max_{y \in \R_+^*} \kappa_{t,b}(y) = \frac{e^{-b^2/2}}{tb\sqrt{2\pi}} \leq C_s(t) b^{-1}$. 
\item Assumptions \ref{assump:compat} and \ref{assump_ad_strong} are verified for kernels such that $\int_{0}^y k(u) \dd u  =_{y \to + \infty} O( \log(y)^2)$ as for $y \gg 0, \kappa_{t,b}(y) \sim \frac{1}{2b \sqrt{2\pi}}e^{-\frac{1}{2b^2}\log(y)^2}$. Hence the result is verified for $\frac{1}{b^2}$ large enough compared to $\sup_{y>1} (\int_0^y k(u) \dd u/\log(y)^2).$
\item Assumption \ref{assump:unif_int} is easily  checked as the kernel is bounded in $t$ on any compact set of $\R_+$. 
\item Assumption \ref{assump:gamma_2_inf} can be shown with the change of variables $z = \log(y/x)/b$, and we have 
\begin{align*}
    &\int_{\R_+^*} \kappa_{t,b}(y)^2 \dd y = \frac{1}{2\pi b} \int_{\R}t  e^{bz}e^{-(z - b)^2} \dd z, \\
    &\int_{\R_+^*} \kappa_{t,b}(y)^3 \dd y = \frac{1}{(2\pi)^{3/2} b^2} \int_{\R}t^2  e^{2bz}e^{-3/2(z - b)^2} \dd z. 
\end{align*}
\item Assumption \ref{assump:intt} is shown in \cite{esstafa:hal-04112846} for $\eta = 0$ and we have for $y \in \R_+$, 
\begin{align*}
    \int_{t\in \R_+} \kappa_{t,b}(y) \dd t = e^{-b^2/2} \text{ and } \int_{t\in \R_+^*} \kappa_{t,b}(y)^2 \dd t = \frac{e^{-3b^2/4}}{2by\sqrt{\pi}}. 
\end{align*}
\end{itemize}
\textbf{Weibull kernel.} The Weibull kernel is defined for $t \in \R_+$, $y \in \R_+$ and $b>0$ by 
\begin{align*}
    \kappa_{t,b}(y) = \frac{1}{tb} \big(\frac{y}{t}\big)^{1/b-1} e^{-\big(\frac{y}{t}\big)^{1/b} }.
\end{align*}
\begin{itemize}
    \item Definition \ref{def:cont_ass} and Assumption \ref{ass:order}. It is shown in \cite{esstafa:hal-04112846} that the kernel integrates to 1 and that 
    \begin{align*}
        \E[Z_{t,b}] = t \Gamma(1+b) \text{ and } Var(Z_{t,b}) = t^2 (\Gamma(1+2b) - \Gamma^2(1+b)). 
    \end{align*}
    So the kernel verifies Definition \ref{def:order_beta} for an order $\beta$ up to $2$ and $\gamma = 1$.
  Furthermore, Assumption \ref{ass:order} can be checked as follows,
   \begin{align*}
        &\int_{|y-\E[Z_{t,b}]| > \eta} b^{-2}(y-\E[Z_{t,b}])^2 \frac{1}{tb} \big(\frac{y}{x}\big)^{1/b-1} e^{-\big(\frac{y}{x}\big)^{1/b} } \dd y  \\
        &=  \int_{u > A/b^2} \frac{u^2}{t} \big(\frac{bu+\E[Z_{t,b}]}{t}\big)^{1/b-1} e^{-\big(\frac{bu+\E[Z_{t,b}]}{t}\big)^{1/b} } \dd u +  \int_{A> b^2u > \eta/b} \frac{u^2}{t} \big(\frac{bu+\E[Z_{t,b}]}{t}\big)^{1/b-1} e^{-\big(\frac{bu+\E[Z_{t,b}]}{t}\big)^{1/b} } \dd u\\
        & \qquad +  \int_{-\E[Z_{t,b}]/b < u < -\eta/b} \frac{u^2}{t} \big(\frac{bu+\E[Z_{t,b}]}{t}\big)^{1/b-1} e^{-\big(\frac{bu+\E[Z_{t,b}]}{t}\big)^{1/b} } \dd u\\
        & \leq C \int_{u > A/b^2} \frac{u^2}{t} e^{-\frac{1}{2}\big(\frac{bu+\E[Z_{t,b}]}{t}\big)^{1/b} } \dd u ++  C\frac{A^3}{tb^6} e^{-\frac{1}{2}\big(\frac{\eta+\E[Z_{t,b}]}{t}\big)^{1/b} } + \frac{E[Z_{t,b}]^2}{b^2} \int_{0 < z < ((-\eta+ E[Z_{t,b}])/t)^{1/b}}  e^{-z } \dd z\\
         & \leq C \int_{v > A/b} \frac{v^2}{t} e^{-\frac{1}{4}\big(\frac{v+\E[Z_{t,b}]}{t}\big)^{1/b} } \dd v +  C\frac{A^3}{tb^6} e^{-\frac{1}{2}\big(\frac{\eta+\E[Z_{t,b}]}{t}\big)^{1/b} } + \frac{E[Z_{t,b}]^2}{b^2} ((-\eta+ E[Z_{t,b}])/t)^{1/b}\\
         & \xrightarrow[b \to 0]{}0.
    \end{align*}

\item Assumption \ref{assump:gamma_sup} is verified as $\max_{y \in \R_+} \kappa_{t,b}(y) = (1-b)^{1-b}e^{-(1-b)}/tb\leq C_s(t) b^{-1}$.
\item Assumptions \ref{assump:compat} and \ref{assump_ad_strong} are verified for bounded kernels provided $b< 1$. 
\item Assumption \ref{assump:unif_int} holds as for any fixed $y$, $\kappa_{t,b}(y)$ as a function of $t$ is continuous and bounded on any compact subset of $\R_+$. 
\item Assumption \ref{assump:gamma_2_inf},
\begin{align*}
   & \int_{\R_+} \frac{1}{t^2b^2} \big(\frac{y}{b}\big)^{2/b-2} e^{-2\big(\frac{y}{b}\big)^{1/b-1} } \dd y = \frac{1}{tb }\int_{\R_+} u^2 e^{-2u } \dd u,\\
    &\int_{\R_+} \frac{1}{t^3b^3} \big(\frac{y}{b}\big)^{3/b-3} e^{-3\big(\frac{y}{b}\big)^{1/b-1} } \dd y = \frac{1}{t^2b^2 }\int_{\R_+} u^3 e^{-3u } \dd u .
\end{align*}
\item Assumption \ref{assump:intt} is shown in \cite{esstafa:hal-04112846} for $\eta = 0$ and we have 
\begin{align*}
    \int_{t\in \R_+} \kappa_{t,b}(y) \dd t = \Gamma(1-b) \text{ and } \int_{t\in \R_+} \kappa_{t,b}(y)^2 \dd t = \frac{2^b\Gamma(2-b)}{4b}. 
\end{align*}
\end{itemize}
\end{proof}

\color{black}

\subsection{Proof of Asymptotic normality}
\label{sec:proof_norm}
We now present the proof to Theorem \ref{thm:as_nor} in Section \ref{subsec:normal}, which is an adaptation of the proof of Theorem 3 presented in \cite{Tanner_Wong_HR_TCL}. We begin with a technical lemma.

\begin{lemma}\label{lemma:eq_var_norm}
Let $\tau$ be a random variable of hazard rate $k$ with $k$ continuous and bounded and $(\kappa_{t,b})$ an associated kernel verifying Definition \ref{def:cont_ass}, with $b_m \xrightarrow[m \to \infty]{} 0 $. We define
\begin{align}
    V_m(\tau) = \frac{1}{1-F(\tau)}(1 - F^{m}(\tau))\kappa_{t,b_m}(\tau).
\end{align}

Then under Assumptions \ref{assump:gamma_sup} and \ref{assump:gamma_2_inf}, we have for $r \in \{1,2,3\}$
    \begin{equation}
    \E[|V_m|^r] =  (1-F(t))^{-r} f(t)\int_{\St} \kappa_{t,b_m}(y)^r \dd y  + o (b_m^{-(r-1)\gamma}).\label{eq:vm_exp}
\end{equation}
\end{lemma}
\begin{proof}
For  $r \in \{1,2,3\}$ and any $\lambda >0$, 
    \begin{align}
    \E[|V_m(\tau)|^r] &= \int_{\St} (1-F(y))^{-r}(1-F(y)^m)^r \kappa_{t,b_m}^r(y) f(y) \dd y \nonumber\\ \nonumber
    &\leq \int_{\St \cap \{|y-t| \leq \lambda\}} (1-F(y))^{-r}(1-F(y)^m)^r \kappa_{t,b_m}^r(y) f(y) \dd y \\ & \qquad  + \int_{\St \cap \{|y-t| > \lambda\}} (1-F(y))^{-r}(1-F(y)^m)^r \kappa_{t,b_m}^r(y) f(y) \dd y.\label{eq:var_power_r}
\end{align}
We have 
\begin{align}
    &\Big|\int_{\St \cap \{|y-t| \leq \lambda\}} (1-F(y))^{-r}(1-F(y)^m)^r \kappa_{t,b_m}^r(y) f(y) \dd y -\frac{f(t)}{(1-F(t))^r}\int_{\St}\kappa_{t,b_m}(y)^r \dd y \Big| \nonumber\\
    & \leq \Big|\int_{\St \cap \{|y-t| \leq \lambda\}}\frac{f(y)}{(1-F(y))^r} ((1-F(y)^m)^r -1)\kappa_{t,b_m}(y)^r  \dd y\Big| \nonumber\\
    & \qquad + \Big|\int_{\St \cap\{ |y-t| \leq \lambda\}}\left(\frac{f(y)}{(1-F(y))^r}- \frac{f(t)}{(1-F(t))^r}\right)  \kappa_{t,b_m}(y)^r  \dd y \Big|+ \Big|\frac{f(t)}{(1-F(t))^r}\int_{\St\cap\{|y-t|\geq \lambda \}}\kappa_{t,b_m}(y)^r \dd y\Big|.\label{eq:norm_decomp}
\end{align}

For the first term of \eqref{eq:norm_decomp}, it holds
\begin{align*}
    &\Big|\int_{\St \cap |y-t| \leq \lambda}\frac{f(y)}{(1-F(y))^r} ((1-F(y)^m)^r -1)\kappa_{t,b_m}(y)^r  \dd y\Big|\\
    &\leq (1-(1-F(t+\lambda)^m)^r) \sup_{[t-\lambda, t + \lambda] \cap \St} \Bigg(\frac{ f(y)}{(1-F(y))^r}\Bigg)
   \int_{\St \cap \{|y-t| \leq \lambda\}}\kappa_{t,b_m}(y)^r \dd y =  o(b_m^{-(r-1)\gamma}),
\end{align*}
by Remark \ref{rem:alpha_beta_bound} and as $F(t+\lambda)^m \to 0$.\\
The second term of \eqref{eq:norm_decomp} is such that 
\begin{align*}
    &\Big|\int_{\St \cap \{|y-t| \leq \lambda\}}\left(\frac{f(y)}{(1-F(y))^r}- \frac{f(t)}{(1-F(t))^r}\right)  \kappa_{t,b_m}(y)^r  \dd y \Big|\\
    & \leq  C_s(t) b_m^{-(r-1)\gamma}\int_{\St \cap\{ |y-t| \leq \lambda\}}\Big|\frac{f(y)}{(1-F(y))^r}- \frac{f(t)}{(1-F(t))^r}\Big|\kappa_{t,b_m}(y)  \dd y = o(b_m^{-(r-1)\gamma}),
\end{align*}
By 
Remark \ref{rem:alpha_beta_bound}
and since  $\int_{\St \cap \{|y-t| \leq \lambda\}}\Big|\frac{f(y)}{(1-F(y))^r}- \frac{f(t)}{(1-F(t))^r}\Big|\kappa_{t,b_m}(y)  \dd y \to 0$ by Definition \ref{def:cont_ass}. \\
Finally, for the third term of \eqref{eq:norm_decomp}, we have
\begin{align*}
    &\Big|\frac{f(t)}{(1-F(t))^r}\int_{\St\cap\{|y-t|\geq \lambda\} }\kappa_{t,b_m}(y)^r \dd y\Big| \leq \frac{f(t)}{(1-F(t))^r} C_s(t) b_m^{-(r-1)\gamma}\mathbb{P}(|Z_{t,b_m}-t| \geq \lambda)  =  o(b_m^{-(r-1)\gamma}),
\end{align*}
by once again using Assumption \ref{assump:gamma_sup} and by Definition \ref{def:cont_ass}.  

Hence, the first term of \eqref{eq:var_power_r} is equivalent to 
\begin{equation*}
    (1-F(t))^{-r} f(t) \int_{\St} \kappa_{t,b_m}(y)^r \dd y.
\end{equation*}
And the second term is such that 
\begin{equation*}
    \int_{\St \cap \{|y-t| > \lambda\}}\hspace{-0.2cm} (1-F(y))^{-r}(1-F(y)^m)^r \kappa_{t,b_m}^r(y) f(y) \dd y \leq  G(t)^{r-1} ||k||_{\infty} \mathbb{P}(|Z_{t,b_m}-t| \geq \lambda)\xrightarrow[m\to \infty]{}0 . 
\end{equation*}
Thus the second term is negligible compared to the first one for $r =1, 2,3$ and \eqref{eq:vm_exp} holds.
\end{proof}

This leads us to the proof of Theorem \ref{thm:as_nor}.

\begin{proof}[Proof of Theorem \ref{thm:as_nor}]
\textbf{Step 1}
We start by introducing an auxiliary estimator for which it will be easier to prove the asymptotic normality. 
Let $R_i$ be the ordered rank of $\tau_i$. Define $W_i = \frac{\kappa_{t,b_m}(\tau_i)}{m-N_{\tau_i^-}}$ such that $\hat{k}_m(t) = \sum_{i=1}^m W_i = W$. 

It is shown in \cite{Tanner_Wong_HR_TCL} (Lemma 2) that for all $i \leq m$ and $i \neq j$,
\begin{align}
   & \E[W_i|\tau_i] = \frac{1}{m}V_m(\tau_i),\\
    &\E[W_i|\tau_j] = \frac{1}{m-1}\int_{\St}\frac{\kappa_{t,b_m}(y)}{1-F(y)} f(y) (1 - F(y)^{m-1}) \dd y + \frac{1}{m(m-1)}U_m(\tau_i),
\end{align}
with $V_m(\tau_i)$ as defined in Lemma \ref{lemma:eq_var_norm} and 
\begin{align*}
    &U_m(\tau_i) = - \int_{\St \cap \{y \leq \tau_i\}}\frac{\kappa_{t,b_m}(y)}{(1-F(y))^2}\big(1 - F(y)^m - mF(y)^{m-1}(1-F(y))\Big) f(y) \dd y.
\end{align*}

We introduce $\hat{W} = \sum_{i=1}^m \E[W|\tau_i] - (m-1) \E[W]$ and $\Delta_m = - \int_{\St}F(y)^{m-1}\kappa_{t,b_m}(y)f(y) \dd y$
such that, by Lemma 2 in \cite{Tanner_Wong_HR_TCL},
\begin{align}
   \hat{W} - E[\hat{W}] = \sum_{i=1}^m \left(\frac{1}{m}V_m(\tau_i) +\frac{1}{m}U_m(\tau_i)+ \Delta_m\right). \label{eq:w_hat}
\end{align}
And $\forall 1\leq i \leq m$, 
\begin{equation*}
    \E\Big[\frac{1}{m}V_m(\tau_i) +\frac{1}{m}U_m(\tau_i)+ \Delta_m\Big] = \E[\E[W|\tau_i] - \E[W]]= 0.
\end{equation*}
Furthermore, by point (i)  in the proof of Theorem 3 in \cite{Tanner_Wong_HR_TCL},
\begin{equation}
    |U_m| =O(\sum_{i=1}^m 1/i) =  O(\log m) \qquad \Delta_m = O\left(\frac{1}{m(m-1)}\right).\label{eq:um_deltam}
\end{equation}

\textbf{Step 3}
Now we want to apply Lyapunov's central limit theorem to $\Hat{W}-\E[\hat{W}]$ as expressed by the sum in \eqref{eq:w_hat} (see e.g. \cite{billingsley1995probability} p.362). Using the bounds shown earlier, there remains to verify that there exists $\delta >0$ such that 
\begin{equation}
    m\E\Big[\Big|\frac{1}{m}V_m(\tau_i) +\frac{1}{m}U_m(\tau_i)+ \Delta_m\Big|^{2+\delta}\Big]Var(\hat{W})^{ -(2+\delta)/2}\xrightarrow[m\rightarrow + \infty]{}0.
\end{equation}
We set $\delta = 1$. 
As $\Delta_m$ is negligible compared to $V_m$ and $U_m$, it is sufficient to show that $Var(\hat{W})^{-3/2} m\E[|V_m(\tau_i)/m + U_m(\tau_i)/m|^3] $ goes to 0. 
In the same way as what is done in (iii) in the proof of Theorem 3 in \cite{Tanner_Wong_HR_TCL}, it can be shown that 
\begin{equation}
    Var(\hat{W}) = mVar(V_m/m + U_m/m + \Delta_m) = \frac{1}{m(1-F(t))}k(t) \alpha_{b_m}(t) + o((m b_m^{\gamma})^{-1}).
\end{equation}
By expanding under the expectation and using the equivalents of $U_m$ and $\E[|V_m|^r]$ given by \eqref{eq:um_deltam} and Lemma \ref{lemma:eq_var_norm}, we have that $Var(\hat{W})^{-3/2} m \E[|V_m/m + U_m/m|^3]$ is of the order of 
\begin{align}
    &\left(m^{-1} \int_{\St} \kappa_{t,b_m}^2(y) \dd y\right)^{-3/2}\frac{1}{m^2} \nonumber
    \cdot \Bigg(\int_{\St} \kappa_{t,b_m}^3(y) \dd y \nonumber\\
    & \qquad+ 3\int_{\St} \kappa_{t,b_m}^2(y) \dd y \cdot log(m) + 3\int_{\St} \kappa_{t,b_m}(y) \dd y \cdot log(m)^2 + log(m)^3\Bigg) \xrightarrow{}0.\label{eq:conv_0_int_k}
\end{align}
This is shown using Assumptions \ref{assump:gamma_sup}, \ref{assump:gamma_2_inf}, Remark \ref{rem:alpha_beta_bound} and $mb_m^{\gamma} \rightarrow 0$.
By applying Lyapunov's central limit theorem to $\hat{W}$, we obtain 
\begin{equation*}
    \frac{\hat{W} - \E[\hat{W}]}{\sqrt{Var(\hat{W})}} \rightarrow \mathcal{N}(0,1).
\end{equation*}
By Theorem 3 (iii) in \cite{Tanner_Wong_HR_TCL}, $\hat{k}_m(t)$ and $\hat{W}$ have the same limiting distribution hence the result on $\hat{k}_m(t)$ follows. 

The expressions of the expectation and variance of $\hat{k}_m(t)$ are given by \eqref{eq:esp_km} and \eqref{eq:var_comp_k}.
\end{proof}

\subsection{Technical lemmas}
\label{sec:lemma_append}
The following two lemmas are technical lemmas needed to prove Proposition \ref{lemma:var_eq_k}. 

\begin{lemma}\label{lemma_var_F}
Let $F$ be a distribution function such that $\forall t \in \R_+$, $F(t) < 1$ then, 
    \begin{equation}
mI_m(y) := \sum_{i = 0}^{m-1} \binom{m}{i}\frac{F(y)^i(1-F(y))^{m-i}}{m-i} \xrightarrow[m\rightarrow + \infty]{}(1-F(y))^{-1}
\end{equation}
uniformly w.r.t. $y$ provided that $|t-y| \leq \lambda$.
\end{lemma}

\begin{proof}  
This result follows directly from Lemma 6 in \cite{Hazard2}.
\end{proof}

\begin{lemma}\label{lemma:var_2term1}
For $k \in \Sigma(\beta,L)$ and under Assumptions  \ref{assump:compat} and \ref{assump:gamma_2_inf}, we have
\begin{equation}
 \frac{m}{\alpha_{b_m}(t)} \int_{\St} \int_{y\leq z} 
 ( F(z)^m(1-F(y)^m)- \frac{1-F(y)}{F(z)-F(y)}(F(z)^m-F(y)^m)) \kappa_{t,b_m}(y)\kappa_{t,b_m}(z) k(y) k(z) \dd y \dd z \xrightarrow[m \to \infty ]{}0
  \end{equation}
\end{lemma}
\begin{proof}
The proof follows directly from Lemma 11 and 12 in \cite{Hazard2} using the fact that $\int \kappa_{t,b_m}(y) k(y) \dd y$ is bounded, Assumption \ref{assump:compat} and the fact that $\alpha_{b_m} \xrightarrow[m \to \infty ]{}+\infty$ with Assumption \ref{assump:gamma_2_inf}. 
\end{proof}

\subsection{Proof of Lemma \ref{lemma:emp_distrib}}

We introduce the empirical distribution function 
\begin{equation*}
    \Tilde{F}_m(x) =\frac{1}{m} \sum_{i=1}^m \one_{\{ \tau_i \leq x \}}.
\end{equation*}
By the Dvoretzky-Kiefer-Wolfowitz Inequality (see e.g. \cite{Vaart_1998},p.346), we have for any $\eta >0$, 
\begin{equation}
    \mathbb{P}(|| \Tilde{F}_m - F||_{\infty}\geq \eta) \leq 2e^{-2m\eta^2}. \label{eq:DKW}
\end{equation}
Since $||\tilde{F}_m - \hat{F}_m || _{\infty} \leq \frac{1}{m}$, we have for $\eta > 1/m$
\begin{align*}
    \mathbb{P}(|| \hat{F}_m - F||_{\infty}\geq \eta) &\leq \mathbb{P}(|| \Tilde{F}_m - F||_{\infty}\geq \eta - \frac{1}{m}) \leq 2e^{-2m(\eta-1/m)^2}  \leq 2e^{-2m\eta^2 + 4\eta}. 
\end{align*}

Thus for any $c_0 \geq \max(\sqrt{l/2}, 1/m)$, 
\begin{align*}
    \mathbb{P}(|| \hat{F}_m - F||_{\infty}\geq c_0 \sqrt{m^{-1}log(m)})\leq 2e^{-2c_0^2 log(m)}e^{4c_0\frac{\sqrt{\log(m)}}{\sqrt{m}}}\leq 2 m^{-2c_0^2}e^{4c_0} \leq c_l m^{-l}.
\end{align*}

Although equation \eqref{eq:DKW} is sufficient to conclude that for any positive integer $l$,  $\mathbb{P}(|| \hat{F}_m - F||_{\infty}\geq \eta) \leq C m^{-l}$ for some constant $C$ which depends on $\eta$ and $l$, we wish to obtain a more explicit result on the constant. The motivation for this is twofold, firstly from an application perspective, the values we will consider for $\eta$ (such as $c_F(t)$) will not necessarily be known and will have to be estimated in practice, it seems therefore judicious to know how they impact the constants in the problem. Secondly, from a theoretical perspective, as the constants we choose for $\eta$ may depend on $t$, it is convenient to know how exactly the upper bound constant also depends on $t$ in order to properly justify the integration of the upper bound when proving the global result.  We proceed as follows.

Furthermore, for a fixed $c>0$ and for any $l \in \mathbb{N}^*$, the Markov inequality yields
\begin{align*}
    \mathbb{P}(|| \hat{F}_m - F||_{\infty}\geq c)& \leq \frac{1}{c^{2l}}\E[|| \hat{F}_m - F||_{\infty}^{2l}]\\
    & \leq 2l \frac{1}{c^{2l}}\int_{0}^{+\infty} x^{2l-1}\mathbb{P}(|| \hat{F}_m - F||_{\infty}>x) \dd x\\
    & \leq 2l \frac{1}{c^{2l}}\Big(\int_{1/m}^{+\infty} x^{2l-1} 2e^{-2m(x-1/m)^2}\dd x + \int_{0}^{1/m}x^{2l-1} \dd x \Big)\\
    & \leq 2l \frac{1}{m^l c^{2l}}\Big(\int_{0}^{+\infty} (y+1)^{2l-1} 2e^{-2y^2}\dd y + \frac{1}{2l}\Big) = \frac{\tilde{c}_l}{c^{2l}} m^{-l}. 
\end{align*}

In turn, we have for any $c_F(t) >0$ and $x >0$, 
\begin{align*}
    \mathbb{P}(F(x)- \hat{F}_m(x) <  -c_F(t))\leq \mathbb{P}(|| \hat{F}_m - F||_{\infty}\geq c_F(t))\leq \tilde{c}_l \frac{1}{c_F(t)^{2l}}m^{-l}. 
\end{align*}
Hence 
\begin{equation*}
    \mathbb{P}(\Omega_{c_0,t}^c) \leq \mathbb{P}((\Omega_{c_0}^*)^c) + \mathbb{P}((\Omega_{t}^*)^c)\leq (c_l +\frac{\tilde{c}_l}{c_F(t)^{2l}}) m^{-l}.
\end{equation*}

\subsection{Additional figures}
\begin{figure}[H]
\begin{subfigure}{0.48\textwidth}
    \centering
    \includegraphics[scale=0.26]{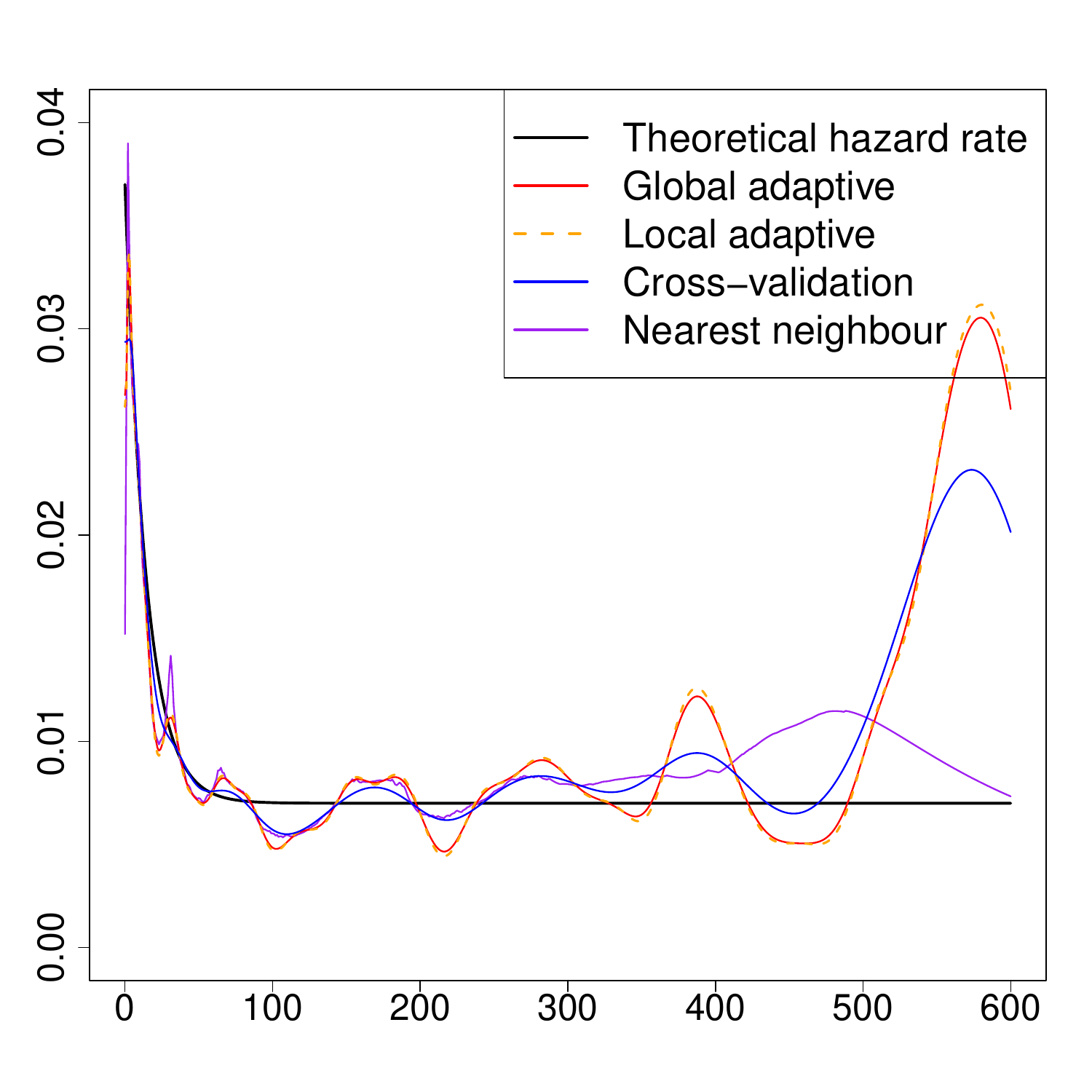}
    \caption{$m = 500$}
    \label{fig:bw_choice_comp500}
    \end{subfigure}
    \begin{subfigure}{0.48\textwidth}
    \centering
    \includegraphics[scale=0.26]{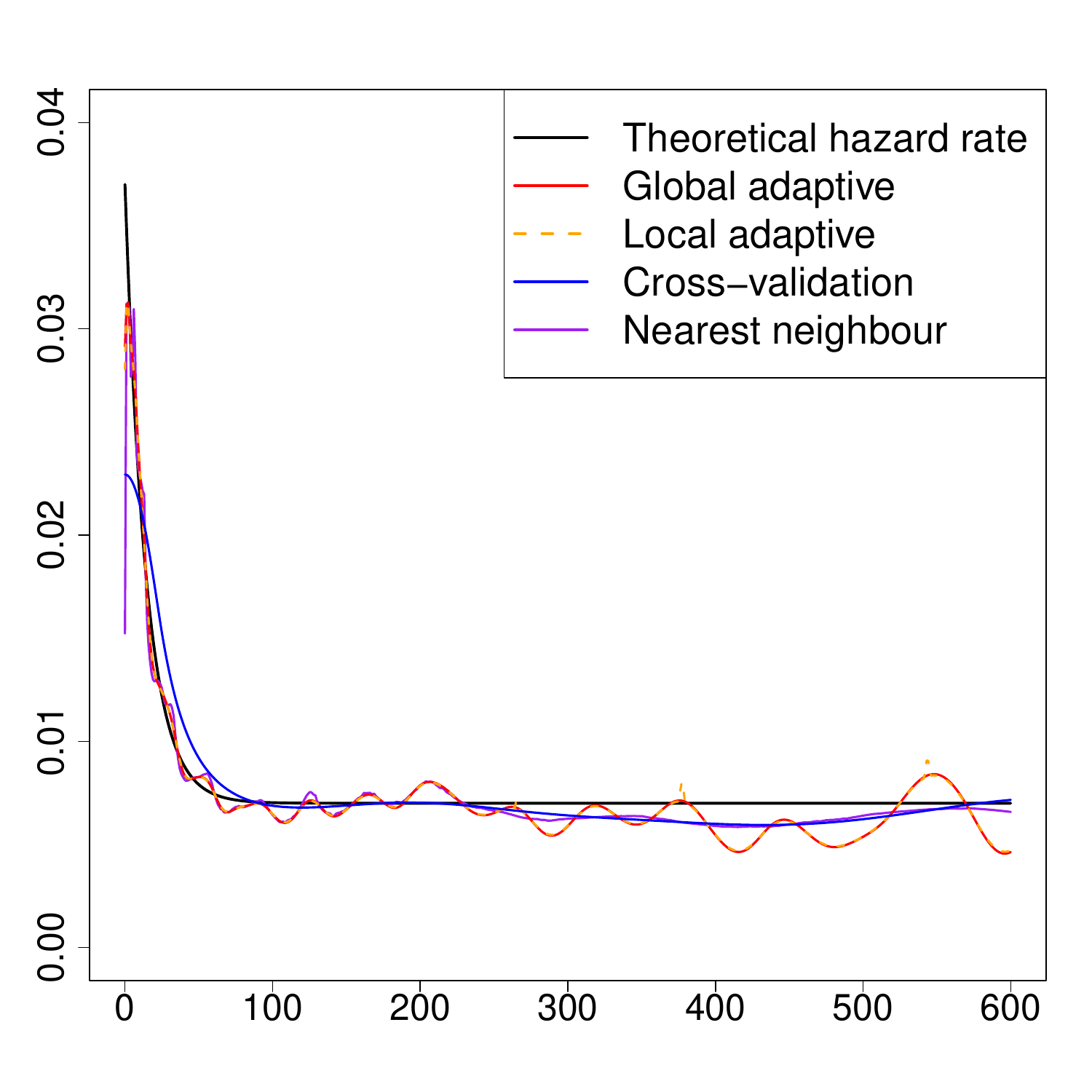}
    \caption{$m=2000$}
   \label{fig:bw_choice_comp2000}
    \end{subfigure}
    \caption{\textbf{Comparison of bandwidth choice methods} on a hazard rate $k(t) = a+c\cdot e^{-dt}$, $a = 7\cdot 10^{-3}, c = 3 \cdot 10^{-2}, d = 7 \cdot 10^{-2}$. }
    \label{fig:bw_choice_comp}
\end{figure}

\bibliographystyle{abbrv}
\bibliography{Biblio.bib}

\end{document}